\documentclass[11pt]{amsart}
\usepackage[nocompress]{cite}
    \usepackage[utf8]{inputenc}
    \usepackage[T1]{fontenc}
\usepackage{graphicx}
\usepackage{csquotes}
\usepackage{amsmath}
\usepackage{amsfonts}
\usepackage{amssymb}
\usepackage{hyperref}
\usepackage{dsfont}
\usepackage{oubraces}

\hypersetup{
    colorlinks=true,
    urlcolor=blue,
    linkcolor=blue,
    breaklinks=true,
    citecolor=darkgreen
}
\usepackage{url}
\usepackage{xcolor,color}
\colorlet{darkgreen}{green!50!black}
\usepackage{lmodern}
\usepackage{amsbsy}
\usepackage{amsmath}

\usepackage[utf8]{inputenc} 
\usepackage[T1]{fontenc}
\usepackage[french,english]{babel}

\usepackage{geometry}
\usepackage{caption}
\usepackage{subcaption}

\def\Re{{\textnormal{Re}}}
\def\R{{\mathbb R}}

\def\C{{\mathbb C}}
\def\Z{{\mathbb Z}}
\def\1{{1\!\!\!1}}

\def\E{{\mathbb E}}

\def\P{{\mathbb P}}

\def\argmax{{\textnormal{argmax}}}
\def\supp{{\rm{supp}}}

\newcommand{\fc}{\mathds{1}}

\let \phi=\varphi

\newcommand{\be}{\begin{equation}}
\newcommand{\ee}{\end{equation}}
\numberwithin{equation}{section}

\newtheorem{theorem}{Theorem}
\newtheorem{prop}{Proposition}[section]
\newtheorem{cor}[prop]{Corollary}
\newtheorem{defi}[prop]{Definition}
\newtheorem{lemma}[prop]{Lemma}
\newtheorem{rem}[prop]{Remark}

\newtheorem{notation}{Notations}

\newcommand{\nb}[1]{#1}
\newcommand{\nr}[1]{}
\newcommand{\ngray}[1]{}
\usepackage{oubraces}

\thanks{This project received funding from ENS Rennes, and from the ANR RESYST (ANR-22-CE40-0002)}
\begin{document}

\title[Martin boundary of a degenerate RBM in a wedge]{Martin boundary of a degenerate Reflected Brownian Motion in a wedge}

%
%\author{Sandro Franceschi}
%\address{Institut Polytechnique de Paris, T\'el\'ecom SudParis, Laboratoire SAMOVAR, 19 place Marguerite Perey, 91120 Palaiseau, France}
%\email{sandro.franceschi@telecom-sudparis.eu}
%%\urladdr{www.math.sc.edu/$\sim$howard} % Delete if not wanted.
%
%
%\author{Irina Kourkova}
%\address{Sorbonne Universite, Laboratoire de Probabilités, Statistiques et Modélisation, 
%UMR 8001, 4 place Jussieu, 75005 Paris, France. 
%}
%\email{irina.kourkova@sorbonne-universite.fr}

\author{Maxence Petit }
\address{Sorbonne Universite, Laboratoire de Probabilités, Statistiques et Modélisation, 
UMR 8001, 4 place Jussieu, 75005 Paris, France. }
\email{maxence.petit@ens-rennes.fr}

%\thanks{This project has received funding from Agence Nationale de la Recherche, ANR JCJC programme under the Grant Agreement ANR-22-CE40-0002.}

\begin{abstract}
We consider an outward degenerate drifted Brownian motion in the quarter plane with oblique reflections on the boundaries. In this article, we explicitly compute the Laplace transforms of the Green’s functions associated with the process. These Laplace transforms are expressed as an infinite sum of products \nr{using the compensation method}\nb{by iterating a functional equation, which is deeply linked to the compensation method}. We also derive the asymptotics of the Green’s functions along all possible paths and determine the (minimal) Martin boundary. Finally, we provide explicit formulae for all the corresponding positive harmonic functions.
\end{abstract}

\maketitle

\section{Introduction and main results}

\subsection*{Context}
The semi-martingale reflecting Brownian motion (SRBM) in two-dimensional convex cones is a classical topic in probability theory. Problems such as existence and uniqueness~\cite{HaRe-81b,taylor_existence_1993},
recurrence and transience conditions \cite{hobson_recurrence_1993, williams_recurrence_1985}, study of stationary distribution properties \cite{harrison_multidimensional_1987, dieker_reflected_2009, dai_reflecting_2011, franceschi_asymptotic_2016}, and many
others have been extensively studied in the literature, mostly under the assumption of a non-degenerate covariance matrix.  %\textcolor{red}{Peut être rajouter une phrase/référence pour de potentiels referes}.

An important problem in transient SRBM is the analysis of Green's functions, which can be divided into two
parts:
\begin{itemize}
\item[$(P_1)$]\label{Q1}  Obtaining the Laplace transforms of the Green's functions,
\item[$(P_2)$] Computing the asymptotics of the Green's functions along all trajectories of the SRBM.
\end{itemize}
Solutions to $(P_1)$ in the half-plane %with only two domains of spatial homogeneity of parameters (the interior and the boundary axis) 
can be expressed directly in terms of a rational function of two variables $(x, Y(x))$, where $Y(x)$ is a branch of a certain two-valued algebraic function, as detailed in~\cite{ernst_franceschi_asymptotic_2021}. However, solving $(P_1)$ \nr{in cones}\nb{in a general cone} %with more areas of parameters' spatial homogeneity 
presents a significantly greater challenge. Specifically, for non-degenerate SRBM in the quarter plane with three domains%($x$-axis, $y$-axis and the interior)
, the Laplace transforms are obtained as singular integral representations via boundary-value problems, as shown in~\cite{franceschi_green_2021, franceschi_explicit_2017}. Although these expressions are explicit, they are not particularly amenable to in-depth analysis. Fortunately, they are not required to resolve the second issue $(P_2)$. In fact, only the locations of the dominant singularities of unknown Laplace transforms are necessary to compute the asymptotics of the Green's functions. Problem $(P_2)$ for non-degenerate SRBM has been solved in the half-plane in~\cite{ernst_franceschi_asymptotic_2021} and in an arbitrary wedge in~\cite{franceschi2023asymptotics}. The approach followed in these articles has been developed in ~\cite{FIM17,kurkova_martin_1998, Kurkova2011, dai_reflecting_2011, franceschi_asymptotic_2016,malyshev_asymptotic_1973}, and can be considered as a version of the so-called kernel method. For more information, see the survey by Y.Q. Zhao~\cite{Zhao2022}. The kernel $\gamma(x,y)$ of the SRBM is given by one half of the quadratic form of the covariance matrix, plus the linear form of the drift inside the cone. The interplay between the branches of algebraic functions $X(y)$ and $Y(x)$, defined by the kernel equation $\gamma(x,y)=0$, allows us to analytically continue unknown Laplace transforms and to determine their singularities. The inverse Laplace transforms, combined with the saddle-point method, then yield asymptotic expansions for the Green’s functions. This procedure provides asymptotic developments of Green's functions with arbitrarily many terms, but with unknown multiplicative constants. These constants may be derived -- albeit somewhat indirectly -- from the solutions to $(P_1).$

The degenerate SRBM in two-dimensional cones, i.e. with a covariance matrix of rank one, has been studied far less extensively. In~\cite{ichiba2021, franceschi2024invariant} it arises as the gap process between three particles moving and colliding in $\R^1$. The construction of this three-particle process relies on the Skorokhod reflection approach, as developed in~\cite{HaRe-81b}, to define pathwise reflected Brownian motion.

In the present article we consider a class of degenerate transient SRBMs in the quadrant, defined by conditions \eqref{donnees}-\eqref{vectors} and solve both problems $(P_1)$ and $(P_2)$. The Laplace transforms of the Green’s functions are expressed in terms of infinite series in product form. This result follows from the compensation method, initially introduced in~\cite{Adan} to obtain the stationary measure for certain degenerate random walks in a quadrant. This approach has since been successfully applied to queueing systems~\cite{Adan1991,adan1994compensation}. It has also been used to derive generating functions for random walks with small steps~\cite{ADAN_2013}, and more recently, to determine the \nr{stationary distributions of recurrent random walks}\nb{harmonic functions of singular walks in the quadrant}~\cite{hoang2022}. In~\cite{franceschi2024invariant}, for instance, it was used to derive the explicit form of the stationary distribution, and in~\cite{franceschi2024martin}, to determine the Martin boundary of killed degenerate Brownian motion in a two-dimensional cone.

In this article, we compute the asymptotics of the Green’s functions along all trajectories. To achieve this, we adapt the approach described earlier and developed in~\cite{FIM17, kurkova_martin_1998, dai_reflecting_2011, franceschi_asymptotic_2016,ernst_franceschi_asymptotic_2021, franceschi2023asymptotics} to this class of degenerate SRBMs. A key difference is that, unlike the non-degenerate case -- in which the kernel equation for the process defines an ellipse -- the kernel equation for the degenerate case defines a parabola in $\R^2$. 
The multiplicative constants in the asymptotic expressions of the Green’s functions, derived from the solution to $(P_1)$, are made explicit in terms of infinite series in product form. The significance of these constants -- viewed as functions of the starting point of the process -- extends beyond asymptotic precision: they also yield all positive harmonic functions for the DRBM via the Martin boundary theory.

Initiated by Martin~\cite{Martin1941}, and further developed by Hunt~\cite{Hunt_1957, Hunt_1957_2, Hunt_1957_3}, Doob~\cite{Doob_1959}, Kunita and
Watanabe~\cite{kunitaWatanabe1965}, this theory is summarised in~\cite{Chung2005MarkovPB, Doob1984ClassicalPT}. Its aim is to describe the asymptotic behavior of the process and to characterize all non-negative superharmonic and harmonic functions. The limits of the Martin kernel along the trajectories of the process, when they exist, compactify the state space and form the so-called Martin boundary. This procedure allows every non-negative harmonic function to be expressed as an integral representation over the Martin boundary. % associated to a family of non-negative harmonic functions on the state space. Each non-negative harmonic function of the process is expressed as an integral of Martin harmonic functions over a subset of the Martin boundary, called the minimal Martin boundary. 
In~\cite{Ney1966, Ignatiouk_Robert_2009, Ignatiouk_Robert_2010,t_Ignatiouk_Robert_2010, duraj2020martin}, the Martin boundary is identified via large deviation principles. %However, since Martin kernels are ratios of Green's functions, once the asymptotics of Green's functions are known, those of Martin kernels follow immediately. 
It has also been obtained from the asymptotics of Green’s functions in~\cite{franceschi2023asymptotics, ernst_franceschi_asymptotic_2021, kurkova_martin_1998,Kurkova2011}. %Note that in order to identify the topological structure of the Martin compactification, we do not need explicit expressions of the multiplicative constants in asymptotics of Green's functions, but only some of their elementary properties. Nevertheless, we need these explicit expressions in order to compute Martin harmonic functions. 
In this article, using the solutions to problems $(P_1)$ and $(P_2)$ described above, we determine the Martin boundary, the minimal one and provide explicit expressions for all\nr{ Martin} positive harmonic functions.

 %\textcolor{gray}{the distribution tail touching the cone for the Brownian motion~\cite{Bañuelos1997}}    

\color{black}

\subsection*{Main results}
\subsubsection*{The degenerate reflected Brownian, assumptions.} We consider a degenerate Brownian motion $(Z_t)_{t\geq0}$ in a quadrant, with oblique reflection at the boundaries. By \textit{degenerate} we mean that the covariance matrix is of rank $1$. This obliquely reflected process was studied in \cite{ichiba2021} and its rigorous definition %(and existence, uniqueness) 
is provided in Section \ref{sec:2}. The parameters of the degenerate reflected Brownian motion are given by:
\begin{equation}\label{donnees}
\Sigma = \begin{pmatrix}
\sigma_1^2 & -\sigma_1\sigma_2\\
-\sigma_1\sigma_2 & \sigma_2^2
\end{pmatrix},
\mu = \begin{pmatrix}
\mu_1  \\
\mu_2
\end{pmatrix},
R = \begin{pmatrix}
1 & r_{2} \\
r_{1} & 1
\end{pmatrix} = \begin{pmatrix}
R_1 & R_{2} 
\end{pmatrix}
\end{equation}
where $\Sigma$ is the degenerate covariance matrix ($\det(\Sigma) = 0$), $\mu$ is the drift and columns of $R$ represent the reflection directions from the axes. The direction $v = (v_1, v_2)^T = ({\sigma_1}, -{\sigma_2})$ is antidiagonal, i.e. $v_1v_2 < 0$ (see Figure~\ref{rebonds}). When the process does not hit the boundaries, it behaves like a one-dimensional Brownian motion along the direction $v$ (plus the drift). 
%This process is a Strong Markov process.
Our main assumptions in this article are as follows:
\begin{equation}\label{drift}
\mu_1 > 0,\quad \mu_2 > 0,
\end{equation}
%\begin{equation}\label{antidiag}
%v_{1}v_{2} < 0
%\end{equation}
\begin{equation}\label{vectors}
r_{1} > -{\frac{\sigma_2}{\sigma_1}},\quad r_{2} > -{\frac{\sigma_1}{\sigma_2}}.
\end{equation}
Assumption~\eqref{drift} ensures that the process is transient, whereas~\eqref{vectors} specifies that the reflection vectors $R_1 = (1, r_{1})^T$ (on $\{x = 0\}$) and $R_2 = (1, r_{2})^T$ (on $\{y = 0\}$)  point outward from the direction $v$ of the Brownian motion (see Figure~\ref{rebonds}). %Note that the condition $v_1v_2 < 0$ is actually assumed as soon as the covariance matrix $\Sigma$ has the shape (\ref{donnees}).

\begin{figure}[hbtp]
\centering
\includegraphics[scale=1.3]{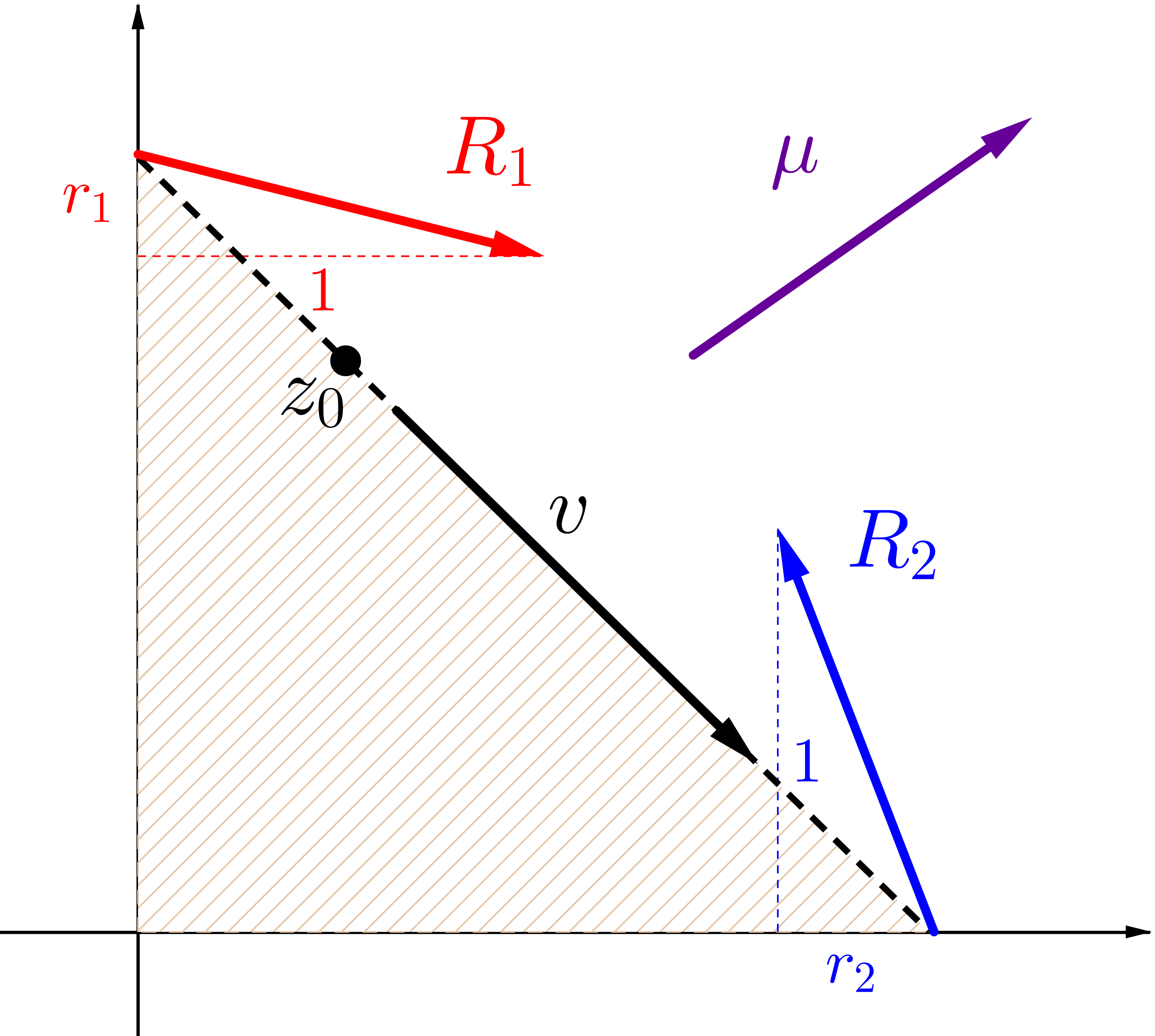}
%\vspace{-0.5cm}
\caption{Reflections $R_1, R_2$ on the edges, the drift $\mu$, and the direction $v$ of the degenerate Brownian motion. The process starting from $z_0$ never reaches the hatched region.}
\label{rebonds}
\end{figure}

\begin{figure}[hbtp]
\centering
\includegraphics[scale=0.7]{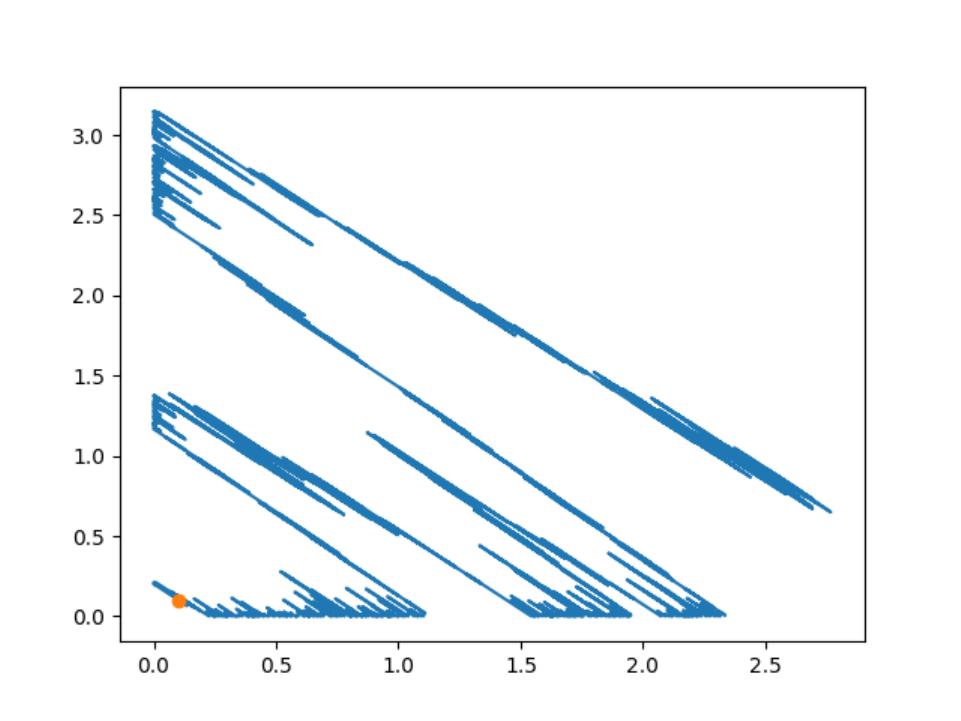}
\vspace{-0.5cm}
\caption{Example of a typical path (over a finite time horizon) of the drifted degenerate Brownian motion. The initial point is marked in orange.}
\label{exemple}
\end{figure}

In Sections 1--8  we state and prove our results under the additional assumption
\begin{equation}\label{simple}  \sigma_1 = \sigma_2 = 1, \quad \mu_ 1 + \mu_2 = 1.
\end{equation}
Results for the general case, i.e. without Assumption~\eqref{simple}, are stated and proved in Section~\ref{App:A}.  In fact, they are easily deduced from the results under~\eqref{simple} by means of a simple space-time transformation. 

\subsubsection*{Green's functions}
We show that for any starting point $z_0\in\R_+^2$, there exists a density $g^{z_0}(\cdot)$ of the Green's measure $G(z_0, \cdot)$ on the quadrant defined as
\begin{equation}\label{mes_green}
G(z_0, A) := \int_0^{+\infty}\P_{z_0}(Z_t \in A)dt = \nr{\int_{\R_+^2}g^{z_0}(z)dz = } \nb{\int_{A}g^{z_0}(z)dz}.
\end{equation}
Functions $g^{z_0}(\cdot)$ are called the Green's functions. We also define the Green's measures on the sides of the wedge
\begin{equation}\label{mes_bord}
H_i(z_0,A) := \E_{z_0}\left[\int_0^\infty\fc_A(Z_t)dL^i_t
\right], \quad i=1, 2
\end{equation}
where $(L^1_t)_{t\geq0}$ (resp. $(L^2_t)_{t\geq0}$) is the local time of the process on the axis $\{x = 0\}$ (resp. $\{y=0\}$).
The measure $H_1$ has its support on the vertical axis and $H_2$ has its support on the horizontal axis. Laplace transforms $\phi(x,y)$ of $G(z_0, \cdot)$ and $\phi_1(y), \phi_2(x)$ of $H_1(z_0,\cdot), H_2(z_0,\cdot)$ are related by the following functional equation  
\begin{equation}\label{Equation_fonctionelle_intro}
    - \gamma(x, y)\varphi(x, y) = \gamma_1(x, y)\varphi_1(y) + \gamma_2(x, y)\varphi_2(x) + e^{(x, y)\cdot z_0}, \  \  \  {Re}(x)<0, {Re}(y)<0
\end{equation}
where
\begin{equation}\label{gamma}
\gamma(x,y) = \frac{1}{2}(x-y)^2 +\mu_1x + \mu_2 y
\end{equation}
and
\begin{equation}\label{gamma12}\gamma_1(x, y) = R^1\cdot (x, y) = x + r_{1}y, \quad \quad \gamma_2(x, y) = R^2\cdot (x, y) = r_{2}x + y.
\end{equation}
It can be viewed as a balance equation for Green's measures between the interior and the edges of the quadrant.
Let us define
\begin{equation}\label{paraboleP}
 \mathcal{P} = \{(x, y) \in \R^2,\ \ \gamma(x, y) = 0\}.
\end{equation}
The functional equation \eqref{Equation fonctionnelle} is similar to that in \cite{franceschi2023asymptotics}; however, an important difference is that $\mathcal{P}$ is now a parabola rather than an ellipse. This distinction is what allows the compensation method to be effective, leading to explicit expressions for the Laplace transforms and positive harmonic functions.

\subsubsection*{Explicit expressions for Laplace transforms}
The first results of the article provide explicit expressions for Laplace transforms $\phi_1$ and $\phi_2$ in terms of infinite series of product forms, given by formulae \eqref{phi1decroit} and \eqref{phi2decroit} which we do not specify here.  Function $\phi$ is derived from $\phi_1$ and $\phi_2$ via the functional equation~\eqref{Equation_fonctionelle_intro}.

\subsubsection*{Asymptotics of Green's functions}
We now focus on the asymptotics of $g^{z_0}(r\cos(\alpha), r\sin(\alpha))$ as $r\to+\infty$ and $\alpha \to \alpha_0 \in [0, \pi/2]$. For any direction $\alpha$, we denote by $(x(\alpha), y(\alpha))$ a corresponding point on the parabola given by \begin{equation}
(x(\alpha), y(\alpha)) = \argmax_{(x, y) \in \mathcal{P}} \left(\cos(\alpha)x + \sin(\alpha)y\right),
\end{equation}
see Figure \ref{colcol}. It can be computed explicitly as: 
\begin{equation}
(x(\alpha), y(\alpha)) = \left(\frac{(\mu_2 - \tan(\alpha)\mu_1)(\mu_2 + \tan(\alpha))(1+\mu_2)}{2(1+\tan(\alpha))^2}, \frac{(\mu_2 - \tan(\alpha)\mu_1)(1 + \mu_1(1+ \tan(\alpha)))}{2(1+\tan(\alpha))^2}\right).
\end{equation}
Let us define two particular directions
\begin{equation}\label{alpha*}
\alpha^* := \left\{
    \begin{array}{ll}
0  \quad &\textnormal{if}\quad (r_{1} + 1)\mu_2 \leq 2 \\
      \arctan\left(\frac{(1+r_{1})\mu_2 - 2}{2 + (1+r_{1})\mu_1} \right) &\textnormal{if}\quad (r_{1} + 1)\mu_2 > 2.
    \end{array}
\right.
\end{equation}
\begin{equation}\label{alpha**}
 \alpha^{**}:= \left\{
    \begin{array}{ll}
\arctan\left(\frac{(1+r_{2})\mu_2 + 2}{(1+r_{2})\mu_1 - 2}\right) \quad &\textnormal{if}\quad (r_{2} + 1)\mu_1 > 2 \\
      \pi/2 &\textnormal{if}\quad (r_{2} + 1)\mu_1 \leq 2,
    \end{array}
\right.
\end{equation}
see Figure \ref{alphaalpha} for their geometric interpretation. 
We always have  $\alpha^* < \alpha^{**}$ as will be proved in Section \ref{sec:6}.

In the following theorem we summarize  the asymptotics of Green's functions for directions $\alpha_0 \in (0, \pi/2)\backslash\{\alpha^*, \alpha^{**}\}$.  The ones for $\alpha_0 \in \{0,\alpha^*, \alpha^{**}, \pi/2\}$ are given later in Theorems~\ref{alpha=0}, \ref{directions_poles} and \ref{thm:7}.

\begin{figure}[htbp]
    \centering\hspace{-1cm}
    \begin{subfigure}{0.5\textwidth}
\centering
        \includegraphics[scale = 3]{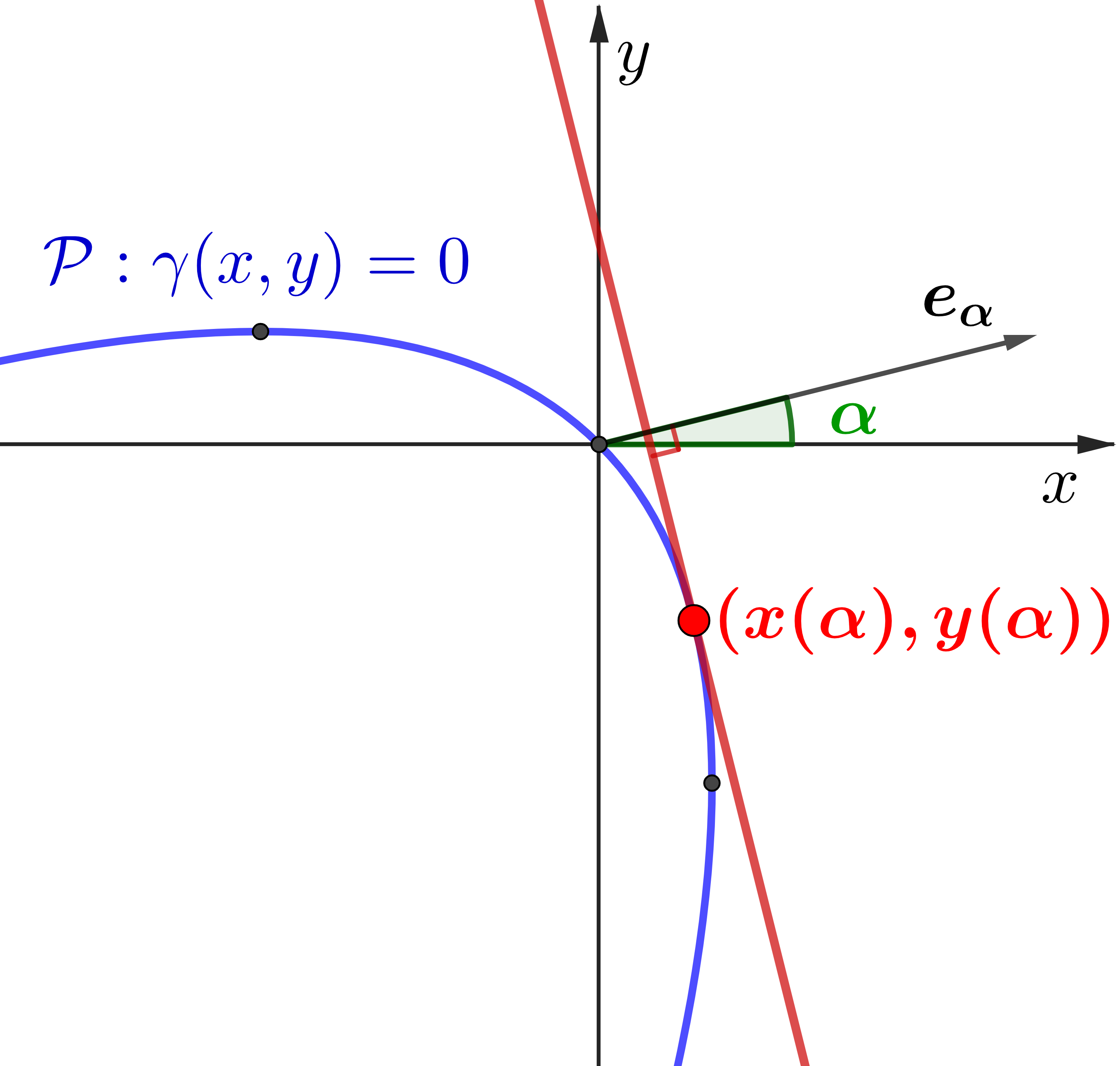}
        \caption{The point $(x(\alpha), y(\alpha))$ maximises the scalar product $\langle (x,y), e_\alpha\rangle$ where $e_\alpha = (\cos(\alpha), \sin(\alpha))$ and $(x,y)$ belongs to the parabola $\mathcal{P}$.}
        \label{colcol}
    \end{subfigure}
    \hfill
    \begin{subfigure}{0.45\textwidth}
\centering\hspace{-1cm}
        \includegraphics[scale = 2.7]{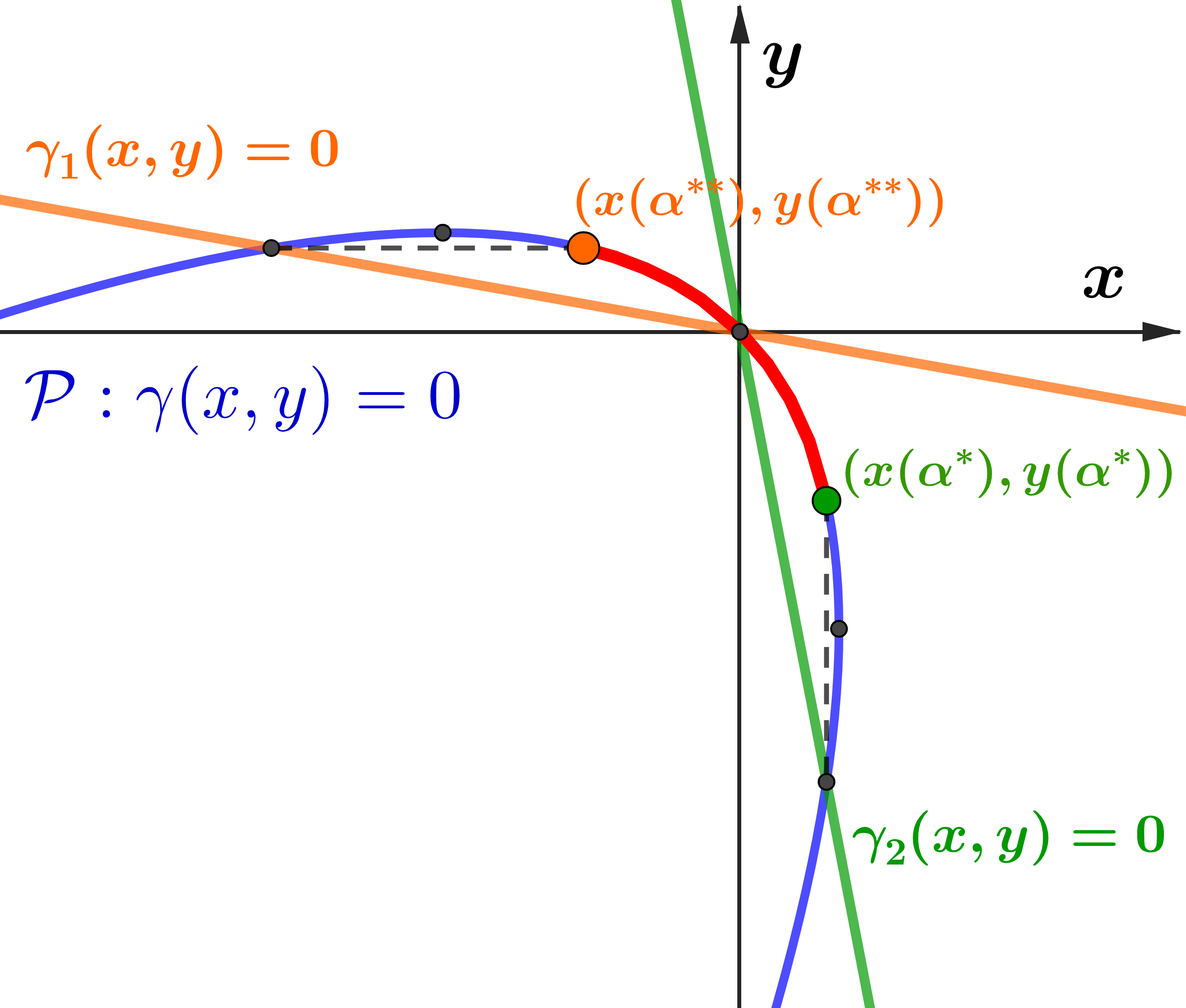}
        \caption{In the case $0 < \alpha^* < \alpha^{**} < \pi/2$, angles $\alpha^*$ and $\alpha^{**}$ introduced in \eqref{alpha*}, \eqref{alpha**} can be defined equivalently using this construction.}
        \label{alphaalpha}
    \end{subfigure}
%    \caption{Main caption for the two figures}
    \label{mmm}
    \caption{Geometric interpretation of $(x(\alpha), y(\alpha))$, $\alpha^*$ and $\alpha^{**}$.}
\end{figure}

\begin{theorem}[Asymptotics in the quadrant, general case]
Assume (\ref{drift}) to (\ref{simple}). Then, the Green's density function $g^{z_0}$ of this process has the following asymptotics as $\alpha \to \alpha_0$ and $r\to\infty$.
\begin{itemize}
\item If $\alpha^* < \alpha_0 < \alpha^{**}$, then
\begin{equation}\label{cas_1}
g^{z_0}(r\cos(\alpha), r\sin(\alpha))
\underset{r\to\infty \atop \alpha \to \alpha_0}{\sim} 
 {c_{\alpha_0}}h_{\alpha_0}(z_0)\frac{e^{-r(\cos(\alpha)x(\alpha) + \sin(\alpha)y(\alpha))}}{\sqrt{r}}.
  \end{equation}

  \item If $\alpha_0 < \alpha^*$, then
\begin{equation}\label{cas_2}
g^{z_0}(r\cos(\alpha), r\sin(\alpha))
\underset{r\to\infty \atop \alpha \to \alpha_0}{\sim} 
c^*h_{\alpha^*}(z_0)e^{-r(\cos(\alpha)x(\alpha^*) + \sin(\alpha)y(\alpha^*))}.
  \end{equation}

  \item If $\alpha_0 > \alpha^{**}$, then
\begin{equation}\label{cas_3}
g^{z_0}(r\cos(\alpha), r\sin(\alpha))
\underset{r\to\infty \atop \alpha \to \alpha_0}{\sim} 
c^{**}h_{\alpha^{**}}(z_0) e^{-r(\cos(\alpha)x(\alpha^{**}) + \sin(\alpha)y(\alpha^{**}))}.
  \end{equation}
\end{itemize}
where ${c_{\alpha_0} = \frac{1}{\sqrt{2\pi(\cos(\alpha_0) + \sin(\alpha_0))}}}$, $c^*$ and $c^{**}$ are positive explicit constants only depending on the parameters of the degenerate reflected Brownian motion (see \eqref{eq:c*}) and where $h_{\alpha}(z_0), h_{\alpha^*}(z_0), h_{\alpha^{**}}(z_0)$ are harmonic functions given in Theorem \ref{thm:main}. Furthermore, $h_{\alpha}(z_0), h_{\alpha^*}(z_0), h_{\alpha^{**}}(z_0)$ are non-zero.
\label{thm1}
\end{theorem}

\subsubsection*{Explicit expressions for positive harmonic functions with the compensation method}
 Let us recall the following definition: a function $h : \R_+^2 \longrightarrow \R$ is harmonic if and only if for all $t \geq 0$ and $z_0 \in \R_+^2$, 
\begin{equation}\label{def:harm}
\E_{z_0}[h(Z_t)] = h(z_0).
\end{equation}
All functions $h_\alpha,  \alpha \in [\alpha^*, \alpha ^{**}]$  are harmonic\nr{and are called Martin harmoinc functions}. 
These functions are explicitly stated in Theorem~\ref{thm:main} below and will be derived in this article using the compensation method.
The essence of this method is to construct functions that satisfy the partial differential equation along with boundary conditions:
\begin{equation}\label{BVP}
\left\{
    \begin{array}{ll}
        {(H_0)} \quad\mathcal{G}h = 0 \quad\quad\quad \quad\quad \textnormal{on} \quad (0, +\infty)^2,   \\
	  {(H_1)} \quad\partial_{R_1}h(0, y) = 0, \quad y \geq 0\\
	  {(H_2)} \quad\partial_{R_2}h(x, 0) = 0,  \quad x \geq 0
    \end{array}
\right.
\end{equation}
where  $\mathcal{G} =\frac{1}{2} \nabla\cdot\Sigma\nabla + \mu\cdot\nabla$.
Those function are harmonic as it will be noticed in Section~\ref{subsec:5.1}.
%Using the compensation method (explained further in the article), the explicit form of $h_\alpha(z_0)$ is obtained as an explicit infinite sum. %(and the corresponding formulae for directions $\alpha = 0, \alpha^*, \alpha^{**}$ and $\alpha = \pi/2$ not yet introduced in Theorem \ref{thm1}). 

For $(a_0,b_0) \in \mathcal{P}$ and $k \in \Z\backslash\{0\},$ we set \begin{equation}\label{eq:an}
a_{2k} = -2k^2 +  2(a_0 - b_0 - \mu_2)k + a_0, \ \ a_{2k+1} = a_{2k}
\end{equation}
\begin{equation}b_{2k} = -2k^2 + 2(a_0 - b_0 + \mu_1) k + b_0, \ \ b_{2k+1} = b_{2k+2}\label{eq:bn}
\end{equation}
\begin{figure}[hbtp]
\centering
\includegraphics[scale=0.7]{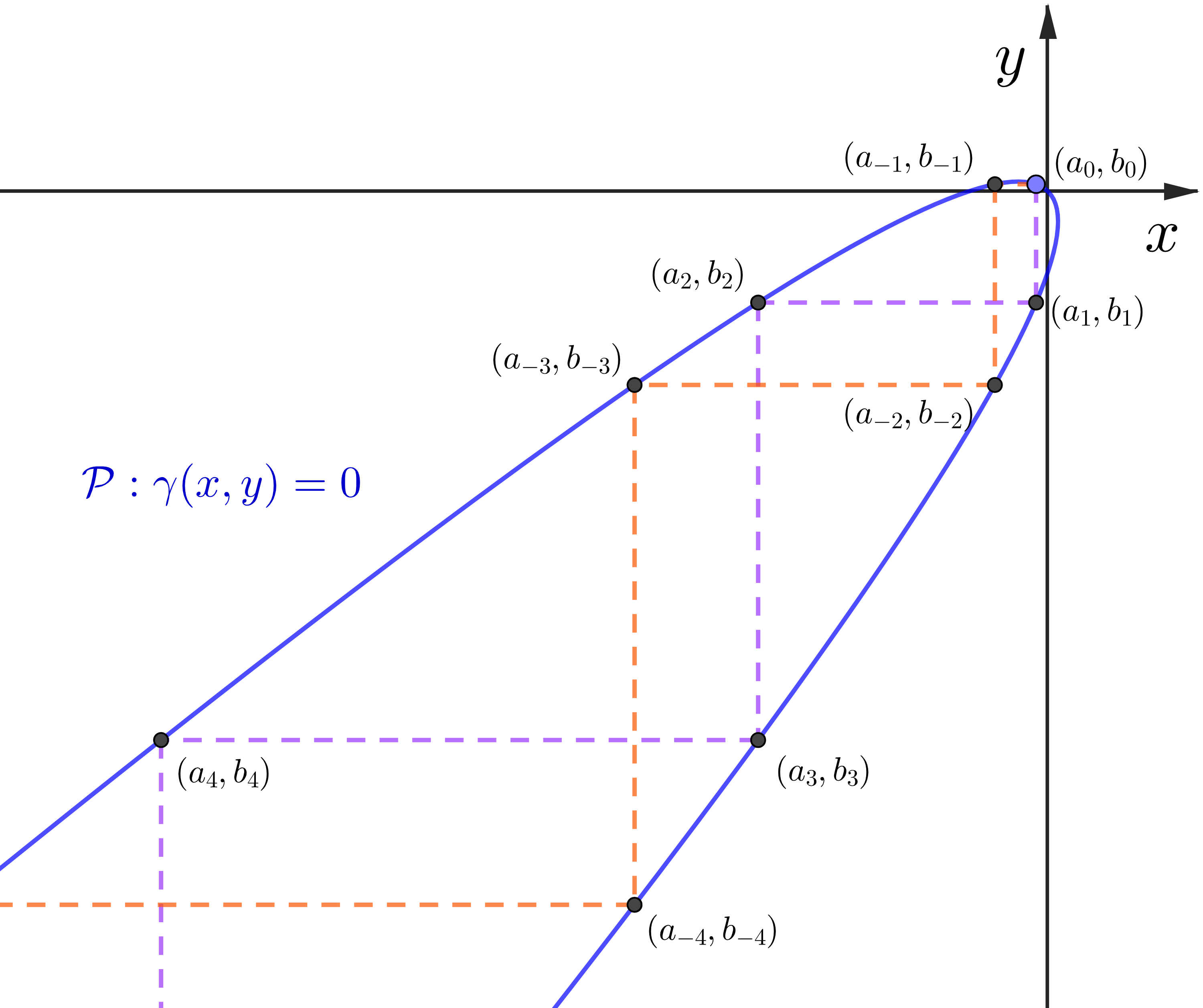}
\vspace{0cm}
\caption{Parabola $\mathcal{P}$ and points $(a_n, b_n)$ on the parabola.}
\label{autoo1}
\end{figure}
As illustrated in Figure~\ref{autoo1}, points $(a_p, b_p) \in \mathcal{P}$ are constructed by following the "downstairs" path on the parabola, 
applying successively automorphisms that leave invariant the first or the second coordinate respectively.
%Each function $h_\alpha$ will be associated with the saddle-point $(x(\alpha), y(\alpha))$ on the parabola (details in Section~\ref{sec:6}). 

\begin{theorem}[Explicit expressions for\nr{Martin} harmonic functions \nb{$(h_\alpha)_{\alpha\in[\alpha^*, \alpha^{**}]}$}]\label{thm:main}
Assume (\ref{drift}) to (\ref{simple}). Then, \nr{the family of Martin harmonic functions is given by}\nb{the functions} $(h_\alpha)_{\alpha \in [\alpha^*, \alpha^{**}]}$ \nb{are harmonic and are given by}\nr{with} the following formulae
\begin{itemize}
\item For $\alpha \in (\alpha^{*}, \alpha^{**})$, taking $(a_0, b_0) = (x(\alpha), y(\alpha))$, we have
\begin{equation}\label{harm1}
h_{\alpha} : z_0 \longmapsto \sum_{m=-\infty}^{+\infty} \kappa_m(\alpha)e^{z_0\cdot (a_{m}, b_{m})}
\end{equation}
%\begin{align}\label{explicite}\phi_2(s) =  \sum_{p=1}^{+\infty} \kappa_p(s)e^{z_0\cdot (a_{p}(s), b_{p}(s))}, \quad \phi_1(s) =  \sum_{p=-\infty}^{1} \kappa_p(s)e^{z_0\cdot (a_{p}(s), b_{p}(s))}
%\end{align}
where $\kappa_0(\alpha) = 1$ and
\begin{equation}\label{eq:kappam}
\kappa_m(\alpha)=\left\{
    \begin{array}{ll}
        & (-1)^m\left[\prod_{k=0}^{\lfloor \frac{m}{2} \rfloor-1} \frac{\frac{\gamma_1}{\gamma_2}(a_{2k+1}, b_{2k+1})}{\frac{\gamma_1}{\gamma_2}(a_{2k+2}, b_{2k+2})}\right]\frac{\gamma_2(a_0, b_0)}{\gamma_2(a_{m}, b_{m})} \quad\textnormal{if}\; m > 0\\
	& (-1)^m\left[\prod_{k=0}^{\lfloor \frac{-m}{2} \rfloor-1} \frac{\frac{\gamma_2}{\gamma_1}(a_{-2k-1}, b_{-2k-1})}{\frac{\gamma_2}{\gamma_1}(a_{-2k-2}, b_{-2k-2})}\right]  \frac{\gamma_1(a_0, b_0)}{\gamma_1(a_{m}, b_{m})} \quad\textnormal{if}\; m < 0
    \end{array}
\right.
\end{equation}
(with the convention $\Pi_{k=0}^{-1} = 1$).

%\begin{align}\label{explicite}\phi_2(s) &=  \frac{-1}{\gamma_2(a_1(s), b_1(s))}e^{z_0\cdot (a_1(s), b_1(s))} + \sum_{n=1}^{+\infty} \left[\prod_{k=0}^{n-1} \frac{\frac{\gamma_1}{\gamma_2}(a_{2k+1}(s), b_{2k+1}(s))}{\frac{\gamma_1}{\gamma_2}(a_{2k+2}(s), b_{2k+2}(s))}\right]\left[  \frac{e^{z_0\cdot (a_{2n}(s), b_{2n}(s))}}{\gamma_2(a_{2n}(s), b_{2n}(s))} - \frac{e^{z_0\cdot z(a_{2n+1}(s), b_{2n+1}(s))}}{\gamma_2(a_{2n+1}(s), b_{2n+1}(s))}\right] 
%\end{align}
%and
%\begin{align}\label{explicite2}\phi_1(s) &=  \frac{-1}{\gamma_1(a_{-1}(s), b_{-1}(s))}e^{z_0\cdot (a_{-1}(s), b_{-1}(s))} + \sum_{n=1}^{+\infty} \left[\prod_{k=0}^{n-1} \frac{\frac{\gamma_2}{\gamma_1}(a_{-2k-1}(s), b_{-2k-1}(s))}{\frac{\gamma_2}{\gamma_1}(a_{-2k-2}(s), b_{-2k-2}(s))}.\right]\left[  \frac{e^{z_0\cdot (a_{-2n}(s), b_{-2n}(s))}}{\gamma_1(a_{-2n}(s), b_{-2n}(s))} - \frac{e^{z_0\cdot (a_{-2n-1}(s), b_{-2n-1}(s))}}{\gamma_1(a_{-2n-1}(s), b_{-2n1}(s))}\right] 
%\end{align}

\item For $\alpha = \alpha^*$,
\begin{itemize}
    \item If $\frac{2}{r_{2} + 1} < \mu_2$, then $\alpha^*= 0$ and $h_{0} : z_0 \longmapsto \partial_\alpha\left[h_\alpha(z_0)\right]_{\alpha=0^+}$.
    \item If $\frac{2}{r_{2} + 1} > \mu_2$, then $\alpha^* > 0$ and taking $(a_0, b_0) = (x(\alpha^*), y(\alpha^*))$,
\begin{equation}\label{harm2}
h_{\alpha^*} : z_0\longmapsto e^{z_0\cdot (a_{1}, b_{1})} + \sum_{m=2}^{+\infty}\widehat{\kappa}_m(\alpha^*)e^{z_0\cdot (a_{m}, b_{m})} 
\end{equation}
where 
$$
\widehat{\kappa}_m(\alpha^*)= (-1)^{m+1}\frac{\gamma_1(a_{1}, b_{1})}{\frac{\gamma_1}{\gamma_2}(a_{2}, b_{2})}\left[\prod_{k=1}^{\lfloor \frac{m}{2} \rfloor-1} \frac{\frac{\gamma_1}{\gamma_2}(a_{2k+1}, b_{2k+1})}{\frac{\gamma_1}{\gamma_2}(a_{2k+2}, b_{2k+2}}\right]\frac{1}{\gamma_2(a_{m}, b_{m})}.
$$

\item If $\frac{2}{r_{2} + 1} = \mu_2$, then $\alpha^*= 0$ and taking $(a_0, b_0) = (x(0), y(0))$,
\begin{equation}\label{harm3}
h_{0} : z_0 \longmapsto 2e^{z_0\cdot (a_0, b_0)} + \sum_{m=-\infty}^{-1} \kappa_m(\alpha^*)e^{z_0\cdot (a_{m}, b_{m})} + \sum_{m=2}^{+\infty} \tilde\kappa_m(\alpha^*)e^{z_0\cdot (a_{m}, b_{m})}
\end{equation}
where 
$$
\tilde\kappa_m(\alpha^*)= (-1)^{m+1} \frac{\gamma_1(a_{1}, b_{1})}{\frac{\gamma_1}{\gamma_2}(a_{2}, b_{2})}\left[\prod_{k=1}^{\lfloor \frac{m}{2} \rfloor-1} \frac{\frac{\gamma_1}{\gamma_2}(a_{2k+1}, b_{2k+1})}{\frac{\gamma_1}{\gamma_2}(a_{2k+2}, b_{2k+2})}\right]\frac{1}{\gamma_2(a_{m}, b_{m})}.
$$
\end{itemize}
\item For $\alpha = \alpha^{**}$, symmetrical formulae hold replacing $r_1$ by $r_2$, $\mu_1$ by $\mu_2$ and $0$ by $\frac{\pi}{2}$.

\end{itemize}

\end{theorem}

Note that if $\alpha < \alpha^*$ or $\alpha > \alpha^{**}$, expression~\eqref{harm1} may define a harmonic function that is not necessarily non-negative everywhere.

The Martin boundary \nb{and its minimality} \nr{is}\nb{are} derived from \nb{Theorem~\ref{thm1} and} Theorem~\ref{thm:main}, \nb{together with the further technical results in Theorems~\ref{alpha=0}, \ref{directions_poles} and \ref{thm:7} concerning the asymptotics of Green functions along the directions $0,\alpha^*, \alpha^{**}$ and $\pi/2$}.\nr{, along with its minimality.}

\begin{theorem}[Martin Boundary]\label{thm:Martin}
Under (\ref{drift}) to (\ref{simple}), the Martin boundary $\Gamma$ of the degenerate reflected Brownian motion is \nr{given by}\nb{homeomorphic to} $[\alpha^*, \alpha^{**}]$ \nr{through}\nb{via the mapping} \nr{the homeomorphism} 
\begin{equation}\alpha \in [\alpha^*, \alpha^{**}] \longmapsto h_\alpha(\cdot)/h_\alpha(0) \in \Gamma.
\end{equation}
Furthermore, the Martin boundary is minimal.
\end{theorem}
\begin{figure}[hbtp]
\centering
\includegraphics[scale=0.5]{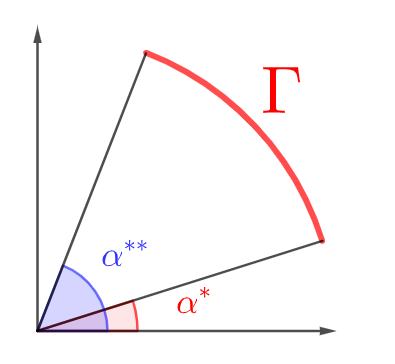}
\vspace{0cm}
\caption{Martin boundary $\Gamma$ when $0 < \alpha^*$ and $\alpha^{**} < \pi/2$.}
\label{Fig:Martin}
\end{figure}

\nb{\begin{rem}[On Assumptions \eqref{drift} and \eqref{vectors}, and possible extensions]$\quad$\label{rem:generalisation}
\begin{itemize}
\item[-] Regarding assumption \eqref{drift}, similar results could be established under the more general condition $\mu_1\sigma_2 + \mu_2\sigma_1 > 0$. This condition is equivalent to the orientation of the parabola toward $x\to -\infty$ and $y\to -\infty$. It is also necessary for the convergence of the expressions defining $h_\alpha$ — specifically, equation \eqref{harm1}.
Namely, if $\mu_2 < 0$, the Laplace transform $\phi_2$ would have a pole at zero. Due to the technical nature of this paper, we have chosen to restrict our analysis to Assumption \eqref{drift}. Investigating how the Martin boundary is affected by the presence of such a pole could be an interesting direction for future work. %For instance, for certain parameter regimes, we conjecture that the Martin boundary could reduce to a single point — something that never occurs under \eqref{drift}.
%The direction of the parabola $\mathcal{P}$ is determined by the drift $\mu$. If \eqref{drift} fails, one branch of the parabola escapes to infinity in the $x > 0$ or $y > 0$ direction. As a result, the series defining the functions $h_\alpha$ -- namely, equation \eqref{harm1} -- may diverge for certain starting points $z_0$.
\item[-]
If \eqref{vectors} is not satisfied, the arguments which yield the explicit expressions of the harmonic functions fail. In particular, attempts to construct the functions $h_\alpha$ without this assumption often lead to signed functions which, while possibly harmonic, are not necessarily non-negative. For interested readers, the only step in our argument that fails for general reflection vectors is equation \eqref{le_reste}, which may offer a direction for future investigation.
\end{itemize}
\end{rem}
}

%Note that, with the general Martin boundary theory, every nonnegative harmonic function can be written uniquely as an integral on the Martin boundary $[\alpha^*, \alpha^{**}]$ of the Martin harmonic functions found in Theorem \ref{thm:main}.

\subsection*{Plan of the article}
In Section~\ref{sec:2}, we define the degenerate reflected Brownian motion. We then derive the functional equation~\eqref{Equation_fonctionelle_intro} in Section~\ref{sec:3} and meromorphically extend Laplace transforms on the edges up to their singularities. %The study of poles of those extended those functions is then done in Section \ref{sec:4}. 
In Section~\ref{sec:5}, we obtain the explicit form of the Laplace transforms \nr{using the compensation method$\;$}\nb{iterating the functional equation \eqref{Equation_fonctionelle_intro}}\ngray{$\;$Even if the final expressions coincide with those of the compensation method, it is more precise to mention the iteration of the functional equation}. Next, in Section~\ref{sec:6} we carry out preparatory work to derive the asymptotics of Green’s functions.  These asymptotics are computed in all directions in Sections~\ref{sec:7} and~\ref{sec:8} by the saddle-point method. \nr{We then}\nb{This enables us to} prove Theorem~\ref{thm1} and Theorem~\ref{thm:main} by employing the explicit expressions from Section~\ref{sec:5}. In Section~\ref{sec:9} we establish the asymptotics of the Martin kernel and identify all \nr{Martin}\nb{the} harmonic functions. We also prove the minimality of the Martin boundary and conclude the proof of Theorem~\ref{thm:Martin}. Finally, in Section~\ref{App:A} we treat the general case of the model without Assumption~\eqref{simple} via a linear transformation of space and time.
%\tableofcontents

%\section{Introduction}

%\subsection*{Context}
\section{Definition of the process}\label{sec:2}

Throughout the following, the filtered space we consider is always the space of continuous functions $\mathcal{C}(\R_+,\R_+^2)$ with the standard $\sigma-$field and the usual filtration. \nb{The following background definition is taken from \cite{taylor_existence_1993}, where the non-degenerate reflected Brownian motion is studied.}

\begin{defi}[Degenerate reflecting Brownian motion]\label{defi}
Let $\Sigma, R$ and $\mu$ be defined as in \eqref{donnees}. A degenerate reflecting Brownian motion (DRBM) associated with the data $(\Sigma, \mu, R)$ is a process $(Z_t)_{t\geq 0}$ and a family of measures $(\P_{z_0})_{z_0\in\R_+^2}$ such that $(Z_t)_{t\geq 0}$ can be written as
\begin{equation}\label{semimart}
Z_t = X_t + RL_t \in \R_+^2, \ \ \  t \geq 0,
\end{equation}
where
\begin{itemize}
\item $(X_t - \mu t)_{t\geq0}$ is an adapted degenerate Brownian motion (with zero drift) of covariance $\Sigma$ starting from $z_0$ under $\P_{z_0}$.
\item L is an adapted $2$-dimensional process starting from $0$ such that $\P_{z_0}-$almost surely, its components $L^1, L^2$ are continuous and non-decreasing with $\supp(dL^i) \subset \{t \geq 0, Z^i_t = 0 \}$; that is, $L^i$ increases only when $Z^i_t = 0$.

\end{itemize}

\end{defi}

Note that under $\P_{z_0}$, $Z$ can be written as $(Z_t)_{t\geq0} = (z_0 + vB_t + \mu t + RL_t)_{t\geq0}$ where $(B_t)_{t\geq0}$ is a one-dimensional Brownian motion and $v = ({\sigma_1}, -{\sigma_2})$ ($= (1, -1)$ under \eqref{simple}) is the unique eigenvector (up to a scalar multiplication) associated with the positive eigenvalue of the covariance matrix.

\begin{theorem}[Existence, uniqueness and Strong Markov property]\label{process}
Suppose that $|r_1r_2| < 1$. Then, for any starting point $z_0$, there exists a DRBM  associated with $(\Sigma, \mu, R)$. The processes $Z$ and $(Z, L)$ are pathwise unique (according to the associated degenerate Brownian motion). Furthermore, $Z$ is a semi-martingale, a Feller process (i.e., for any $t\geq 0,$ $x \longmapsto \E_x[f(Z_t)]$ is continuous whenever $f$ is bounded and continuous), and a Strong Markov process. %\textbf{a Markov process ... pourquoi? Peut être solution Theorems 12,2.4, 12.2.3 of dans Multidimensional diffusion processes / Daniel W. Stroock, S. R. Srinivasa Varadhan, à voir}
\end{theorem}

%\textcolor{gray}{"The Feller property together with the continuity of the sample pathsimplies that X is a strong Markov process." ?? Dans THE HEAT EQUATION AND REFLECTED BROWNIAN MOTION IN TIME-DEPENDENT DOMAINS BY KRZYSZTOF BURDZY,1 ZHEN-QING CHEN2 AND JOHN SYLVESTER3}

\begin{proof}
Define the matrix $Q = I-R$, whose spectral radius is $\rho(Q) = \sqrt{|r_1 r_2|} < 1$. \nb{By Theorem~$1$ in \cite{HaRe-81b}, for any continuous path $x = (x_t)_{t\geq0} \subset \R^2$ there exists a unique solution $(z_t)_{t\geq0} = \psi(x)$ of the Skorokod problem $$z_t = x_t + R(l^1_t, l^2_t)^T, \quad t\geq0$$ where $(z_t)_{t\geq 0} \subset \R_+^2$ and for $i\in\{1,2\}$, $(l^i)_{t\geq0}$ is a continuous, increasing function with $\supp(dl^i) \subset \{t \geq 0, z^i_t = 0 \}$. Moreover, $\psi$ is continuous in the topology of uniform convergence on compact sets. This yields the stated result with $Z = \psi(X)$.}
\nr{Therefore, Theorem~$1$ in \cite{HaRe-81b} implies all the stated results.}
\end{proof}
\nb{As in the non-degenerate case, there may be existence and uniqueness in law if $R$ is a general $\mathcal{S}$-matrix \cite{taylor_existence_1993} (without assuming $|r_1r_2|<1$), but not pathwise uniqueness \cite{bass2024}. To avoid excessive technicality, we work under assumption $|r_1r_2| < 1.$}

\begin{prop}[Transience]\label{prop:transient}
Under conditions (\ref{drift}) and (\ref{vectors}), the DRBM is a transient Markov process.
\end{prop}

\begin{proof}
Consider $w = ({\sigma_2}, {\sigma_1})$, %\textcolor{red}{Réponse à Sandro : "supposer (1.5) ?" J'avais plutôt l'idée de supposer ça à partir de la partie 3. En effet, la partie 2 permet de donner des propriété sur le processus, peu importe les paramètres}
which is orthogonal to the direction of the Brownian motion. It suffices to note that $(Z_t\cdot w)_{t\geq0}$ is almost surely strictly increasing and tends to $+\infty$ since $Z_t\cdot w \geq \mu\cdot wt$ by \eqref{vectors}.
\end{proof}
%The following proposition is also admitted.
%\begin{prop}
%Under Assumptions \ref{Ass1} and \ref{Ass2}, the degenerate reflected Brownian motion is a strong Markov process.
%\end{prop}
%\begin{rem}
%Il se peut que les droites ne coupent pas la parabole ($r12$ ou $r21 = -1$).  %On a pas l'air de pouvoir compenser dans ce cas ($r12$ et $r21\neq -1?$). Dans l'article de Ichiba + Karatzas, $r_{11} = r_{22} = 1, r_{2} = r_{1} = -1/2$.
%\end{rem}
We recall the definition of Green's measure $G(z_0, \cdot)$ and $H_i(z_0,\cdot)$ from \eqref{mes_green} and \eqref{mes_bord}. Assumption~\eqref{drift} on the drift is crucial for the following proposition.

\begin{prop}[Densities and Laplace transforms]
Suppose that assumptions \eqref{drift} and \eqref{vectors} hold. Then, Green's measure $G(z_0, \cdot)$ has a density $g^{z_0}(\cdot)$ with respect to the Lebesgue measure. Functions $g^{z_0}(\cdot)$ are called Green's functions. Furthermore, measures $H_i(z_0,\cdot)$ ($i =1, 2$), have densities $f_i^{z_0}(\cdot)$ with respect to the one-dimensional Lebesgue measure.
\end{prop}

%The dependence in $z_0$ in $g^{z_0}(\cdot)$ and $h_i^{z_0}(\cdot)$ may be removed further in the notation for the sake of brevity. We just prove here the existence of the density $g^{z_0}(\cdot)$.

\begin{proof}%[Proof of the existence of $g^{z_0}$]

Let $A$ be a compact set of $\R_+^2$ at a positive distance of the edges. Define the stopping times:
$$\nr{\sigma = \inf\{t\geq\tau,\ Z_t \in A\}}\nb{\sigma = \inf\{t\geq0,\ Z_t \in A\}}, \quad\tau = \inf\{t\geq\sigma, \ Z_t \in \partial \R_+^2\}.$$
\nr{Following exactly the same arguments as in \cite{harrison_brownian_1987}, Lemma~$9$ of Section~$7$, we can reduce the proof to showing the fact that}
\nb{Considering the back-and-forth trajectories between $A$ and $\partial\R_+^2$ (see \cite[Lemma~$9$ of Section~$7$]{harrison_brownian_1987}), we can reduce the proof to showing that:} $$\E_{z_0}\left[\int_{\sigma}^{\tau} 1_{A}(Z_s)ds\right] = 0.$$
Then, by the Strong Markov property, it suffices to prove the result for a non-reflected degenerate Brownian motion. By Assumption~(\ref{drift}), rotating the plane so that the $x$-axis aligns with the drift direction reduces the problem to one-dimensional Brownian motion.
The proposition then follows from elementary properties of the latter.
\end{proof}

\begin{defi}[Laplace transforms of Green's measures] We denote the Laplace transforms of $G(z_0,\cdot)$ by
$$\varphi(x, y) := \E_{z_0}\left[\int_0^\infty e^{(x, y)\cdot Z_t}dt\right]
=\int_{\mathbb{R}_+^2} e^{(x,y)\cdot z}g^{z_0}(z) dz
$$ 
and the Laplace transforms of $H_1(z_0,\cdot), H_2(z_0,\cdot)$ by $$
\varphi_1(y):= \E_{z_0}\left[\int_0^\infty e^{(0,y)\cdot Z_t}dL^1_t\right],
\quad 
\varphi_2(x):= \E_{z_0}\left[\int_0^\infty e^{(x,0)\cdot Z_t}dL^2_t\right].$$ 
For brevity, we omit the dependence on the starting point in the notation for the Laplace transforms. However, when relevant, we will denote this dependence explicitly as $\varphi^{z_0}(x, y), \varphi_1^{z_0}(y)$ and $\varphi_2^{z_0}(x)$.
\end{defi}

\nr{Finally, we outline some properties of the Laplace transforms that will be useful later.
\begin{lemma}[Decay of Laplace transforms]\label{lap_decroit} Suppose \eqref{drift} to \eqref{simple}.
For any initial point $z_0 = (a_0, b_0)$ and $p\geq0$,
\begin{equation}\label{eq:lap_decroit}
\varphi_2^{z_0}(-p) \leq e^{-p(a_0+b_0)}\varphi_2^{z_0}(0).
\end{equation}
The symmetric result holds for $\phi_1^{z_0}$.
\end{lemma}
\begin{proof}
Note that the support of measure $H_2((a_0, b_0), \cdot)$ is $\left[a_0+b_0, +\infty\right)$, so
$$\varphi_2^{z_0}(-p) = \int_{a_0+b_0}^{+\infty}e^{-px}f_2^{z_0}(x)dx \leq e^{-p\left(a_0+b_0\right)}\int_{a_0+b_0}^{+\infty}f_2^{z_0}(x)dx = e^{-p\left(a_0+b_0\right)}\varphi_2^{z_0}(0).$$
\end{proof}}

\section{Functional equation, kernel and analytic continuation}\label{sec:3}
From now on, we assume~(\ref{drift}) to (\ref{simple}). As mentioned in the introduction, Laplace transforms $\phi, \phi_1, \phi_2$ are linked by a functional equation.

\begin{prop}[Functional equation] \label{functional}
If $Re(x) < 0$ and $Re(y) < 0$, then $\varphi_1(y)$, $\varphi_2(x)$ and $\varphi(x,y)$ converge and the following equation holds
\begin{equation}\label{Equation fonctionnelle}
    - \gamma(x, y)\varphi(x, y) = \gamma_1(x, y)\varphi_1(y) + \gamma_2(x, y)\varphi_2(x) + e^{(x, y)\cdot z_0}
\end{equation}
where $\gamma, \gamma_1$ and $\gamma_2$ are defined in \eqref{gamma}, \eqref{gamma12}.
%$$\begin{cases}
%        \gamma(x, y) = \frac{1}{2}(x, y)\cdot\Sigma(x, y) + (x, y)\cdot\mu = \frac{1}{2}(\sigma_{1}x^2 - 2\sqrt{\sigma_{1}\sigma_{2}}xy + \sigma_{2}y^2) + \mu_1 x + \mu_2 y\\
%        \gamma_1(x, y) = R^1\cdot (x, y) = x + r_{1}y \\ 
 %       \gamma_2(x, y) = R^2\cdot (x, y) = r_{2}x + y
%    \end{cases}$$
%    where $R^1, R^2$ are the two columns of the reflection matrix $R$. The polynomial $\gamma$ is called the kernel
\end{prop}

\begin{proof}
We apply Itô's formula to the semimartingale $(Z_t)_{t\geq0}$ and the function $
(u, v) \longmapsto e^{xu + yv}$. Then,
\begin{align}
        e^{(x,y)\cdot Z_t} - e^{(x,y)\cdot z_0} &= \int_0^t\ e^{(x,y)\cdot Z_t}(x,y)^T\cdot dB_s + \gamma(x,y)\int_0^te^{(x,y)\cdot Z_s}ds + \sum_{i=1}^2\gamma_i(x,y)\int_0^t e^{(x,y)\cdot Z_s}dL^i_s.
\end{align}
where $(B_t)_{t\geq0} = (X_t - \mu t)_{t\geq0}$ is the non reflected degenerate Brownian motion associated with the process (see Definition~\ref{defi}).
Next, taking the expectation and letting $t$ to $+\infty$, we derive \eqref{Equation fonctionnelle}\nr{by the same arguments as in \cite[Proposition 2.7]{franceschi2023asymptotics}.}\nb{. See \cite[Proposition 2.7]{franceschi2023asymptotics} for a detailed version of the proof in the non-degenerate case.} 
\end{proof} 
%The main difference is that $\Sigma$ has a rank one, so the Brownian motion locally move toward the vector $v$ in the interior of the cone. %\textcolor{gray}{We denote $\lambda = Tr(\Sigma) = \sigma_{1} + \sigma_{2}$ the non-null eigenvalue of $\Sigma$ and $v$ one normalized eigenvector associated.} 
%Then, $\gamma(x, y) = \frac{1}{2}\left((x, y)\cdot v\right)^2 + (x, y)\cdot\mu$ where $$v = \left(\sqrt{\sigma_1},-\sqrt{\sigma_2}\right).$$ 
%\color{gray}
%\begin{rem}
%To simplify the computations, we will take without loss of generality $\sigma_1 = \sigma_2 = 1$ (i.e. $v = (1, -1)$) and $\mu_1 + \mu_2 = 1$. The generalization to every $\sigma_1, \sigma_2$ and $\mu$ is done in Appendice \ref{App:A} by the axis-dilatation (i.e. change of variables) $$       \tilde x =   \frac{\sqrt{\sigma_1}}{\frac{\mu_1}{\sqrt{\sigma_1}}+\frac{\mu_2}{\sqrt{\sigma_2}}}x, \quad   \tilde y = \frac{\sqrt{\sigma_2}}{\frac{\mu_1}{\sqrt{\sigma_1}}+\frac{\mu_2}{\sqrt{\sigma_2}}}y.
%$$
%\end{rem}
%This notation allow us to treat only the case $v = \frac{1}{\sqrt 2}(1, -1)$ as is [Ichiba Karatzas].
Considering $\gamma(x, y)$ as a polynomial in $x$ (resp. $y$) with coefficients depending on $y$ (resp. $x$), we obtain two complex branches $Y^+(x)$, $Y^-(x)$ (resp. $X^+(y)$, $X^-(y)$) satisfying $\gamma(x, Y^\pm(x)) = \gamma(X^\pm(y), y) = 0$:% (\textcolor{blue}{on utilise $\sigma_{12}<0$})

\begin{equation}\label{Y-}
Y^{\pm}(x) = x  - \mu_2 \pm\sqrt{-2x + \mu_2^2}, \ \ X^{\pm}(y) =y  - \mu_1 \pm\sqrt{-2y + \mu_1^2}.
\end{equation}

We have one branching point $x_{max} = \frac{\mu_2^2}{2} > 0$ 
$\big($resp. $y_{max}= \frac{\mu_1^2}{2} > 0 \big)$ for $Y^{\pm}$ (resp. $X^\pm$). The square roots are chosen to be defined as holomorphic functions on $\C \backslash (-\infty, 0)$ and take non-negative values on the non-negative reals.

\begin{lemma} \label{lemXYY}
Let $u, v \in \R$ such that $u + iv \notin [x_{max}, +\infty[$. Then, we have:
\begin{equation}\label{Ymoins}
\Re(Y^\pm(u + iv)) =u - \mu_2
\pm \frac{1}{\sqrt{2}}\sqrt{\mu_2^2 - 2u  + \sqrt{(\mu_2^2 -2u)^2 + 4v^2}} 
\end{equation}
%\begin{equation}
%\Re(Y^-(u + iv)) = \frac{1}{\sigma_{2}}\left(\sqrt{\sigma_1\sigma_2}u - \mu_2 - \frac{1}{\sqrt{2}}\sqrt{\mu_2^2 - 2(\mu_2\sqrt{\sigma_1\sigma_2} + \mu_1\sigma_{2})u  + \sqrt{(\mu_2^2 -2(\mu_2\sqrt{\sigma_1\sigma_2} + \mu_1\sigma_{2})u)^2 + 4(\mu_2\sqrt{\sigma_1\sigma_2} + \mu_1\sigma_{2})^2v^2}} \right)
%\end{equation}
If $u, v \in \R$ satisfy $u + iv \notin [y_{max}, +\infty[$, then
\begin{equation}\label{Xmoins}
\Re(X^\pm(u + iv)) = u - \mu_1 \pm  \frac{1}{\sqrt{2}}\sqrt{\mu_1^2 - 2u + \sqrt{(\mu_1^2 - 2u)^2 + 4v^2}}.
\end{equation}
Let $\delta = \min(\mu_1, \mu_2) > 0$. Then, $\Re(Y^-(x)) < 0$ for all $x$ such that $\Re(x) < x_{max} + \delta, x \notin [x_{max}, +\infty[$. Similarly, $\Re(X^-(y)) < 0$ for all $y$ such that $\Re(y) < y_{max} + \delta, y \notin [y_{max}, +\infty[$.
\end{lemma}
%\color{gray}
%\begin{proof}
%The equations (\ref{Ymoins}) and (\ref{Xmoins}) are directly obtained by solving $(a + ib)^2 = {2(-\mu_2\sqrt{\sigma_1\sigma_2} - \mu_1\sigma_{2})x + \mu_2^2}$. 
%\end{proof}
%In particular, $\Re(Y^-(u + iv)) \leq  \frac{1}{\sigma_{2}}\left(\sqrt{\sigma_1\sigma_2}u - \mu_2\right)$. Since the quantity $$Y^\pm(x_{max}) = \frac{1}{\sigma_{2}}\left(\frac{(\sqrt{\sigma_1\sigma_2})\mu_2^2}{2(-\mu_2\sqrt{\sigma_1\sigma_2} - \mu_1\sigma_{2})} - \mu_2\right) \leq -\frac{1}{\sigma_{2}}\mu_2$$ 
%\color{black}

\begin{proof}
Equations (\ref{Ymoins}) and (\ref{Xmoins}) follow directly from the expression~\eqref{gamma} of $\gamma$. The last statements come from the inequalities $x_{max} = \frac{\mu_2^2}{2} < \mu_2$ and $y_{max} = \frac{\mu_1^2}{2} < \mu_1$.
\end{proof}
%In particular, we can continue Laplace transforms $\phi_1$ and $\phi_2$ in a larger domain. Indeed, replacing $(x,y) = (x, Y^-(x))$ in the functional equation (\ref{Equation fonctionnelle}) we get a meromorphic expression of $\phi_2$ which is valid in an extended domain. We can do the symmetric reasoning for $\phi_1$.

%
%\begin{figure}[hbtp]
%\centering
%\includegraphics[scale=1.2]{parabole_simple}
%\vspace{-0.1cm}
%\caption{Parabola. $V$ is the Vertex of the parabola and $v$ a vector of direction of the Brownian motion. The drift $\mu$ is orthogonal to the parabola at he origin.}
%\label{parabolesimple}
%\end{figure}
%

\begin{cor}[Continuation of Laplace transforms] \label{coro:cont}
The Laplace transforms $\varphi_1$ and $\varphi_2$ can be extended as meromorphic functions on $\{y \in \C, \Re(y) < y_{max} + \delta\} \backslash[y_{max}, y_{max} + \delta]$ and $\{x \in \C, \Re(x) < x_{max} + \delta\} \backslash[x_{max}, x_{max} + \delta]$ respectively via the formulae:
\begin{equation}
    \label{cont2}
\varphi_1(y)=  \frac{ -\gamma_2(X^-(y),y) \varphi_2(X^-(y)) -\exp \big(a_0 X^-(y) + b_0 y \big)}{\gamma_1(X^-(y),y)}
\end{equation} 
\begin{equation}
    \label{cont}
\varphi_2(x)=  \frac{ -\gamma_1(x,Y^-(x)) \varphi_1(Y^-(x)) -\exp \big(a_0 x + b_0 Y^{-}(x) \big)}{\gamma_2(x, Y^{-}(x))}.
\end{equation} 
\end{cor}
\begin{proof}
    This follows directly from Lemma~\ref{lemXYY} and the functional equation~\eqref{Equation fonctionnelle}.
\end{proof}

From now on, $\phi_1$ and $\varphi_2$ will be considered over their extended domains. 
%Those continuation formula give then the explicit poles of $\varphi_1$ and $\varphi_2$.
%\color{gray}
%\begin{prop}[Poles of the Laplace transform] 
%\begin{itemize} 
%\item[(i)] $x=0$ (resp. $y = 0$) is not a pole of $\phi_2(x)$ (resp. $\varphi_1(y)$).
%\item[(ii)]  If $x^{*}$ (resp. $y^{**}$) is a pole of $\phi_2(x)$ (resp. $\phi_1(y))$ ,%$\{x= u+iv: \   u  \in ] 0, x_{max} +\delta[  \}$
%then  $(x^*, Y^{-}(x^*))$ (resp. $(X^{-}(y^{**}), y^{**})$ is a unique non-zero solution of the system of two equations 
%\begin{equation}
%    \label{za}
% \gamma(x,y)=0, \  \  \   \ \gamma_2(x,y)= r_{2}x+ y=0 \quad (resp.\quad \gamma(x,y)=0, \  \  \   \  \gamma_1(x,y)=x+ r_{1}y=0).
%\end{equation}  
%   Moreover, $x^*$ (resp. $y^*$) is real and belongs to $]0, x_{max}[$ (resp. $]0, y_{max}[$).
%Furthermore, this solution exist only if and only if
%   $$x_{max}r_{2} + Y^{\pm}(x_{max}) > 0 \quad (resp. \quad y_{max}r_{1} + X^\pm(y_{max}) > 0).$$
%\end{itemize} 
%When these solutions exist we have
%\begin{equation}
%x^* = 2\frac{\mu_2r_{2} - \mu_1}{\sigma_{1} + 2\sqrt{\sigma_1\sigma_2}r_{2}+ \sigma_{2}r_{2}^2} 
% \quad \text{and} \quad
%  y^{**} = 2\frac{\mu_1r_{1} - \mu_2}{\sigma_{1}r_{1}^2 + 2\sqrt{\sigma_1\sigma_2}r_{1} + \sigma_{2}}.
%  \label{eq:defx*y**}
%\end{equation}
%In this case, we define $y^* = Y^+(x^*)$ and $x^{**} = X^+(y^*)$.
%\label{pole} 
%\end{prop} \color{black}
Let us define
\begin{equation}
x^* = 2\frac{\mu_2r_{2} - \mu_1}{(1 + r_{2})^2}, 
 \quad \quad
  y^{**} = 2\frac{\mu_1r_{1} - \mu_2}{(r_{1} + 1)^2}.
  \label{eq:defx*y**}
\end{equation}
If equation $\gamma_2(x, Y^{-}(x)) = 0$ (resp. $\gamma_1(X^-(y),y) = 0$) has a solution in the complex plane, then it is unique and is given by $x = x^*$ (resp. \nr{$y=y^*$ }\nb{$y=y^{**}$}). We also define
\begin{equation}\label{eq:defx**y*}
 y^* = Y^+(x^*),\ \ \ \nr{x^{**} = X^+(y^*) }\nb{x^{**} = X^+(y^{**}) }, \end{equation}
\nb{see Figure~\ref{fig:x*y*x**y**}.}
\begin{prop}[Poles of Laplace transform] 
$ $

\begin{itemize} 
\item[(i)] $x=0$ (resp. $y = 0$) is not a pole of $\phi_2(x)$ (resp. $\varphi_1(y)$).
\item[(ii)]  If $x$ (resp. $y$) is a pole of $\phi_2(x)$ (resp. $\phi_1(y))$, %$\{x= u+iv: \   u  \in ] 0, x_{max} +\delta[  \}$
then $x = x^*$ (resp. $y = y^{**}$) and $\gamma_1(x^*, Y^{-}(x^*)) = 0$ (resp. $\gamma_2(X^-(y^{**}),y^{**}) = 0$). Furthermore, $x^*$ is a pole of $\phi_2$ (resp. $y^{**}$ is a pole of $\phi_1$)  if and only if $(r_{2} + 1)\mu_2 > 2 \quad (resp. \quad (r_{1} + 1)\mu_1 > 2).$
\end{itemize} 
\label{pole} 
\end{prop}
%Note that If this condition is satisfied, then $x^* \in ]0, x_{max}[$ (resp. $y^* \in ]0, y_{max}[$).
\begin{figure}[hbtp]
\centering
\includegraphics[scale=1.6]{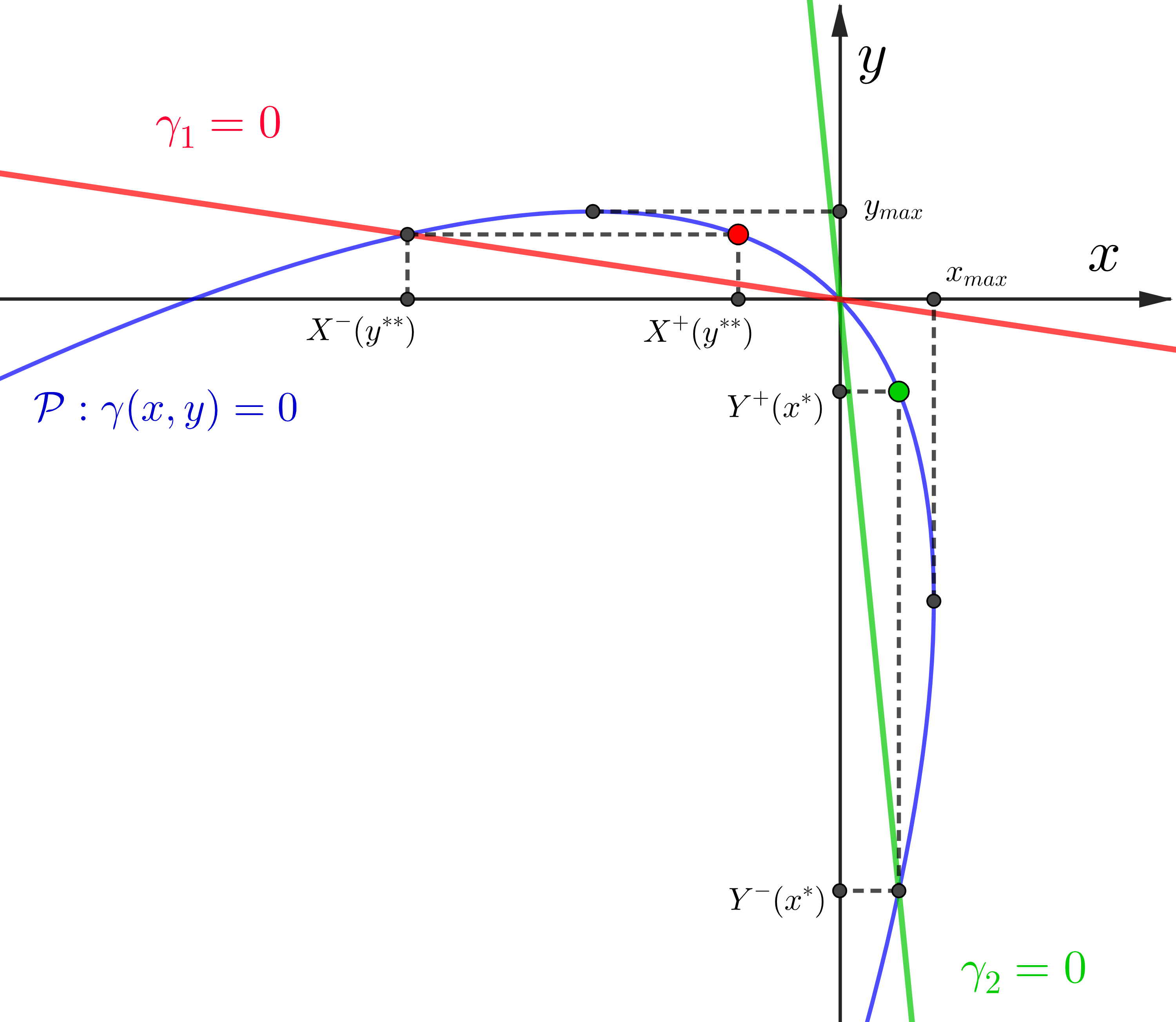}
\vspace{-0cm}
\caption{In the case of the figure, both $\phi_1$ and $\phi_2$ have a pole.}
\label{fig:x*y*x**y**}
\end{figure}

\begin{proof}
The first point follows from the continuation formula (\ref{cont}) since \nr{$\gamma_2(0, Y^-(0)) = -\mu_2 - \mu_2^2/2 \neq 0$ }\nb{$\gamma_2(0, Y^-(0)) = -2\mu_2\neq0$ }.

For (ii), if $x$ is a pole of $\phi_2$, it follows from \eqref{cont} that $\gamma_2(x, Y^-(x)) = 0$, which implies that $x = x^*.$  Moreover, the Laplace transform $\phi_2$ is holomorphic in ${\rm Re}(x)<0$. Thus, $x^*$ being a pole of $\phi_2$, must be positive. Note that equation $\gamma_2(x, Y^{-}(x)) = 0$ has a positive solution if and only if $\gamma_2(x_{max}, Y^{\pm}(x_{max})) > 0$. This last condition is equivalent to $(r_{2} + 1)\mu_2 > 2$.

Let us assume that $(r_{2} + 1)\mu_2 > 2$. Then, $x^* > 0$. Since $\gamma_2(x^*, Y^-(x^*)) = 0$, it follows from \eqref{cont} that $x^*$ is a pole of $\phi_2$ if the numerator of the right-hand side of \eqref{cont2} does not vanish at $x^*$. The last fact holds true, and \nb{is actually equivalent to the non-nullity of the function $h_{\alpha^*}(z_0)$ defined in \eqref{harm2}: this equivalence and the non-nullity are postponed to the end of Section~\ref{sub:6.2}.}\nr{ its proof is postponed to Section~\ref{sub:7.2} as a consequence of Theorem~\ref{directions_poles}}.
\end{proof}

The following proposition provides some estimates for the Laplace transforms\nr{and completes Lemma~\ref{lap_decroit}}\ngray{$\,$I deleted this part because Lemma~\ref{lap_decroit} is now after}. These estimates will be useful in Section~\ref{sec:6}.

\begin{prop}[\nb{Decay of Laplace transforms on $\textnormal{Re} = -\epsilon$}]\label{lap_dec_inv}
    Let $z_0 = (a_0, b_0) \in \R_+^2$ be an initial condition with $a_0 \neq 0, b_0\neq 0$ and $\epsilon > 0$. Then there exist constants $c, C > 0$ such that for $l=1, 2$,
    \begin{equation}\label{eq:lap_dec_inv}
        \forall v \in \R,\quad\quad|\phi_l^{z_0}(-\epsilon + iv)| \leq Ce^{-c\sqrt{|v|}}.
    \end{equation}
\end{prop}
\begin{proof}
    The expressions $\gamma_1(-\epsilon + iv, Y^-(-\epsilon + iv))$ and $\gamma_2(-\epsilon + iv, Y^-(-\epsilon + iv))$ grow linearly with respect to $v$ as $v$ tends to $\pm\infty.$ Furthermore, expression~\eqref{Ymoins} provides inequality $\exp(b_0Y^-(-\epsilon + iv)) \leq C_1e^{-b_0\sqrt{c_1|v|}}$ for some constants $c_1, C_1 > 0$ and  $$\phi_1^{z_0}(Y^-(-\epsilon + iv)) \leq \phi_1^{z_0}(0)e^{(a_0 + b_0)Re(Y^-(-\epsilon + iv))} \leq C_2 e^{-c_2\sqrt{|v|}}$$
    for some constants $c_2, C_2 > 0.$
    Finally, equation~\eqref{cont} implies the conclusion for $\phi_2.$ The proof for $\phi_1$ is analogous. 
\end{proof}
\nb{We also give some further estimates for Laplace transforms that will be useful in Section~\ref{subsec:4.3}. \ngray{This Lemma has just been moved from Section 2 to this place.}
\begin{lemma}[Decay of Laplace transforms]\label{lap_decroit} Assume that \eqref{drift} to \eqref{simple} hold.
For any initial point $z_0 = (a_0, b_0)$ and $p\geq0$,
\begin{equation}\label{eq:lap_decroit}
\varphi_2^{z_0}(-p) \leq e^{-p(a_0+b_0)}\varphi_2^{z_0}(0).
\end{equation}
The symmetric result holds for $\phi_1^{z_0}$.
\end{lemma}
\begin{proof}
By \eqref{vectors}, note that the support of the measure $H_2((a_0, b_0), \cdot)$ is $\left[a_0+b_0, +\infty\right)$. Then,
$$\varphi_2^{z_0}(-p) = \int_{a_0+b_0}^{+\infty}e^{-px}f_2^{z_0}(x)dx \leq e^{-p\left(a_0+b_0\right)}\int_{a_0+b_0}^{+\infty}f_2^{z_0}(x)dx = e^{-p\left(a_0+b_0\right)}\varphi_2^{z_0}(0).$$
\end{proof}}

\section{The compensation method \nb{and the explicit expressions of the Laplace transforms}}\label{sec:5}

\subsection{Heuristic of the compensation method}\label{subsec:5.1}
\nr{If $h$ is}\nb{Let $h$ be} a smooth function satisfying the following partial differential equation with boundary conditions:
\begin{equation}\label{BVP}
\left\{
    \begin{array}{ll}
        {(H_0)} \quad\mathcal{G}h = 0 \quad\quad\quad \quad\quad \textnormal{on} \quad (0, +\infty)^2   \\
	  {(H_1)} \quad\partial_{R_1}h(0, y) = 0, \quad y \geq 0\\
	  {(H_2)} \quad\partial_{R_2}h(x, 0) = 0,  \quad x \geq 0
    \end{array}
\right.
\end{equation}
\nb{with  $\mathcal{G} = \frac{1}{2}\nabla\cdot\Sigma\nabla + \mu\cdot\nabla$,} then $h$ is \nr{a }harmonic\nr{ function} (see~\cite[Section $6$]{ernst_franceschi_asymptotic_2021}). To demonstrate this, one may apply Itô's formula to the process $(Z_t)_{t\geq0}$ and $h \in C^2(\R_+^2, \R)$:
$$h(Z_t) = h(Z_0) + \int_0^t \nabla h(Z_s)dB_s + \int_0^t\mathcal{G}h(s)ds + \sum_{i=1}^2\int_0^t R_i\cdot\nabla h(Z_s)dL^i_s$$
where $(B_t)_{t\geq0} = (X_t - \mu t)_{t\geq0}$ is the non reflected degenerate Brownian motion associated with the process (see Definition~\ref{defi}). If $h$ satisfies~(\ref{BVP}), then $h(Z_t) = h(Z_0) + \int_0^t \nabla h(Z_s)dB_s$, and \nb{thus} $\E[h(Z_t)] = \E[h(Z_0)]$ \nb{(at least formally),} which implies that $h$ is harmonic (cf~\eqref{def:harm}).

The principle of the compensation method is to find functions of the form $h(x, y) = \sum_{n\in\Z}c_ne^{a_n x + b_n y}$ such that each exponential term satisfies condition ${(H_0)}$: $\mathcal{G}e^{a_n x + b_n y} = 0$ (i.e. $(a_n, b_n) \in \mathcal{P}$, see Figure~\ref{autoo1}) and to ``compensate'' the constants $(c_n)_{n\in\Z}$ \nb{so as} to ensure that conditions ${(H_1)}$ and ${(H_2)}$ are satisfied. 
We require:
\begin{equation}\label{formeh}
	h(x,y) = ... + \overunderbraces{&\br{2}{\in (H_2)}& &\br{3}{\in (H_2)}&}%
  { &c_{-2}e^{a_{-2} x + b_{-2} y} +  &c_{-1}e^{a_{-1} x + b_{-1} y} &+& c_0e^{a_0 x + b_0 y} &+&c_1e^{a_1 x + b_1 y}&+ c_2e^{a_2 x + b_2 y}& +...\;}%
  {& &\br{3}{\in (H_1)}& &\br{2}{\in (H_1)}&}.
\end{equation} 
Given that conditions $(H_0), (H_1), (H_2)$ are linear, it follows that $h$ is a harmonic function. By a direct computation, we find that conditions $(H_2)$ \nr{in}\nb{on} the right-hand side of~\eqref{formeh} are satisfied if and only if $a_{2k} = a_{2k+1}$ and
$c_{2n + 1}= -\frac{\gamma_2(a_{2n}, b_{2n})}{\gamma_2(a_{2n+1}, b_{2n+1})}c_{2n}$ for any integer $k.$ 
Similarly, conditions $(H_1)$ in the right-hand side of~\eqref{formeh} are satisfied if and only if $b_{2n+1} = b_{2n+2}$ and $c_{2n + 2} = -\frac{\gamma_1(a_{2n+1}, b_{2n+1})}{\gamma_1(a_{2n+2}, b_{2n+2})}c_{2n+1}$ for any integer $n$.

\nb{We will see in Section~\ref{sub:6.1} that the harmonic functions we obtain can be written as $$(x,y) \longmapsto \gamma_1(a_0,b_0)\phi_1^{(x,y)}(b_0) + \gamma_2(a_0,b_0)\phi_2^{(x,y)}(a_0) + e^{a_0x + b_0y}.$$ The explicit expressions of $\phi_1$ and $\phi_2$ in Section~\ref{subsec:4.3} then provide the exact formula \eqref{formeh} suggested by the compensation method. Moreover, the approach of Section~\ref{sub:6.1} justifies why the harmonic functions given by \eqref{formeh} are non-negative when $(a_0, b_0)$ is well chosen.}

\subsection{Parabola and automorphisms}
Let us recall that $\mathcal{P}$ is the parabola defined by $\mathcal{P} = \{(x, y) \in \R^2,\ \ \gamma(x, y) = 0\}$ (see \eqref{paraboleP}). \nb{Before defining the sequence $((a_n,b_n))_{n\in\Z}$ motivated by Section~\ref{subsec:5.1} (see Figure~\ref{autoo1}), we first give a parametrisation of $\mathcal{P}$.}

\nr{\begin{defi}[Automorphisms $\zeta$ and $\eta$]
For all $(x, y) \in \mathcal{P}$, equation $\gamma(x,v) = 0$ in $v$ (resp. $\gamma(u,y) = 0$ in $u$) has two solutions $y$ and $\tilde y$ (resp. $x$ and $\tilde x$), which may be equal. We then define $\zeta$ (resp. $\eta$) the automorphism on $\mathcal{P}$ as $\zeta(x,y) = (x, \tilde y)$ (resp. $\eta(x,y) = (\tilde x, y)$).
\end{defi}}
\ngray{I moved here the following proposition}
\begin{prop}[Parameterisation of $\mathcal{P}$]\label{x(s)}
The parabola~$\mathcal{P}$ (see~\eqref{paraboleP}) admits the following parameterisation:
\begin{equation}\label{parametrisation}
  \left\{
    \begin{array}{ll}
        x(s) &=   -\frac{1}{2}s(s - 2\mu_2)\\
      y(s) &= - \frac{1}{2}s\left(s + 2\mu_1\right)
    \end{array}
\right. , \quad s \in \R.  
\end{equation}
This means that $\{(x,y) \in \R^2, \ \gamma(x,y) = 0\} = \{(x(s),y(s)), \ s \in \R\}$.
\end{prop}

\begin{proof}
The relation $\gamma(x(s),y(s)) = 0$ is easily verified by substituting $x(s), y(s)$ into the \nr{equation for}\nb{expression \eqref{gamma} of} $\gamma(x,y)$. Furthermore, the parameterisation is injective. To show this, assume that $$\left\{
    \begin{array}{ll}
           s\left(s - 2\mu_2\right) = s'\left(s' - 2\mu_2\right)\\
           s\left(s + 2\mu_1\right) = s'\left(s' + 2\mu_1\right)
    \end{array}
\right..$$ Subtracting the second equation from the first gives $2s(\mu_1 + \mu_2) = 2s'(\mu_1 + \mu_2)$ which implies $s = s'$. Similarly, surjectivity can be verified by elementary considerations.% ensuring that the entire parabola is covered by the parameterisation.
\end{proof}
\ngray{There in no additional content here, just a rearrangement.}
\nb{To define the ``downstairs'' as in Figure~\ref{autoo1}, we introduce two transformations on the parabola which leave the first (resp. second) coordinate invariant. This is the aim of the following proposition (which also serves as a definition). This proposition is illustrated by Figure~\ref{automorphismes}.
\begin{prop}[Automorphisms $\eta, \zeta$]
For $s \in \R$, we define
\begin{equation}\label{eta_zeta}
    \left\{
    \begin{array}{ll}
        \zeta s = -s + 2\mu_2\\
	  \eta s  = -s - 2\mu_1.
    \end{array}
\right.
\end{equation}
Then, $x(\zeta s) = x(s)$ and $y(\eta s ) = y(s)$ for all $s \in \R$. Therefore, $\phi_2(x(\zeta s)) = \phi_2(x(s))$ and $\phi_1(y(\eta s)) = \phi_1(y(s))$ in their respective domains of definition. Furthermore, for all $n\in\Z$ and $s\in \R$, we have
\begin{equation}\label{eq:etazetan}
(\eta \zeta)^n s =\nr{s+2n }\nb{s - 2n}.
\end{equation}
\end{prop}}

\begin{proof}
The formulae $x(-s + 2\mu_2) = x(s)$ and $y(-s-2\mu_1) = y(s)$ are easily verified. \nr{Then, expressions of $\zeta, \eta$ follow from elementary considerations of the second degree polynomial $\gamma(x,y)$ in $x, y$.} The expression of $(\eta\zeta)^n$ is a consequence of expressions of $\eta, \zeta$ and of the equation $\mu_1 + \mu_2 = 1$ (see Assumption~\eqref{simple}).
\end{proof}

Note that $\zeta^2 = Id$, $\eta^2 = Id. $\nr{ and that $\zeta$ (resp. $\eta$) leaves the first (resp. second) coordinate invariant (see Figure~\ref{automorphismes})} By the parameterisation \nb{\eqref{parametrisation}}, \nr{made in Proposition~\ref{x(s)}, }$\zeta$ and $\eta$ can be regarded as reflections \nb{(see \eqref{eta_zeta})}, and their composition as a translation \nb{(see \eqref{eq:etazetan}).} \nr{This is the purpose of the following proposition.}

\begin{figure}[hbtp]
\centering
\includegraphics[scale=5]{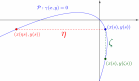}
\vspace{0cm}
\caption{Parabola $\mathcal{P}$ and automorphisms $\eta$ and $\zeta$.}
\label{automorphismes}
\end{figure}

\begin{lemma}[Explicit form of $(a_n, b_n)$]\label{anbn}
\nr{Let $(a_0, b_0) \in \mathcal{P}$. }\nb{Let $s\in \R$ and $(a_0, b_0) = (x(s),y(s))$.} For any integer $n \in \Z$, we define
\nr{$$(a_{2n}, b_{2n}) = (\eta\zeta)^n(a_0,b_0), \quad (a_{2n+1}, b_{2n+1})  =  \zeta(\eta\zeta)^n(a_0,b_0)$$}
\nb{$$(a_{2n}, b_{2n}) = \left(x((\eta\zeta)^ns),y((\eta\zeta)^ns)\right), \quad (a_{2n+1}, b_{2n+1})  = \left(x(\zeta(\eta\zeta)^ns), y(\zeta(\eta\zeta)^ns)\right)$$}
(see Figure \ref{autoo1}). Then, for any $n \in \Z$ the following expressions hold.
\begin{equation}
a_{2n} = -2n^2 +  2(a_0 - b_0 - \mu_2)n + a_0, \ \ a_{2n+1} = a_{2n}
\end{equation}
\begin{equation}b_{2n} = -2n^2 + 2(a_0 - b_0 + \mu_1) n + b_0, \ \ b_{2n+1} = b_{2n+2}.
\end{equation}
\end{lemma}

\begin{proof}
The invariance of the first and the second coordinate of $\zeta$ and $\eta$ respectively implies equalities $a_{2n+1} = a_{2n}$ and $b_{2n+1} = b_{2n+2}$. \nr{For the}\nb{The} explicit expressions of $a_{2n}$ and $b_{2n}$\nr{, they} are \nb{obtained from the explicit expression \eqref{eq:etazetan}.} \nr{found by setting up and solving an elementary linear recurrence relation of second order.}
\end{proof}

\begin{notation} \label{not:1} \label{not:2} \nb{For $s\in\R$, we define}
\begin{equation}
z(s) = (x(s), y(s))
\end{equation}
\nb{as the point of the parabola corresponding to the parameter $s \in \R.$} We also define
\begin{equation}
s_{max} = \mu_2, \ \ s_{min}  = -\mu_1
\end{equation}
\nr{and use the notation}\nb{and write} $\gamma_i(s)$ \nb{instead of} $\gamma_i(x(s), y(s))$ for $i=1, 2$. \nr{For $s \in \R$, we also denote $\zeta s$ (resp. $\eta s$) the real number satisfying $\zeta z(s) = z(\zeta s)$ (resp. $\eta z(s) = z(\eta s)$).}\ngray{This notation is no longer necessary since $\eta$ and $\zeta$ only act on $s \in \R$.} \nr{We also define}\nb{Finally, let} $s^*$ and $s^{**}$ \nb{be defined} as
\begin{equation}\label{s*}
s^* = \frac{2}{r_{2} + 1}, \ \ \ s^{**} =  \frac{-2}{r_{1} + 1}.
\end{equation}
 % such that $(x(s^{**}), y(s^{**})) = (X^+(y_{max}), y_{max})$.
\end{notation}
With these notations, $$z(s_{max}) = (x_{max}, Y^{\pm}(x_{max}))\quad \textnormal{and}\quad z(s_{min}) = (X^\pm(y_{max}), y_{max}).$$ Note that \nb{the curve} $(z(s))_{s \in [s_{min}, s_{max}]}$ is the \nr{part}\nb{portion} of the parabola from $(x_{max}, Y^+(x_{max}))$ to $(X^+(y_{max}), y_{max})$ going counterclockwise (see Figure~\ref{fig:x*y*x**y**}).
Furthermore, $x(s^*) = x^*$ and $y(s^{**}) = y^{**}$ with definition~\eqref{eq:defx*y**}. %Conditions $(r_2 + 1)\mu_2 > 2$ and $(r_1 + 1)\mu_1 > 2$ of Proposition~\ref{pole} are equivalent to $s^* > \mu_2$ and $s^{**} < -\mu_1$ respectively. 
We can now provide explicit expressions for the Laplace transforms $\phi_1$ and $\phi_2.$

\subsection{Explicit expression of Laplace transforms via the compensation approach}\label{subsec:4.3}

\begin{theorem}[Explicit expressions for Laplace transforms]\label{harmoniques}
%\textcolor{red}{On a plus besoin de supposer que $z_0 \neq (0,0)$.} 
Let $z_0 \in \R_+^2\setminus\{(0,0)\}$ be the initial condition. Then, for any $s \in (\max(s_{min}, s^{**}), \min(s_{max}, s^*))$,%
\nr{\begin{align}\phi_2(x(s)) &=  \frac{-1}{\gamma_2(\zeta s)}e^{z_0\cdot z(\zeta s)} + \sum_{n=1}^{+\infty} \left[\prod_{k=0}^{n-1} G(s+2k)\right]\left[  \frac{e^{z_0\cdot z(s + 2n)}}{\gamma_2(s + 2n)} - \frac{e^{z_0\cdot z(\zeta (s+2n))}}{\gamma_2(\zeta (s+2n))}\right] 
\end{align}} 
\nb{\begin{align}\label{explicite}\phi_2^{z_0}(x(s)) &=  \frac{-1}{\gamma_2(\zeta s)}e^{z_0\cdot z(\zeta s)} + \sum_{n=1}^{+\infty} \left[\prod_{k=0}^{n-1} G(s-2k)\right]\left[  \frac{e^{z_0\cdot z(s - 2n)}}{\gamma_2(s - 2n)} - \frac{e^{z_0\cdot z(\zeta (s-2n))}}{\gamma_2(\zeta (s-2n))}\right] 
\end{align}}
where 
\begin{equation}\label{eq:G} 
    G(s) = \nr{\frac{\frac{\gamma_1}{\gamma_2}(\zeta s)}{\frac{\gamma_1}{\gamma_2}(s + 2)}}\nb{\frac{\frac{\gamma_1}{\gamma_2}(\zeta s)}{\frac{\gamma_1}{\gamma_2}(s - 2)}}.
\end{equation}

\color{black}
Similarly, for all $s \in (\max(s_{min}, s^{**}), \min(s_{max}, s^*))$ %\textcolor{red}{Exactly when $\gamma_2(s - nq) \neq 0 \quad \forall n \geq 1$},
\nr{\begin{align}\phi_1(y(s)) &=  \frac{-1}{\gamma_1(\eta s)}e^{z_0\cdot z(\eta s)} + \sum_{n=1}^{+\infty} \left[\prod_{k=0}^{n-1} \tilde G(s-2k)\right]\left[  \frac{e^{z_0\cdot z(s - 2n)}}{\gamma_1(s - 2n)} - \frac{e^{z_0\cdot z(\eta (s-2n))}}{\gamma_1(\eta (s-2n))}\right] 
\end{align}}
\nb{\begin{align}\label{explicite2}\phi_1^{z_0}(y(s)) &=  \frac{-1}{\gamma_1(\eta s)}e^{z_0\cdot z(\eta s)} + \sum_{n=1}^{+\infty} \left[\prod_{k=0}^{n-1} \tilde G(s+2k)\right]\left[  \frac{e^{z_0\cdot z(s + 2n)}}{\gamma_1(s + 2n)} - \frac{e^{z_0\cdot z(\eta (s+2n))}}{\gamma_1(\eta (s+2n))}\right] 
\end{align}}
where $$\tilde G(s) =  \nr{ \frac{\frac{\gamma_2}{\gamma_1}(\eta s)}{\frac{\gamma_2}{\gamma_1}(s - 2)}}\nb{\frac{\frac{\gamma_2}{\gamma_1}(\eta s)}{\frac{\gamma_2}{\gamma_1}(s + 2)}}.$$ %\textcolor{red}{We have the same result as $\phi_2$ in the case that $\gamma_2(s - Nq) = 0$ for some $N \geq 1$}.
\end{theorem}

Before \nr{applying the compensation method}\nb{proving Theorem~\ref{harmoniques}}, we establish a technical lemma.
\begin{lemma}\label{lem:gamma(s)nonnul}
    For all $n \geq 1$ and $s \in (\max(s_{min}, s^{**}), \min(s_{max}, s^*))$, we have \nr{$\gamma_1(s + 2n) \neq 0$}\nb{$\gamma_1(s - 2n) \neq 0$} and \nr{$\gamma_2(\zeta(s + 2n)) \neq 0$}\nb{$\gamma_2(\zeta(s - 2n)) \neq 0$}. Furthermore, $\gamma_2(s), \gamma_2(\zeta s), \gamma_1(s), \gamma_1(\eta s)$ are also non-zero.
\end{lemma}
\begin{proof}
    We define \nb{two portions of the parabola $E^+$ and $E^-$ given by} $$E^+ = \{(x, Y^+(x)), \; x \leq X^+(0)\} \quad \textnormal{and} \quad E^- = \{(X^-(y)), y), \; y \leq Y^-(0)\}.$$  By Assumption~(\ref{vectors}), the line $\{\gamma_2 = 0\}$ (resp. $\{\gamma_1 = 0\}$) cannot pass through $E^-$ (resp. $E^+$). Additionally, note that $\eta(E^-) \subset E^+$ and $\zeta(E^+) \subset E^-$. Since $s \in (s_{min}, s_{max})$, $z((\eta\zeta)^n s) = \nr{z(s + 2n)}\nb{z(s - 2n)}$ belongs to $E^+$ for all $n \geq 1$. Thus, $\nr{\gamma_2(\zeta(s + 2n)) \neq 0}\nb{\gamma_2(\zeta(s - 2n)) \neq 0}$ for any $n \geq 0$. By similar reasoning, $\nr{\gamma_1(s + 2n) \neq 0}\nb{\gamma_1(s - 2n) \neq 0}$ for any $n\geq0$. The last statement comes from the fact that $s \in (s^{**}, s^*).$
\end{proof}
\begin{proof}[Proof of Theorem~\ref{harmoniques}] \ngray{I didn't put in red and blue the change of sign in for $s\pm 2k$ because of the amount of those in this proof, but I closely paid attention.}
\nr{We now apply the compensation method}\nb{The main idea of the proof is to get a recursive formula for Laplace transforms}. To do this, we rewrite the functional equation~(\ref{Equation fonctionnelle}) in $z(\zeta s)$ and $z(\eta \zeta s) = z(s - 2)$, which holds because $x(\zeta s), y(\zeta s), x(s-2)$ and $y(s-2)$ are negative:
$$
\left\{
    \begin{array}{ll}
         0 = \gamma_1(\zeta s) \phi_1(y(\zeta s)) + \gamma_2(\zeta s) \phi_2(x(\zeta s)) + e^{z_0\cdot z(\zeta s)}  \\
	   0 = \gamma_1(s - 2) \phi_1(y(s -2)) + \gamma_2(s - 2) \phi_2(x(s - 2)) + e^{z_0 \cdot z(s-2)}.
    \end{array}
\right. 
$$
By the invariance of $\phi_2$ (resp. $\phi_1$) under $\zeta$ (resp. $\eta$), we have $\phi_2(x(\zeta s)) = \phi_2(x(s))$ and $\phi_1(y(s-2)) = \phi_1(y(\zeta s))$.
%We have $\gamma_1(s) \neq 0$ because of the choice of $s$.
%If $\gamma_1(\zeta s) = 0$, then $\gamma_2(\zeta s) 
%\neq 0$ (this is obvious in the case $\det(R) > 0$, in the other case this comes from the fact that the choice of the beginning zone for $s_0$). Then by a straightforward replacement, $$\phi_1(s) = \frac{1}{\gamma_1(s)}\left(\gamma_2(s)\frac{e^{(x(\zeta s), y(\zeta s))\cdot z_0}}{\gamma_2(\zeta s)} - e^{(x(s), y(s))\cdot z_0}\right).$$
Then, by eliminating $\phi_1(y(\zeta s))$ from the equations \nb{(which is possible by Lemma~\ref{lem:gamma(s)nonnul})}, we obtain:
\begin{align}
\phi_2(x(s)) &= \frac{\frac{\gamma_1}{\gamma_2}(\zeta s)}{\frac{\gamma_1}{\gamma_2}(s-2)} \phi_2(x(s - 2))+\left[\frac{\frac{\gamma_1}{\gamma_2}(\zeta s)e^{z(s-2)\cdot z_0}}{\gamma_1(s-2)} - \frac{e^{z(\zeta s)\cdot z_0}}{\gamma_2(\zeta s)} \right]\\
&= \label{1pas} G(s)\phi_2(x(s - 2))+ \left[\frac{G(s)}{\gamma_2(s - 2)}e^{z_0\cdot z(s - 2)} - \frac{e^{z_0\cdot z(\zeta s)}}{\gamma_2(\zeta s)}\right].
\end{align}
Similarly, we get:
\begin{equation}
\phi_2(x(s-2)) = G(s-2)\phi_2(x(s - 4))+ \left[\frac{G(s-2)}{\gamma_2(s - 4)}e^{z_0\cdot z(s - 4)} - \frac{e^{z_0\cdot z(\zeta (s-2))}}{\gamma_2(\zeta (s-2))}\right].
\end{equation}
Substituting this into~\eqref{1pas}, we get $$\phi_2(x(s)) = G(s)G(s - 2)\phi_2(x(s - 4)) + G(s)G(s - 2)\frac{e^{z_0\cdot z(s - 4)}}{\gamma_2(s-4)} - G(s)\frac{e^{z_0\cdot z(\zeta(s-2))}}{\gamma_2(\zeta(s-2))} + \frac{G(s)}{\gamma_2(s -2)}e^{z_0\cdot z(s - 2)} - \frac{e^{z_0\cdot z(\zeta s)}}{\gamma_2(\zeta s)}.$$
Then, by induction on $N$, we obtain the following equality for all $N \geq 1$:
\nr{\begin{equation}
\phi_2(x(s)) =\left[\prod_{k=1}^N G(s + 2k)\right]\phi_2(x(s + 2(N + 1))) -  \frac{e^{z_0\cdot z(\zeta s)}}{\gamma_2(\zeta s)} + \sum_{n=1}^{N} \left[\prod_{k=0}^{n-1} G(s+2k)\right]\left[  \frac{e^{z_0\cdot z(s + 2n)}}{\gamma_2(s + nq)} - \frac{e^{z_0\cdot z(\zeta (s+2n))}}{\gamma_2(\zeta (s+2n))}\right]. 
\end{equation}}
\ngray{(correction of an index and adjustment by a boundary term $\left[\prod_{k=0}^{N} G(s-2k)\right] \frac{e^{z_0\cdot z(s - 2(N+1))}}{\gamma_2(s - 2(N+1))} $. This additional term tends to $0$ as $N\to+\infty$ because of the exponential term. Note that this term, at rank $N$, is part of the big sum at rank $N+1$.}
\nb{\begin{equation}
\phi_2(x(s)) =\left[\prod_{k=0}^N G(s - 2k)\right]\phi_2(x(s - 2(N + 1))) -  \frac{e^{z_0\cdot z(\zeta s)}}{\gamma_2(\zeta s)} +\left[\prod_{k=0}^{N} G(s-2k)\right] \frac{e^{z_0\cdot z(s - 2(N+1))}}{\gamma_2(s - 2(N+1))} 
\end{equation}
$$+ \sum_{n=1}^{N} \left[\prod_{k=0}^{n-1} G(s-2k)\right]\left[  \frac{e^{z_0\cdot z(s - 2n)}}{\gamma_2(s - 2n)} - \frac{e^{z_0\cdot z(\zeta (s-2n))}}{\gamma_2(\zeta (s-2n))}\right] $$}
%for all $n \in \N$ (where the empty product is equal to $1$ by convention).Remark by the way that if $n \leq N(s) - 1$, necessarily $\gamma_2(\zeta(s + nq)) \neq 0$. Indeed, if $\det(R) > 0$, the line $\{\gamma_2 = 0\}$ doesn't meet the down branch of the parabola "before" (i.e. at a superior abscissa 
The proof is then reduced to proving the following limit:
\begin{equation}\label{le_reste}
\left[\prod_{k=0}^n G(s - 2k)\right] \phi_2(x(s - 2(n + 1))) \underset{n\to+\infty}{\longrightarrow} 0.
\end{equation}
\nr{To justify this, note that the explicit formula of $G$ and the Lemma \ref{anbn} give constants $a, b, c, d, e, f, g, h$ such that 
\begin{equation}G(s + 2k) = \frac{(k^2 + ak + b)(k^2 + ek + f)}{(k^2 + ck + d)(k^2 + gk + f)}.
\end{equation}
We have in particular $$a = \frac{\beta +4r_{1} + \beta'r_{1}}{1 + r_{1}}, \quad c = \frac{\beta +4r_{1} + \beta'r_{1} + 4}{1 + r_{1}} $$
$$e = \frac{\beta' +r_{2}\beta'}{1 + r_{2}},\quad g = \frac{\beta' +r_{2}\beta + 4}{1 + r_{2}} $$
with the notation $\beta = - x(s) + y(s) + \mu_2$, $\beta' = -x(s) + y(s) - \mu_1$. Then,
%\color{gray}
%\begin{equation}
%a - c + e - g =- 4 \left(\frac{1}{1 + r_{1}\sqrt{\sigma_1/\sigma_2}}+ \frac{1}{1 +  r_{2}\sqrt{\sigma_2/\sigma_1} }\right)
%\end{equation} and $a-c + e-g < 0$ thanks to the Assumption \ref{Ass3}.
%\color{black}
\begin{equation}\label{exposant}
a - c + e - g =- 4 \left(\frac{1}{1 + r_{1}}+ \frac{1}{1 +  r_{2}}\right)
\end{equation} and $a-c + e-g < 0$ by Assumption (\ref{vectors}).}
\ngray{I noticed a simplification of $G$, since $x(s)$ and $y(s)$ have $s$ is factor. I also rectified an exponent.}
\nb{To justify this, note using formula~\eqref{eq:G} and Lemma~\ref{anbn} that:
\begin{equation}G(s - 2k) = \frac{(k+a)(k+ b)}{(k+ c)(k+ d)}.
\end{equation}
for some constants $a, b, c, d$ defined by: $$a =\frac{-s}{2} + \frac{r_{1}}{1 + r_{1}}, \quad b = 1 -\frac{s}{2} + \frac{\mu_2r_2 - \mu_1}{1 + r_{2}} $$
$$c = \frac{-s}{2} + \frac{1}{1 + r_{2}}, \quad d = 1 -\frac{s}{2} + \frac{\mu_2 - \mu_1r_1}{1 + r_{1}}.$$
By elementary considerations, the following asymptotic behavior holds: 
\begin{equation}\label{prodG}
\left[\prod_{k=0}^n G(s + 2k)\right]\underset{n\to\infty}{\sim}Cn^{a - c + b - d}
\end{equation} where $C$ is a real constant. Moreover,
\begin{equation}\label{exposant}
a + b - c - d = 2 - 2 \left(\frac{1}{1 + r_{1}}+ \frac{1}{1 +  r_{2}}\right)
\end{equation}
since $\mu_1 + \mu_2 = 1$.} \nr{The inequality $|\phi(s + (n+1)q)| \leq \phi(0) < +\infty$ concludes.}\nb{Then, the exponential decay in \eqref{eq:lap_decroit} for $\phi_2$, together with the polynomial rate of expression \eqref{prodG}, yields \eqref{explicite}. Note that inequality \eqref{eq:lap_decroit} is the only (and crucial) reason why we work under Assumption~\eqref{vectors}.} Equation~\eqref{explicite2} is obtained with symmetric arguments.
\end{proof}

\nb{\begin{rem} The exponent given by \eqref{prodG} is exactly the parameter $-2\gamma$ introduced in \cite{Dreyfus2025}, which determines the algebraic nature of the Laplace transforms for the same degenerate particle model in the recurrent case. Furthermore, the constants $\kappa_m = \kappa_m(\alpha)$ in \eqref{harm1} satisfy
%Since $$\gamma_i(a_{2n}, b_{2n}) \underset{|n|\to+\infty}{\sim} \gamma_i(a_{2n+1}, b_{2n+1}) \underset{|n|\to+\infty}{\sim} -2(1+r_i)n^2$$ for $i= 1, 2$, there exists $C_\pm > 0$ such that constants $\kappa_m = \kappa_m(\alpha)$ in \eqref{harm1} satisfy:
\begin{equation}
\kappa_m\underset{m\to\pm\infty}{\sim} C_\pm m^{-2\gamma - 2}.
\end{equation}
%where $$\beta=2 \left(\frac{1}{1 + r_{1}}+ \frac{1}{1 +  r_{2}}\right)$$
for some constant $C_\pm > 0$ where $-2\gamma - 2 < 0$ by (\ref{vectors}).
\end{rem}}
\ngray{I also added this little remark about the decrease of the coefficients, which makes a link with an upcoming article.}
\color{gray}
%\begin{lemma}\label{prepa}
%Let be $a, b, c, d$ some constants such that $(k^2 + ak + b)(k^2 + ck + d) \neq 0$ for all $k \in \N$. Then $$\prod_{k=0}^n\frac{k^2 + ak + b}{k^2 + ck + d}\underset{n\to\infty}{\sim}Cn^{a-c}$$ for some constant $C \neq 0$.
%\end{lemma}
%\begin{proof}
%Even if it means to consider the product $\prod_{k=k_0}^n$ with $k_0$ large enough, we can suppose that the quantity is positive so we can take the logarithm. We suppose $k_0 = 0$ without loss of generality. Then, $\log\left(\frac{k^2 + ak + b}{k^2 + ck + d}\right) = \frac{a-c}{k} + O(k^{-2})$. Using a equivalence summation theorem, we get $$ \sum_{k=0}^n\log\left(\frac{k^2 + ak + b}{k^2 + ck + d}\right) = (a-c)\log(n) + O(1)$$
%and the desired result by taking the exponential.
%\end{proof}
\color{black}

%\color{red}
%\begin{rem}
%With the hypothesis on the model (especially on $R_1, R_2$, see Assumption \ref{Ass3}), the formulae of Theorem \ref{harmoniques} hold for $z_0 = (0,0)$. Indeed, we can apply the same proof but we have ti justify differently the convergence $\left[\prod_{k=0}^n G(s + kq)\right] \phi_1(s + (n + 1)q) \underset{n\to+\infty}{\longrightarrow} 0.$ Here, the Lemma \ref{prepa} give the explicit speed of convergence : $\left[\prod_{k=0}^n G(s + kq)\right] \underset{n\to\infty}{\sim}Cn^{a-c + e-g}$ where 
%$$a - c + e - g =- 4 \left(\frac{1}{1 + r_{1}\sqrt{\sigma_1/\sigma_2}}+ \frac{1}{1 +  r_{2}\sqrt{\sigma_2/\sigma_1} }\right)$$
%so that with Assumption \ref{Ass3} we get $a-c + e-g < -1 < 0$ and the conclusion. 
%\end{rem}
%\color{black}

In \eqref{explicite} (resp.~\eqref{explicite2})  $\phi_2$ (resp. $\phi_1$) is not \nr{expressed in the variable}\nb{given as a function of} $x$ (resp. $y$) but \nr{in}\nb{of} $s$. We therefore establish the following corollary.

\begin{cor} 
\label{decroit_im}
The following expressions hold in the domains $Re(x)< x_{max}$ and $Re(y) < y_{max}$, respectively: 
\begin{equation} \label{phi2decroit}
\phi_2(x) =  \frac{-1}{\gamma_2(x, Y^-(x))}e^{z_0\cdot (x, Y^-(x))}\quad\quad\quad\quad\quad\quad\quad\quad\quad\quad\quad\quad\quad\quad\quad\quad\quad\quad\quad\quad\quad\quad\quad
\end{equation}
$$\quad\quad + \sum_{n=1}^{+\infty} \left[\prod_{k=1}^{n}  \frac{\frac{\gamma_1}{\gamma_2}(\psi_{2k-1}(x, Y^+(x))}{\frac{\gamma_1}{\gamma_2}(\psi_{2k}(x, Y^+(x)))}\right]\left[ \frac{e^{z_0\cdot \psi_{2n}(x, Y^+(x))}}{\gamma_2(\psi_{2n}(x, Y^+(x))} - \frac{e^{z_0\cdot \psi_{2n+1}(x, Y^+(x))}}{\gamma_2(\psi_{2n +1}(x, Y^+(x))}\right]$$
\begin{equation} \label{phi1decroit}
\phi_1(y) =  \frac{-1}{\gamma_1(X^+(y), y)}e^{z_0\cdot (X^+(y), y)}\quad\quad\quad\quad\quad\quad\quad\quad\quad\quad\quad\quad\quad\quad\quad\quad\quad\quad\quad\quad\quad\quad\quad
\end{equation}
$$\quad\quad\quad\quad\quad+ \sum_{n=1}^{+\infty} \left[\prod_{k=1}^{n}  \frac{\frac{\gamma_2}{\gamma_1}(\psi_{-2k+1}(X^+(y), y))}{\frac{\gamma_2}{\gamma_1}(\psi_{-2k}(X^+(y), y))}\right]\left[ \frac{e^{z_0\cdot \psi_{-2n}(X^+(y), y))}}{\gamma_1(\psi_{-2n}(X^+(y), y))} - \frac{e^{z_0\cdot \psi_{-2n+1}(X^+(y), y)}}{\gamma_1(\psi_{-2n +1}(X^+(y), y))}\right] $$
where
$$\psi_{2n}(a,b) = \left(-2n^2 +  2(a - b - \mu_2)n + a,  -2n^2 + 2(a - b + \mu_1) n + b\right)$$
and
$$\psi_{2n+1}(a,b) = \left( -2n^2 +  2(a - b - \mu_2)n + a,  -2(n+1)^2 + 2(a - b + \mu_1) (n+1) + b\right).$$
\end{cor} 

\begin{proof} 
By Lemma~\ref{anbn} and equalities $z(s) = (x(s), Y^+(x(s)) = (X^+(y(s)), y(s))$ for $s \in (s_{min}, s_{max})$, equations \eqref{phi2decroit} and \eqref{phi1decroit} hold on  the curve $\{(x,y)=(x(s), y(s)) : s \in \big( (\max(s_{min}, s^{**}), \min(s_{max}, s^*)) \big) \} $. By 
Corollary~\ref{coro:cont}, Laplace transforms $\phi_2(x)$ and $\phi_1(y)$ are  meromorphic on $Re(x) < x_{max}$ and $Re(y) < y_{max}$ respectively. Consequently, the explicit expressions \eqref{phi2decroit} and \eqref{phi1decroit} remain valid in these domains.
\end{proof} 

%for $s \in (s_{min}, s_{max})$ Indeed, these expressions have been established in Theorem \ref{harmoniques} for $x = x(s)$ and $y = y(s)$, $s \in (\max(s_{min}, s^{**}), \min(s_{max}, s^*))$. Here we have used Lemma~\ref{anbn} and equalities $z(s) = (x(s), Y^+(x(s)) = (X^+(y(s)), y(s))$ for $s \in (s_{min}, s_{max})$. Since expressions~\eqref{phi2decroit} and \eqref{phi1decroit} are meromorphic on $Re(x) < x_{max}$ and $Re(y) < y_{max}$, these equations hold in these domains by properties of meromorphic functions. 
%\textcolor{red}{Réponse à Sandro : " $Y^+$ a une coupure en xmin aussi non ?" Non, $xmin$ n'existe pas vu qu'on est sur la parabole : on n'a qu'un point de branchement.  "expliquer pourquoi on choisi parfois $Y^-$ et parfois $Y^+$ dans la formule (4.15)" ok!} 

\color{black}
\section{Laplace inverse and saddle-point method}\label{sec:6}
To avoid certain technical complications, we first derive the asymptotic behavior of the Green functions $g^{z_0}$ for $z_0 \neq 0$, and later address the case $z_0 = 0$ with additional arguments. 
 
\subsection{Inverse laplace theorem and saddle-point.}
%This section follows the steps of parts $4, 5, 6, 7$ of \cite{franceschi2023asymptotics} so we will refer often to results on this paper since the reasoning is the same. The main difference is the analytic computing for the saddle-point method. 
%The following Lemma allows to consider later the inverse Laplace theorem.
%\begin{lemma} Let $z_0 $ be an initial condition \textcolor{blue}{with both coordinates positive}. Then for $\epsilon > 0$, there is some $c, C >0$ satisfying 
%\begin{equation}\label{decroissance}\varphi_1(- \epsilon + iv) \leq  C e^{-c\sqrt{|v|}}, \ \ v \in \R.
%\end{equation}
%%\begin{equation}\label{decroissance}
%%\textcolor{red}{\varphi_1^{(0,0)}(- \epsilon + iv) \leq \frac{C}{v^{1/4}} , \ \ v \geq 1.} 
%%\end{equation}
%We have the same estimation for $\phi_2$.
%\end{lemma}
%\begin{proof}
%Thanks to the expression (\ref{Ymoins}), we have $Re(Y^-(-\epsilon + iv))\sim_{|v| \to +\infty} -\sqrt{|v|}$ up to a multiplicative constant. With the same asymptotic, remark that in (\ref{phi2decroit}) the first term dominates the sum. Then $\varphi_1(- \epsilon + iv) $ is equivalent to the first term of (\ref{Ymoins}) and we get the conclusion.
%\end{proof}

Let $z_0 \neq (0,0)$ be a starting point of the process. The inverse Laplace transform formula (see \cite[Theorem 24.3 and 24.4]{doetsch_introduction_1974} and \cite{brychkov_multidimensional_1992}) yields the following representation for $g^{z_0}(a, b)$: for $\epsilon>0$ sufficiently small,
\begin{align}\label{eq:6.1}
g^{z_0}(a,b)&= \frac{1}{(2\pi i)^2 } \int_{-\epsilon-i \infty}^{-\epsilon+ i \infty} \int_{-\epsilon - i \infty }^{-\epsilon+ i \infty} \phi^{z_0}(x,y) \exp(-a x -by) dxdy
\end{align}
where the convergence is in the sense of principal value. This can be justified by the functional equation~\eqref{Equation fonctionnelle} and the decay properties of the Laplace transforms established in Proposition~\ref{lap_dec_inv}.% Indeed, \eqref{Equation fonctionnelle} provides the following formula for $x,y \in -\epsilon + i\R $
%\begin{equation}\label{eq:5.2}
%    \varphi(x, y) =-\frac{\gamma_1(x, y)\varphi_1(y) + \gamma_2(x, y)\varphi_2(x) + e^{(x, y)\cdot z_0}}{\gamma(x, y)}.
%\end{equation}
%Furthermore, estimations $|\phi_1(-\epsilon + iy)| \leq C_1e^{-c_1\sqrt{|y|}}$ and $|\phi_2(-\epsilon + ix)| \leq C_2e^{-c_2\sqrt{|x|}}$ hold for some $c_1, c_2, C_1, C_2 > 0$. This can be proved using expressions~\eqref{cont2} and \eqref{cont} of Laplace transforms and applying Lemmas~\ref{eq:lap_decroit} and~\ref{lemXYY}.  \textbf{((((Pas si $a_0 = 0$ ou $b_0 = 0$, mais on peut s'en sortir en disant que $g$ cest la moyenne des $g$ avec des conditions initiales autour))))} Then, by elementary (but very technical) calculations, the convergence in principal value of~\eqref{eq:6.1} can be proved. We have chosen not to go into more details about the convergence.

%\color{gray}
%\begin{proof}[Sketch of proof of Laplace inverse formula]
%\begin{equation}
%    \int_{-\epsilon - iR}^{-\epsilon + iR}-\frac{\gamma_1(x, y)}{\gamma(x, y)}e^{-ax}dx = -i\left[\frac{\gamma_1(x, y)e^{-ax}}{a\gamma(x, y)}\right]_{x=-\epsilon - iR}^{x=-\epsilon + iR} + \int_{-\epsilon - iR}^{-\epsilon + iR}\frac{d}{dx}\left[\frac{\gamma_1(x, y)}{\gamma(x, y)}\right]\frac{e^{-ax}}{a}dx.
%\end{equation}
%Then,
%\begin{equation}
%    \int_{-\epsilon - iR'}^{-\epsilon + iR'}\int_{-\epsilon - iR}^{-\epsilon + iR}\frac{\gamma_1(x, y)\varphi_1(y)}{\gamma(x, y)}e^{-ax}dxdy = \int_{-\epsilon - iR}^{-\epsilon + iR}\varphi_1(y)\int_{-\epsilon - iR'}^{-\epsilon + iR'}\frac{\gamma_1(x, y)}{\gamma(x, y)}e^{-ax}dxdy.
%\end{equation}
%\end{proof}
\color{black}

%\textcolor{gray}{(We have indeed the convergence of $\int_{-\epsilon-i R}^{-\epsilon+ i R} \int_{-\epsilon - i T}^{-\epsilon+ i T} \phi(x,y) \exp(-a x -by)$ as $R$ and $T$ tend to $+\infty$ independently.} 

\begin{lemma}[From double to simple integrals]\label{3_integ}
Denote by $z_0=(a_0,b_0)$ the starting point of the process. Then, for any $(a,b)\in \mathbb{R}_+^2$ satisfying $a > 0$ or $b>0$,
$$g(a,b)= I_1(a,b)+ I_2(a,b) + I_3(a,b)$$ where 
$$ I_1(a,b)= \frac{1}{2 \pi i } \int_{-\epsilon - i \infty}^{- \epsilon + i \infty }
 \phi_2 (x) \gamma_2 (x, Y^{+} (x))
 \exp( -a x - b Y^{+}(x)) \frac{dx}{\partial_y\gamma(x, Y^{+}(x)) }, $$
$$ I_2(a,b)= \frac{1}{2 \pi i } \int_{-\epsilon - i \infty}^{- \epsilon + i \infty }
 \phi_1 (y) \gamma_1 ( X^{+} (y), y)  \exp( -a X^{+}(y)-by) \frac{dy}{\partial_x\gamma( X^{+}(y), y)}, $$
$$ I_3(a,b)= \frac{1}{2\pi i }    \int_{-\epsilon - i \infty}^{- \epsilon + i \infty }
 \exp(a_0 x + b_0 Y^{+}(x)) 
 \exp( -a x - b Y^{+}(x)) \frac{dx}{\partial_y\gamma(x, Y^{+}(x)) } \ \  \hbox{ if }b>b_0, $$
 $$ 
 I_3(a,b)= \frac{1}{2\pi i }  \int_{-\epsilon - i \infty}^{- \epsilon + i \infty }
 \exp(a_0 X^{+}(y)  + b_0 y ) 
 \exp( -a X^{+}(y) - b y ) \frac{dy}{\partial_x\gamma(X^{+}(y), y) } \ \  \hbox{ if }a>a_0.
 $$   
 \label{lem:I123}
\end{lemma} 
\begin{proof}
By the functional equation~\eqref{Equation fonctionnelle}, $\phi(x,y)$ can be decomposed as:
\begin{equation}\label{eq:5.2}
    \varphi(x, y) =-\frac{\gamma_1(x, y)\varphi_1(y)}{\gamma(x,y)}  \nr{+}\nb{-} \frac{\gamma_2(x, y)\varphi_2(x)}{\gamma(x,y)}  \nr{+}\nb{-} \frac{e^{(x, y)\cdot z_0}}{\gamma(x, y)}.
\end{equation}
Substituting this expression into the double integral \eqref{eq:6.1}, $g^{z_0}(a,b)$ is written as the sum of three double integrals. \nr{The complex residue theorem, applied to the integral with respect to the variable $y$ at the residue $y=Y^+(x)$, transforms the double integral
$$\frac{1}{(2\pi i)^2 } \int_{-\epsilon-i \infty}^{-\epsilon+ i \infty} \int_{-\epsilon - i \infty }^{-\epsilon+ i \infty} -\frac{\gamma_2(x, y)\varphi_2(x)}{\gamma(x, y)} \exp(-a x -by) dydx$$
into the simple one $I_1(a,b)$. 
Please refer to~\cite[Lemma $4.1$]{franceschi2023asymptotics} for further details. The proofs for $I_2, I_3$ follow from analogous arguments.}
\nb{Let us consider the first term, given by  
\begin{equation}\label{eq:1stterm}
\frac{-1}{(2\pi i)^2} \int_{-\epsilon - i\infty}^{-\epsilon + i\infty} \varphi_2(x) \int_{-\epsilon - i\infty}^{-\epsilon + i\infty} \frac{\gamma_2(x, y)}{\gamma(x, y)} e^{-a x - b y} \, dy \, dx.
\end{equation}
Let $C_R$ be the closed oriented contour defined by  
\[
C_R = \{-\epsilon + i t \mid t \in [-R, R]\} \cup \{-\epsilon + R e^{-i\theta} \mid \theta \in [-\pi/2, \pi/2]\}.
\]  
By applying the residue theorem along the contour $C_R$ and considering the asymptotics as $R \to +\infty$ (see~\cite[Lemma 4.1]{franceschi2023asymptotics} for more details), we obtain the identity  
\[
\int_{-\epsilon - i\infty}^{-\epsilon + i\infty} \frac{\gamma_2(x, y)}{\gamma(x, y)} e^{-a x - b y} \, dy = \frac{\gamma_2(x, Y^{+}(x)) }{\partial_y \gamma(x, Y^{+}(x))}e^{-a x - b Y^{+}(x)},
\]  
so that expression~\eqref{eq:1stterm} equals $I_1(a, b)$.  
The remaining terms are handled analogously.}

\end{proof}
To find the asymptotics of these integrals as $a, b \to +\infty$, we use the saddle-point method. For any $\alpha \in [0, \pi/2]$, let $(x(\alpha), y(\alpha))$ be defined as
\begin{equation}\label{xalpha}
(x(\alpha), y(\alpha)) = \argmax_{(x, y) \in \mathcal{P}} \left(\cos(\alpha)x + \sin(\alpha)y\right),
\end{equation}
%where $e_\alpha = (\cos(\alpha), \sin(\alpha))$ and fix 
%\begin{equation}(x(\alpha), y(\alpha)) = (x(s(\alpha)), y(s(\alpha))).\end{equation}
see Figure~\ref{colcol}. \nr{We denote by $s(\alpha) \in \R$ the real number satisfying $(x(\alpha), y(\alpha)) = (x(s(\alpha)), y(s(\alpha)))$.}\nb{For $\alpha \in [0,\pi/2]$, we define the real number $\mathfrak{s}(\alpha) \in \R$ by $$\mathfrak{s}(\alpha) = \argmax_{s \in \R} (\cos(\alpha)x(s) + \sin(\alpha)y(s)).$$ Note that $(x(\alpha), y(\alpha)) = (x(\mathfrak{s}(\alpha)), y(\mathfrak{s}(\alpha)))$, using notation \eqref{parametrisation}.} By studying the variations of the function $s \longmapsto x(s)\cos(\alpha) +y(s)\sin(\alpha)$, we prove that  
\begin{equation}\label{mathfraks}
    \mathfrak{s}\colon\left\{
    \begin{aligned}
        &[0, \pi/2]\longrightarrow [s_{min}, s_{max}]\\
        &\quad\alpha\quad\longmapsto \nr{s(\alpha)=}\frac{\mu_2 - \tan(\alpha)\mu_1}{1 +\tan(\alpha)}
    \end{aligned}
    \right.\quad\quad \nb{(\textnormal{with}\quad\mathfrak{s}(\pi/2) = -\mu_1 = s_{min})}
\end{equation}
is a $C^\infty$ diffeomorphism, and
\begin{equation}\label{s_alpha}
    \mathfrak{s}^{-1}\colon\left\{
    \begin{aligned}
         &[s_{min}, s_{max}]\longrightarrow [0, \pi/2]\\
        &\quad\quad s \quad\quad\longmapsto \arctan\left(\frac{\mu_2 - s}{s + \mu_1}\right)
    \end{aligned}
    \right. \quad\quad \nb{(\textnormal{with}\quad\mathfrak{s}^{-1}(-\mu_1) = \pi/2).}
\end{equation} 
\nr{With notations from} \nb{Using the definitions of $\alpha^*, \alpha^{**}, x^*$ and $y^{**}$ given by} \eqref{alpha*}, \eqref{alpha**} and \eqref{eq:defx*y**}, if $x^*$ (resp. $y^*$) is a pole of $\phi_2$ (resp. $\phi_1$), then $x(\alpha^*)=x^*$ (resp. $y(\alpha^{**})=y^{**}$). Since $s^{**} < 0 < s^*$ (see Notation~\ref{not:2}), then the monotonicity of~\eqref{s_alpha} implies that $0 \leq \alpha^* < \alpha_\mu < \alpha^{**} \leq \pi/2$, where $\alpha_\mu = \arctan(\mu_2/\mu_1) \in (0, \pi/2)$ is the angle of the drift.  % and $\phi_2$ (resp. $\phi_1$) has a pole  if and only if $\alpha^* > 0$ (resp. $\alpha^{**} < \pi/2$). 
We \nr{use the same notation as in}\nb{follow the notation of}~\cite{franceschi2023asymptotics} and define:
 \begin{equation}
     \label{fff}
 F(x,\alpha)=  - \cos (\alpha) x - \sin(\alpha) Y^{+}(x) + \cos (\alpha) x(\alpha) + \sin (\alpha)  y(\alpha) 
 \end{equation}
\begin{equation}
G(y, \alpha) =   -\cos(\alpha) X^{+}(y)- \sin(\alpha) y 
       + \cos (\alpha) x(\alpha) + \sin (\alpha) y(\alpha) .
\end{equation}
By construction, the equations $\nr{F'_x}\nb{\partial_xF}(x(\alpha),\alpha) = 0$ and $\nr{G'_y}\nb{\partial_yG}(y(\alpha),\alpha) = 0$ hold.
Then, \nb{by differentiating equations \eqref{fff} and $\gamma(x,Y^+(x)) = 0$, we get} for any $\alpha \in (0, \pi/2]$, \begin{equation} \label{all}
(Y^{+}(x))'\Bigm|_{x=x(\alpha)} = - \frac{ \partial_x\gamma(x(\alpha), y(\alpha))  }{ \partial_y\gamma(x(\alpha), y(\alpha))} = -\frac{\cos (\alpha)}{\sin (\alpha)},\quad (Y^{+}(x))''\Bigm|_{x=x(\alpha)} = -\frac{(1 + {\rm ctan}\,(\alpha))^2 }{ \partial_y\gamma(x(\alpha), y(\alpha))}.
\end{equation} 
Therefore,
\begin{equation}
\label{fzfz}
\nr{F''_{x}}\nb{\partial^2_{xx}F}(x(\alpha), \alpha) = \frac{(\sin(\alpha) + \cos(\alpha))^2 }{ \partial_y\gamma(x(\alpha), y(\alpha) ) \sin (\alpha) }>0,\  \  \ \alpha \in (0, \pi/2].
\end{equation} 
Similarly,
 $$\nr{G''_{y}}\nb{\partial^2_{yy}G}(y(\alpha), \alpha)=\frac{(\sin(\alpha) + \cos(\alpha))^2 }{ \partial_x\gamma(x(\alpha), y(\alpha) ) \cos (\alpha) }>0,\  \  \ \alpha \in [0, \pi/2).$$  

\subsection{Contour of steepest descent} \label{def_contour}
Let $\alpha_0 \in (0, \pi/2]$. The key \nr{point}\nb{idea} of the saddle-point method is to use the parameterised Morse lemma. \nb{Since $\partial^2_{xx}F(x(\alpha), \alpha) > 0$, Lemma A.1 from \cite{franceschi2023asymptotics}}\nr{\cite[Lemma A.1]{franceschi2023asymptotics} to obtain a path} \nb{yields some $\epsilon > 0, \eta >0$ and a family of smooth paths} $\Gamma_{x,\alpha} =\{x(it, \alpha) \mid  t \in [-\epsilon, \epsilon] \}, |\alpha -\alpha_0|<\eta$ \nr{and some constants $\epsilon > 0$ so}\nb{such} that 
\begin{equation}
\forall t \in [-\epsilon, \epsilon],\ \ \ \ \ F(x(it, \alpha), \alpha)= -t^2.
\end{equation}
\nr{This is possible because $F''_x(x(\alpha), \alpha) > 0$.} For further details on the construction, please refer to~\cite[Appendix A]{franceschi2023asymptotics}. Define \nr{$x^+_\alpha = x(i\epsilon, \alpha)$ and $x^-_\alpha = x(-i\epsilon, \alpha)$.} 
\nb{\begin{equation}\label{eq:xalpha+}
x^+_\alpha = x(i\epsilon, \alpha),\quad  x^-_\alpha = x(-i\epsilon, \alpha).
\end{equation}} In particular,
  \begin{equation}
      \label{fe}
   F(x^+_\alpha, \alpha)=-\epsilon^2, \  \  F(x^-_\alpha, \alpha)=-\epsilon^2.
   \end{equation}
Furthermore, $Im(x_\alpha^+) > 0$ and $Im(x_\alpha^-) < 0$ (see Figure \ref{Changing path for $I_2$} and construction in \cite{franceschi2023asymptotics}). The same construction holds for $\Gamma_{y,\alpha} =\{y(it, \alpha) \mid t \in [-\epsilon, \epsilon] \}$ for $G$ and $\alpha_0 \in [0, \pi/2)$. These paths satisfy $$\Gamma_{x,\alpha}=\underrightarrow{\overleftarrow{X^{+}(\Gamma_{y,\alpha})}} \quad \textnormal{and}\quad \Gamma_{y,\alpha}=\underrightarrow{\overleftarrow{Y^{+}(\Gamma_{x,\alpha}) }}, \quad 0 < \alpha < \pi/2.$$ \nr{We define $\Gamma_{x,0}=\underrightarrow{\overleftarrow{X^{+}(\Gamma_{y,0})}}$ $\big($resp. $\Gamma_{y,\pi/2}=\underrightarrow{\overleftarrow{Y^{+}(\Gamma_{x,\pi/2}) }}$ $\big)$.}\ngray{This is not necessary, so I deleted this because of the amount of notation.} The arrows above and below the paths indicate reversed orientations, this notation is taken from \cite[Chapter 5.3]{FIM17}. %We let the interested reader about precise construction to look at \cite{franceschi2023asymptotics}.

\subsection{Shift of the integration contours and contribution of the poles} 
\label{sec:shift}
%In the same way of \cite{franceschi2023asymptotics}, 
We now apply the saddle-point method. To do this, we shift the integration contours of $I_1$, $I_2$ and $I_3$ to contours passing through the saddle-point and following the steepest descent contours $\Gamma_{x,\alpha}$ and $\Gamma_{y,\alpha}$.
We define $T_{x, \alpha}= S^-_{x, \alpha}+ \Gamma_{x, \alpha} + S^+_{x, \alpha}$  and 
     $T_{y, \alpha}= S^-_{y, \alpha}+ \Gamma_{y, \alpha} + S^+_{y, \alpha}$  for $\alpha \in [0, \pi/2]$ where
$$S^+_{x, \alpha}=\{x^+_{\alpha} +it \mid t \geq 0 \},\ \ 
 S^-_{x, \alpha}=\{x^-_{\alpha} - it \mid t \geq 0 \},$$ 
        $$S^+_{y, \alpha}=\{y^+_{\alpha} +it \mid t \geq 0 \}, \ \  
 S^-_{y, \alpha}=\{y^-_{\alpha} - it \mid t \geq 0 \}.$$

\begin{figure}[hbtp]
\centering
\includegraphics[scale=2]{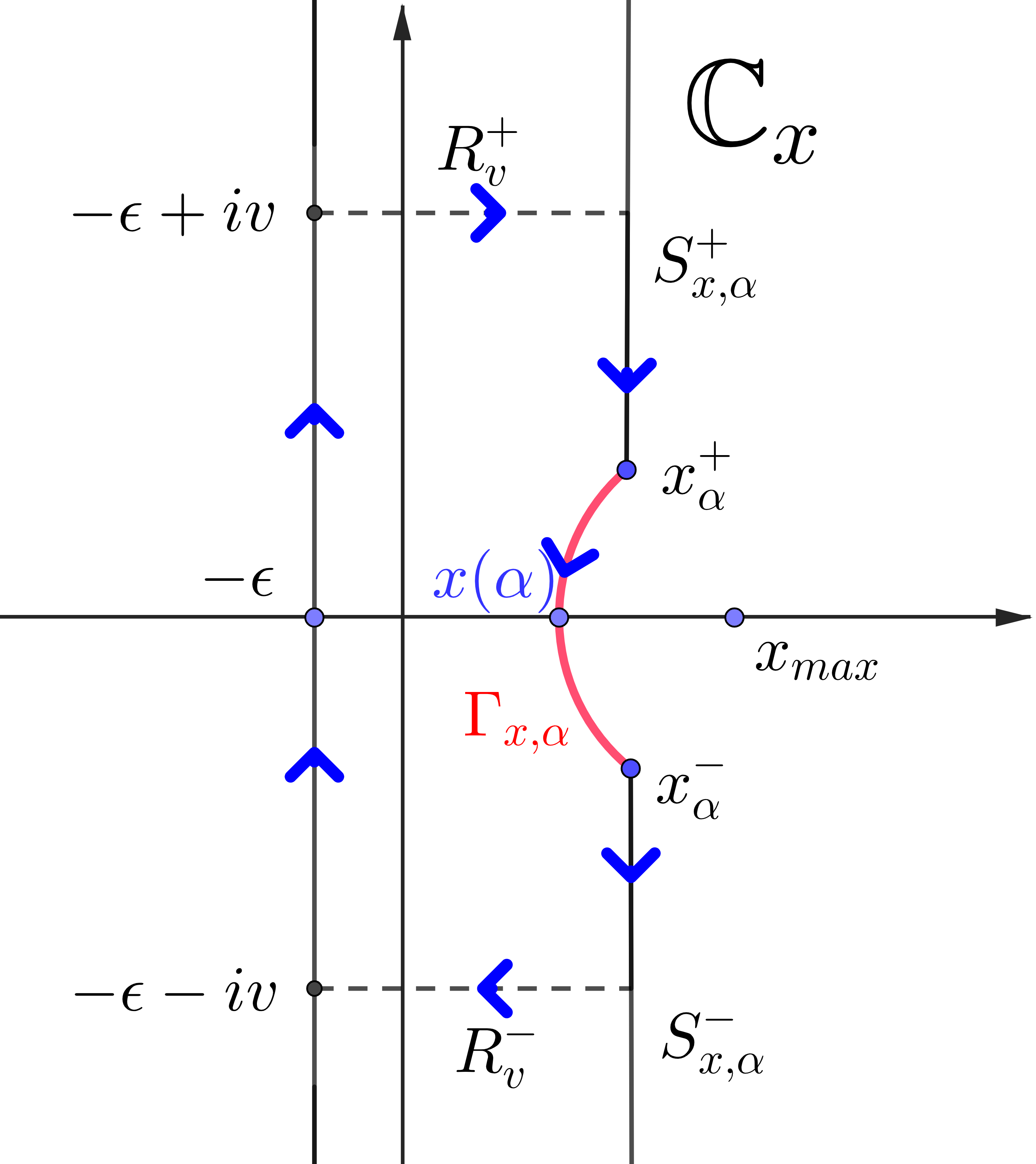}
\vspace{0cm}
\caption{Changing path for $I_2$. Here, $x(\alpha) < x^*$.}
\label{Changing path for $I_2$}
\end{figure}

%\textcolor{purple}{\it On pourrait ici préciser un peu mieux $\alpha^*, \alpha^{**}$ par Lemma \ref{pole}.  }
%We denote for convenience $\alpha^*=0$ if $x^*$ does not exist and $\alpha^{**}=\pi/2$ if $y^*$  does not exist.
%By the residue theorem, we can set the following result.
\begin{lemma}[\nr{Changing paths and pole}\nb{Contour deformation and contribution of the pole}] \label{residus}
Let $\alpha \in [0, \pi/2]\backslash\{\alpha^*, \alpha^{**}\}$ and $z_0 \neq (0,0)$ be the initial condition of the process. Then for any $a,b>0$,
\begin{equation}\label{I1}
    I_1(a,b)=\frac{ \big({\rm -res}_{x=x^*} \phi_2(x) \big ) \gamma_2(x^*, y^*)}{ \partial_y\gamma(x^*, y^*)} \exp(-ax^* -b y^* ) {\bf 1}_{ \alpha <\alpha^*}
\end{equation}
$$\quad\quad\quad+\ \frac{1}{ 2 \pi i } \int\limits_{T_{x, \alpha}}\frac{\phi_2(x) \gamma_2(x, Y^+(x))}{\partial_y\gamma(x, Y^+(x))} \exp(-ax -b Y^+(x))dx, $$
\begin{equation} I_2(a, b) = \frac{ \big({\rm -res}_{y=y^{**}} \phi_1(y) \big) \gamma_1(x^{**}, y^{**})}{ \partial_x\gamma(x^{**},y^{**})} \exp(-ax^{**} -by^{**}) {\bf 1}_{ \alpha >\alpha^{**}}
\end{equation}
$$\quad\quad\quad+ \frac{1}{2\pi i } \int\limits_{T_{y, \alpha}} 
\frac{\phi_1(y) \gamma_1(X^+(y), y)}{\partial_x\gamma(X^+(y), y)} \exp(-aX^+(y) -b y)dy,$$
\begin{equation} I_3(a,b)= \frac{1}{2\pi i }  \int\limits_{T_{x, \alpha}}   \exp((a_0-a) x + (b_0-b) Y^{+}(x)) 
  \frac{dx}{\partial_y\gamma(x, Y^{+}(x)) } \ \  \hbox{ if }b>b_0, 
  \end{equation}
\begin{equation}
 I_3(a,b)= \frac{1}{2\pi i }\int\limits_{ T_{y, \alpha}}
 \exp((a_0-a) X^{+}(y)  + (b_0-b) y ) 
 \frac{dy}{\partial_x\gamma(X^{+}(y), y) } \ \  \hbox{ if }a>a_0.
\end{equation}
\end{lemma} 

%One may remark that we have $\gamma_2(x^*, y^*){\rm res}_{x^*}\phi_2<0$ and $\gamma_1(x^{**}, y^{**}){\rm res}_{y^{**}}\phi_1<0$.

\begin{proof} 
The shift of the path is illustrated in Figure~\ref{Changing path for $I_2$} and is the same as in \cite[Lemma 6.1]{franceschi2023asymptotics}. The proof of~\eqref{I1} is a direct consequence of the residue theorem, provided that\nr{we show that} the integrals over the horizontal contours $R^+_v$ and $R^-_v$ tend to $0$ as $v$ tends to $+\infty.$ Then, it remains to prove that for any sufficiently small $\eta>0$,
 $$ \sup_{ u \in [X^+(y_{max})-\eta, x^{max}+ \eta]}
   \Big| \frac{ \phi_2(u+iv) \gamma_2(u+iv, Y^+(u+iv))}{ \gamma_y'(u+iv, Y^+(u+iv))} \exp(-a(u+iv) - b Y^+(u+iv)) \Big| \to 0, \ \  \hbox{ as } v\to \infty.$$
By the functional equation (\ref{Equation fonctionnelle}) and continuation formula \eqref{cont}, the term inside the supremum is equal to 
$$ \left| \frac{ \left(\gamma_1(u+iv, Y^-(u+iv))\phi_1(Y^-(u+iv)) + e^{a_0(u+iv) + b_0Y^-(u+iv)} \right)\gamma_2(u+iv, Y^+(u+iv))}{\gamma_2(u+iv, Y^-(u+iv)) \gamma_y'(u+iv, Y^+(u+iv))} \right|$$
$$\times |\exp(-a(u+iv) - b Y^+(u+iv))|. $$
By \eqref{Y-}, $Re(Y^\pm(u+iv))$ grows like $\pm\sqrt{|v|}$ uniformly in $u \in  [X^+(y_{max})-\eta, x^{max}+ \eta]$ as $|v|\to +\infty$. 
Furthermore, %we have by (\ref{Ymoins}) that
%\textcolor{gray}{\begin{equation}
%\partial_y\gamma(x, Y^{+}(x))= \sqrt{-2 
%     (\mu_2 \sqrt{\sigma_1\sigma_2} + \mu_1 \sigma_{2}) x + \mu_2^2}
%  \end{equation}} 
$\gamma_2(u+iv, Y^\pm(u+iv))$ grows linearly in $v$ uniformly in $u \in  [X^+(y_{max})-\eta, x^{max}+ \eta]$ as $v \to +\infty$ by Assumption~\eqref{vectors}. The same asymptotics hold for $\gamma_1(u+iv, Y^-(u+iv))$. Moreover, $\partial_y\gamma(u+iv, Y^{+}(u+iv))= \sqrt{-2(u+iv) + \mu_2^2}$, so this expression grows with rate $\sqrt{v}$, uniformly in $u \in  [X^+(y_{max})-\eta, x^{max}+ \eta]$. Considering the exponential decay of $\phi_1$ (see Lemma \ref{lap_decroit}) we get the conclusion for $I_1$. Formulae for $I_2$ and $I_3$ are obtained similarly. 
%so this does not vanish as $|Im(x)| \to +\infty$.   
\end{proof} 

\subsection{Negligibility of some integrals}
For any pair $(a,b) \in {\mathbb{R}}_+^2$  let
$\alpha(a,b)$ be the angle in $[0,\pi/2]$ such that $\cos (\alpha) = \frac{ a }{ \sqrt{a^2+b^2}}$  and  $\sin(\alpha) = \frac{ b }{\sqrt{a^2+ b^2}}$.
We now aim to evaluate the asymptotics of the integrals over $T^\pm_{x,\alpha}$ and $T^\pm_{y,\alpha}$ in Lemma~\ref{residus} as $\sqrt{a^2 + b^2} \to +\infty$ and $\alpha(a,b) \to \alpha_0$ for some $\alpha_0 \in [0, \pi/2].$ In the next lemma, we establish exponential bounds for the integrals over the vertical contours $S^{\pm}_{x,\alpha}$, $S^{\pm}_{y,\alpha}$. These bounds imply that the main contribution to the above asymptotics comes from the integrals over the steepest descent contours $\Gamma_{x,\alpha}$, $\Gamma_{y,\alpha}$, while those over $S^\pm_{x,\alpha}$ and $S^\pm_{y,\alpha}$ turn out to be negligible.

\begin{lemma}[Negligibility of the integrals over $S_{x,\alpha}^\pm$ and $S_{y,\alpha}^\pm$] 
\label{pp}
Suppose $z_0 \neq (0,0).$ Let $K$ be a compact neighbourhood of $z_0$ in the quadrant satisfying $d((0,0), K) > 0$. Let $\alpha_0 \in [0, \pi/2]$. Then, for sufficiently small $\eta>0$, there exist \nr{some }constants $r_0>0$ and $D_{\alpha_0}>0$ such that for any $z \in K$ and any pair $(a,b)$ satisfying $\sqrt{a^2+b^2}>r_0$ and $|\alpha(a,b) - \alpha_0|<\eta$, the following inequalities hold
\begin{equation}\label{truc1}
\Big|\int\limits_{S^{\pm}_{x, \alpha}}
\frac{\phi_2^z(x) \gamma_2(x, Y^+(x))}{\partial_y\gamma(x, Y^+(x))} \exp\big(-ax -b Y^+(x)\big)dx \Big| \leq D_{\alpha_0} \exp\Big(-a x(\alpha) -b y(\alpha) - \epsilon \sqrt{a^2+b^2}  \Big),
\end{equation}
\begin{equation}
    \label{truc2}
\Big|\int\limits_{S^{\pm}_{y, \alpha}}
\frac{\phi_1^z(y) \gamma_1(X^+(y), y)}{\partial_x\gamma(X^+(y), y)} \exp\big(-aX^+(y) -b y \big)dy \Big| \leq D_{\alpha_0} \exp\Big(-a x(\alpha) -b y(\alpha) - \epsilon^2\sqrt{a^2+b^2}\Big).
\end{equation}

 If $b>b_0$,
\begin{equation}
\label{truc3}
\Big| \int\limits_{S^{\pm} _{x, \alpha}}  
\exp((a_0-a) x + (b_0-b) Y^{+}(x))
  \frac{dx}{\partial_y\gamma(x, Y^{+}(x)) } \Big| \leq  
  \frac{D_{\alpha_0}}{b-b_0} \exp\Big(-a x(\alpha) -b y(\alpha) -\epsilon^2\sqrt{a^2+(b-b_0)^2}\Big).
\end{equation}  

 If $a>a_0$,
\begin{equation}
\label{truct4}
\Big| \int\limits_{S^{\pm} _{y, \alpha}}  
\exp((a_0-a) X^+(y) + (b_0-b) y )
  \frac{dy}{\partial_x\gamma(X^+(y), y) } \Big| \leq  
 \frac{D_{\alpha_0}}{a-a_0}\exp\Big(-a x(\alpha) -b y(\alpha) -\epsilon^2\sqrt{(a-a_0)^2+b^2}\Big).
\end{equation}

\end{lemma}

\begin{proof}
We start by showing (\ref{truc1}). Using notations (\ref{fff}) \nb{and \eqref{eq:xalpha+}}, this inequality can be rewritten as
\begin{equation}
\label{lefths}
\Big|\int\limits_{v>0}   
\frac{\phi_2^z(x_\alpha^+ +iv) \gamma_2(x_\alpha^+ + iv, Y^+(x_\alpha^+ + iv))}{\partial_y\gamma(x_\alpha^++iv, Y^+(x_\alpha^+ + iv))} \exp\left(-aiv-b\left(Y^+(x_\alpha^+ + iv) - Y^+(x_\alpha^+)\right)\right)dx\Big|\leq D_{\alpha_0}
\end{equation}
where $\alpha=\alpha(a,b)$.

Suppose first that $\alpha_0 > 0$. Let $\alpha > 0$ and $0<\eta < \alpha_0/2$. Since $\partial_y\gamma(x_\alpha^++iv, Y^+(x_\alpha^+ + iv)) = \sqrt{-2(x_\alpha^++iv) + \mu_2^2}$, this expression does not vanish and grows \nr{with}\nb{at} rate $\sqrt{|v|}$ as $v \to +\infty$, uniformly in $\alpha$ with $|\alpha - \alpha_0|< \eta.$ Similarly, $\gamma_2(x_\alpha^+ + iv, Y^+(x_\alpha^+ + iv))$ grows with speed $|v|$ as $v \to +\infty$, uniformly in $\alpha$, $|\alpha - \alpha_0|< \eta$. Then we have, for all $v\geq0,$
$$\sup_{|\alpha - \alpha_0|< \eta} \frac{\gamma_2(x_\alpha^+ + iv, Y^+(x_\alpha^+ + iv))}{\partial_y\gamma(x_\alpha^++iv, Y^+(x_\alpha^+ + iv))} \leq C_\eta(1 + \sqrt{v})$$
for some \nb{constant} $C_\eta > 0$. If $|\alpha - \alpha_0|< \eta$, then by \eqref{Ymoins}, there exists a constant $C'_\eta > 0$ such that
\begin{equation}\label{rere}
Re\left(\sqrt{a^2+b^2}\big(F(x_\alpha^+ +iv,\alpha) -F(x_\alpha^+, \alpha)\big)\right) = b \left(Re(Y^+(x_\alpha^+ + iv)) - Re(Y^+(x_\alpha^+))\right) \geq C'_{\eta}b\sqrt{v}
\end{equation}
for any $v\geq 1$.
Furthermore, using the continuation formula (\ref{cont}), the estimates \eqref{eq:lap_decroit} and the continuity of $\phi_2^{z_0}(0)$ in $z_0$ (see \eqref{phi2decroit}), there exists a constant $D$ such that $|\phi_2^z(x_\alpha^+ +iv)| \leq D$ for all $v\geq0$, $z \in K$ and $|\alpha - \alpha_0|< \eta$. Then, the left-hand side of (\ref{lefths}) is bounded by 
$$DC_\eta C'_\eta\left(2 + \int\limits_{v>1}  (1 + \sqrt{v})e^{-b\sqrt{v}}dv\right) = DC_\eta C'_\eta\left(2+\frac{1}{b^2} + \frac{4}{b^3}\right) \leq D_{\alpha_0}$$
for some \nb{constant} $D_{\alpha_0} >0$ since $b \to +\infty$ (because $\alpha_0 > 0$).
This inequality implies (\ref{truc1}).

Now suppose that $\alpha_0 = 0$. We no longer use estimate (\ref{rere}), as it would produce terms of order $\frac{1}{b}$, and here $b$ may be close to zero. Let $z = (a_1, b_1) \in K$. We write continuation formula~\eqref{cont} for $\phi_2^z(x_\alpha^+ +iv)$, which splits into two terms:
\begin{equation}
\phi_2^z(x_\alpha^+ +iv) = -\frac{\gamma_1(x_\alpha^++iv, Y^-(x_\alpha^++iv))\phi_1^z(Y^-(x_\alpha^++iv))}{\gamma_2(x_\alpha^++iv, Y^-(u+iv))} - \frac{e^{a_1(x_\alpha^++iv) + b_1Y^-(x_\alpha^++iv)}}{\gamma_2(x_\alpha^++iv, Y^-(u+iv))},
\end{equation}
and  we substitute into the right-hand side of~\eqref{lefths}. Then, the integral \eqref{lefths} can be written as the sum of two terms\nr{ respectively}. For the first term, note that there are some constants, $c, C_0 >0$ independant  on $\alpha \in [0, \eta]$ and $z\in K$, such that
\begin{equation}\label{1er}
\left|\frac{ \gamma_2(x_\alpha^+ + iv, Y^+(x_\alpha^+ + iv))\gamma_1(x_\alpha^++iv, Y^-(x_\alpha^++iv))\phi_1^z(Y^-(x_\alpha^++iv))}{\gamma_2(x_\alpha^++iv, Y^-(u+iv))\partial_y\gamma(x_\alpha^++iv, Y^+(x_\alpha^+ + iv))} \exp\left(-aiv - b \left(Y^+(x_\alpha^+ + iv) - Y^+(x_\alpha^+)\right)\right)\right|
\end{equation}
$$\leq C_0(\sqrt{v}+1)e^{(a_1 + b_1)Re(Y^-(x_\alpha + iv))}\phi_1^{z}(0) \leq C_0(\sqrt{v}+1)e^{(a_1 + b_1)c\sqrt{v}}\phi_1^{z}(0).$$
for any $v \geq 0$. We recall that function $z \longmapsto \phi_1^{z}(0)$ is continuous, and therefore locally bounded. The integral of $(\ref{1er})$ over $v > 0$ can then be bounded by a positive constant which is (locally) independent of $z$ and of $0\leq \alpha \leq \eta$. The second term is given by
\begin{equation}\label{2eme}
\int_0^{+\infty}\frac{ \gamma_2(x_\alpha^+ + iv, Y^+(x_\alpha^+ + iv)) e^{a_1(x_\alpha^++iv) + b_1Y^-(x_\alpha^++iv)}}{\gamma_2(x_\alpha^++iv, Y^-(u+iv))\partial_y\gamma(x_\alpha^++iv, Y^+(x_\alpha^+ + iv))} \exp\left(-aiv - b \left(Y^+(x_\alpha^+ + iv) - Y^+(x_\alpha^+)\right)\right)dv.
\end{equation}
%Since we want bounds (locally) uniformly to the initial condition, we cannot just suppose $b_0 > 0$ and conclude the same way.
%We may not have $b_1 > 0$ so we cannot conclude the same way 
Note that if $b_1 = 0$, the quotient in the integrand is \nb{of order} $O(1/\sqrt{v}$) as $v \rightarrow+\infty$.
Moreover, it suffices to bound the integral over $(v_0, +\infty)$ \nr{with}\nb{for} some $v_0 > 0$, since the integrand is \nr{bounded uniformly in}\nb{uniformly bounded with respect to} $\alpha \in [0,\eta]$ and $z \in K$. By integration by parts, the integral over $(v_0, +\infty)$ equals
 \begin{equation}
\label{ipp}
\frac{\gamma_2(x_\alpha^+ +iv_0, Y^+(x_\alpha^++iv_0))e^{a_1(x_\alpha^++iv_0) + b_1Y^-(x_\alpha^++iv_0)}\exp\Big( -aiv_0 - b (Y^+(x_\alpha^+ +iv_0) -Y^+ (x_\alpha)) \Big)}{\gamma_2(x_\alpha^++iv, Y^-(x_\alpha^++iv_0))\partial_y\gamma (x_\alpha^+ +iv_0, Y^+(x_\alpha^++iv_0)) (-ai-b
\frac{d}{dv}(Y^+(x_\alpha^++iv))_{v=v_0})}
\end{equation}
$$
 - \int_{v_0}^\infty \frac{d}{dv}\left(
\frac{\gamma_2(x_\alpha^+ +iv, Y^+(x_\alpha^++iv))e^{a_1(x_\alpha^++iv) + b_1Y^-(x_\alpha^++iv)}}{\gamma_2(x_\alpha^++iv, Y^-(x_\alpha^++iv))\partial_y\gamma (x_\alpha^+ +iv, Y^+(x_\alpha^++iv)) (-ai-b
\frac{d}{dv}(Y^+(x_\alpha^++iv))}\right)        
$$   
$$\times \exp\big(  -aiv - b (Y^+(x_\alpha^+ +iv) - Y^+ (x_\alpha^+))  \big) dv.$$
Furthermore, $\frac{d}{dv}(Y^+(x_\alpha^++iv)) = i\left(1 - \frac{1}{\sqrt{\mu_2^2 - 2(x_\alpha + iv)}}\right)$ and $Re\left(1 - \frac{1}{\sqrt{\mu_2^2 - 2(x_\alpha + iv)}}\right) \geq 1/2$ for all $v \geq v_0$ with $v_0$ large enough and $0<\alpha<\eta$. With some calculations, the integrand of \eqref{ipp} is \nb{of order} $O(1/v^{3/2})$ as $v\rightarrow+\infty$. Hence, the integral in~\eqref{ipp} is bounded by a positive constant independent of $\alpha$ and of $z \in K$. This establishes the bound in~\eqref{truc1}. Inequalities~\eqref{truc2}, \eqref{truc3} and \eqref{truct4} are obtained similarly.
% \textcolor{gray}{To prove (\ref{truc2}), we could use a similar (and technical) approach as Lemma 7.3 of \cite{franceschi2023asymptotics}. This would use the study of the decreasing of Laplace transforms $\phi_1$ and $\phi_2$, with the direct expression (\ref{phi2decroit}) or with a Boundary Value Problem approach as in Appendix $C$ of \cite{franceschi2023asymptotics}. We chose not to detail this (very) technical point.}
\end{proof}
\color{black}

\section{Proof of Theorem \ref{thm1}}\label{sec:proof}\label{sec:7}
In Section \ref{sub:6.1}, we establish the asymptotics \nr{of}\nb{stated in} Theorem~\ref{thm1}. In Section~\ref{sub:6.2}, we show that \nr{$h_\alpha(z_0) \neq 0$ for all $\alpha \in (\alpha^*, \alpha^{**})$ and$z_0 \in \R_+^2$}\nb{all the constants $h_\alpha(z_0)$  appearing in the asymptotics of Theorem~\ref{thm1} are non-zero,} which completes the proof of the theorem.

\subsection{Asymptotics in Theorem~\ref{thm1}}\label{sub:6.1}
We now have the tools to derive the asymptotics stated in Theorem \ref{thm1} where $h_{\alpha_0}(z_0)$ is given by~(\ref{harm1}), $h_{\alpha^*}(z_0)$ by~(\ref{harm2}) (with the symmetric formula for $h_{\alpha^{**}}(z_0)$), and
\begin{equation}\label{eq:c*}
c^* = \frac{\gamma_2(x^*,y^*)}{\partial_y\gamma(x^*, y^*)}\frac{x'(s^*)}{\gamma_2(x'(\zeta s^*),y'(\zeta s^*))},\quad c^{**} = \frac{\gamma_1(x^{**},y^{**})}{\partial_x\gamma(x^{**}, y^{**})}\frac{y'(s^{**})}{\gamma_1(x(\eta s^{**}),y'(\eta s^{**}))}
\end{equation}
\nb{where $x'(s)$ and $y'(s)$ are the derivatives of $x(s)$ and $y(s)$} (see \eqref{parametrisation} for the definition of $x(s)$ and $y(s)$, \eqref{s*} for $s^*$ and $ s^{**}$, \eqref{eq:defx*y**} and \eqref{eq:defx**y*} for $x^*, x^{**}, y^*$ and $y^{**},$ and \eqref{eta_zeta} for $\eta s$ and $\zeta s$).
%\section{Asymptotics of the Green Kernel without pathological directions}\label{sec:7}

%\begin{theorem}[Asymptotics in the quadrant, general case]
%Suppose (\ref{simple}) and $r_{2}, r_{1} > -1$. Then, the Green's density function $g^{z_0}$ of this process has the asymptotics of Theorem \ref{thm:main} when $\alpha \to \alpha_0 \in (0, \pi/2) \backslash\{\alpha^*, \alpha^{**}\}$ and $r\to\infty$ where $h_{\alpha_0}(z_0)$ is given in~(\ref{harm1}),
%$h_{\alpha^*}(z_0)$ by~(\ref{harm2}) (with the symmetric formula for $ h_{\alpha^{**}}(z_0)$) and
%$$c^* = \frac{\gamma_2(x^*,y^*)}{\partial_y\gamma(x^*, y^*)}\frac{x'_s(s^*)}{\gamma_2(x'_s(\zeta s^*),y'(\zeta s^*))},\quad c^{**} = \frac{\gamma_1(x^{**},y^{**})}{\partial_x\gamma(x^{**}, y^{**})}\frac{y'_s(s^{**})}{\gamma_1(x'_s(\eta s^{**}),y'_s(\eta s^{**}))}.$$
%\end{theorem}

\begin{proof}[Proof of \nb{the} asymptotics in Theorem \ref{thm1} when $z_0 \neq (0,0)$]
%(without the non-nullity of constants),
We use the identity $g(a,b) = I_1(a,b) + I_2(a,b) + I_3(a,b)$, using the expressions provided in Lemma~\ref{residus}. By the classical saddle-point method (see details in \cite[Lemma $8.1$]{franceschi2023asymptotics}), the sum of the integrals of Lemma \ref{residus} along $\Gamma_{\alpha,x}$ and $\Gamma_{y,\alpha}=\underrightarrow{\overleftarrow{Y^{+}(\Gamma_{x,\alpha})}}$ has the following asymptotic expansion
\begin{equation}\label{ints_col}
\frac{1}{ 2 \pi i } \int\limits_{\Gamma_{x, \alpha}}\frac{\phi_2(x) \gamma_2(x, Y^+(x))}{\partial_y\gamma(x, Y^+(x))} \exp(-ax -b Y^+(x))dx \ +\  \frac{1}{ 2 \pi i } \int\limits_{\Gamma_{y, \alpha}}
\frac{\phi_1(y) \gamma_1(X^+(y), y)}{\partial_x\gamma(X^+(y), y)} \exp(-aX^+(y) -b y)dy 
\end{equation}
$$+ \frac{1}{2\pi i }\int\limits_{ \Gamma_{y, \alpha}}
 \exp((a_0-a) X^{+}(y)  + (b_0-b) y ) 
 \frac{dy}{\partial_x\gamma(X^{+}(y), y) } $$ $$\ \underset{r\to\infty \atop \alpha \to \alpha_0}{=}\  e^{-r(\cos(\alpha)x(\alpha) + \sin(\alpha)y(\alpha))}\left(\frac{1}{\sqrt{r}}
  \sum_{k=0}^n \frac{c^{z_0}_k(\alpha)}{ r^{k}} + o\left(\frac{1}{r^n\sqrt{r}}\right)\right).$$
where \nb{$(a,b) = (r\cos(\alpha), r\sin(\alpha))$, and}
\nr{\begin{equation}\label{eq:c0}
{c_0^{z_0}(\alpha) = \frac{1}{\sqrt{2\pi(\cos(\alpha) + \sin(\alpha))^2}}\sqrt\frac{\sin(\alpha)}{\partial_y\gamma(x(\alpha), y(\alpha))}h_{\alpha}(z_0) = \frac{1}{\sqrt{2\pi(\cos(\alpha) + \sin(\alpha))}}h_{\alpha}(z_0)}.%\gamma_1(x(\alpha),y(\alpha))\phi_1^{z_0}(y(\alpha)) + \gamma_2(x(\alpha),y(\alpha))\phi_2^{z_0}(x(\alpha)) +  e^{z_0\cdot (x(\alpha),y(\alpha))}.
\end{equation}}
\nb{\begin{align}
c_0^{z_0}(\alpha) &= \frac{1}{\sqrt{2\pi(\cos(\alpha) + \sin(\alpha))^2}}\sqrt\frac{\sin(\alpha)}{\partial_y\gamma(x(\alpha), y(\alpha))}\\
&\times\left(\gamma_1(x(\alpha),y(\alpha))\phi_1(y(\alpha)) + \gamma_2(x(\alpha),y(\alpha))\phi_2(x(\alpha)) + e^{(x(\alpha),y(\alpha))\cdot z_0} \right)\\
& = \frac{1}{\sqrt{2\pi(\cos(\alpha) + \sin(\alpha))}}h_{\alpha}(z_0)
\end{align}
by the explicit expressions of $\phi_1(y(s))$ and $\phi_2(x(s))$ given in \eqref{explicite} and \eqref{explicite2}, evaluated at $s = \mathfrak{s}(\alpha)$ (see \eqref{mathfraks}). }\ngray{With this precision, we see precisely where the explicit expressions of Laplace transforms arise in the constants in asymptotics of Green fonctions.}

Lemma \ref{pp} shows that, when $z_0 \neq (0,0)$, integrals over $S^{\pm}_x, S^{\pm}_y$ are negligible compared to those over paths of steepest descent. Finally, Theorem~\ref{harmoniques} gives the explicit form of residues of Lemma~\ref{residus} providing $h_{\alpha^*}(z_0)$, $h_{\alpha^{**}}(z_0)$. %We prove in Appendix~\ref{App:B} that functions $h_\alpha(z_0)$ are nowhere null, namely that the main term in asymptotics are non null.
\end{proof}

For the case $z_0 = (0, 0)$, we establish two preliminary lemmas. The first one is a consequence of the general Martin boundary theory.
\begin{lemma}\label{harmpasbeau}
For $\alpha \in [\alpha^*, \alpha^{**}]$, $z \longmapsto h_\alpha(z)$ is harmonic on $\R_+^2 \backslash \{(0,0)\}$.
\end{lemma}
\begin{proof}
For $z_0 =(a_0, b_0) \neq (0,0)$, we may consider the process evolving in $\R_+^2 \cap \{(x,y), \ x + y \geq a_0 + b_0\}$. Since $h_\alpha$ is the limit of the quotient of Green's kernels, \cite{kunitaWatanabe1965} implies its harmonicity over all these domains, and \nr{then}\nb{thus} over $\R_+^2 \backslash \{(0,0)\}$.   
\end{proof}% We don't develop this aspect since we construct the Martin boundary of the process (about the process living in the whole quarter plane $\R_+^2$) in Section~\ref{sec:9}.

\begin{lemma}\label{moyenne}
Let $\Theta$ be the contour defined by $\Theta := \{z \in \R_+^2\,:\, |z| = 1\}$ and $T_\Theta := \inf\{t\geq 0, Z_t \in \Theta\}$ the stopping time at $\Theta$. Then, for all $z_0 \in \R_+^2$ satisfying $|z_0| < 1$,
\begin{equation}\label{moyenne!}
h_{\alpha_0}(z_0) = \int_{\Theta} h_{\alpha_0}(z) \P_{z_0}(Z_{T_{\Theta}} = dz).
\end{equation}
\end{lemma}

\begin{proof}
%The equality can be written $h_{\alpha_0}(z_0) = \E_{z_0}\left[h(Z_{T_{\Theta^{z_0}_\varepsilon}})\right]$. 
%In general, this kind of equality does not hold for every stopping time. 
Suppose first that $z_0 \neq (0,0)$. The process $(h_\alpha(Z_t))_{t\geq0}$ is a martingale: indeed, for $t,s \geq 0$, $$\E_{z_0}\left[h_\alpha(Z_{t+s})|\mathcal{F}_t\right] = \E_{Z_t}\left[h_\alpha(Z_{s})\right] = h_\alpha(Z_{t})$$
by the strong Markov property and the harmonicity of $h_\alpha$ (see Lemma \ref{harmpasbeau}). Furthermore, under $\P_{z_0}$, the process $(h_\alpha(Z_{t\wedge T_{\Theta}}))_{t\geq0}$ is bounded above by $\sup_{|z| \leq 1}h_\alpha(z) < \infty$ since $h_\alpha$ is continuous. Then, by the optional stopping theorem for bounded martingales, we obtain $h_{\alpha_0}(z_0) = \E_{z_0}\left[h(Z_{T_{\Theta}})\right]$ which is precisely the desired equality. 

Now suppose $z_0 = (0,0)$ and consider a sequence $(z_n)_{n\geq1}$ in the quarter plane converging to $(0,0)$ such that $0<|z_n| < 1$. By continuity of $h_\alpha$, $h_{\alpha}(z_n)$ converges to $h_\alpha(z_0)$  as $n$ goes to $+\infty$. 
Since equation \eqref{moyenne!} holds for all nonzero initial conditions, it suffices to show that $$\int_{\Theta} h_{\alpha_0}(z) \P_{z_n}(Z_{T_{\Theta}} = dz) \underset{n\to+\infty}{\longrightarrow} \int_{\Theta} h_{\alpha_0}(z) \P_{(0,0)}(Z_{T_{\Theta}} = dz).$$ By continuity and boundedness of $h_\alpha$ on $\{z \in \R_+^2, |z|\leq 1\}$),  it is enough to show that $\mathcal{L}_{z_n}Z_{T_\Theta} \underset{n\to+\infty}{\longrightarrow}\mathcal{L}_{z_0}Z_{T_\Theta}$ weakly where $\mathcal{L}_{z}Z_{T_\Theta}$ denotes the law of $Z_{T_\Theta}$ with initial condition $Z_0 = z$. This follows from Assumption~\eqref{drift}, combined with \cite[Theorem 1]{HaRe-81b}, \nb{which ensures the continuity of the mapping from the non-reflected to the reflected path under the topology of uniform convergence on compacts.}%. Indeed, with Assumption~\eqref{drift}, if $X$ is the non-reflected process starting from $0$, we can show that $\psi(z_n + X)_{T_\Theta(\psi(z_n + X))} \underset{n\to+\infty}\longrightarrow \psi(X)_{T_\Theta(\psi(X))}$ almost surely. This gives the desired convergence.}
%\color{gray}This can be showed while considering $Z^0 = \psi((B_t + \mu t)_{t\geq 0}), Z^{z_n} = \psi((z_n + B_t + \mu t)_{t\geq 0})$ where $B$ is a degenerate Brownian motion of covariance $\Sigma$ and $\psi$ the map used in the proof of Theorem \ref{process}. By using the continuity of $\psi$ for the uniform convergence topology on compact sets and by proving that $(T_\Theta(Z^{z_n}))_{n\geq0} \underset{n\to+\infty}{\longrightarrow}T_\Theta(Z^{0})$ a.s. by properties of the Brownian motion, the weak convergence is obtained. This gives the conclusion for $z_0 = (0,0)$. \color{black}
\end{proof} 

\color{black}
We can now prove Theorem~\ref{thm1} in the case $z_0 = (0,0)$.
\begin{proof}[Proof of \nb{the asymptotics in} Theorem \ref{thm1} for $z_0 = (0,0)$]
By continuity of the process and by the Strong Markov property, if $(a,b)$ lies at a distance $> 1$ from $(0,0)$, then
\begin{equation}\label{C1}
g^{(0,0)}(a,b) = \int_{\Theta}g^{z}(a,b) \P(Z_{T_{\Theta}} = dz).
\end{equation}
Since the constant $C$ from the saddle-point method \cite[Lemma $8.1$]{franceschi2023asymptotics} depends continuously on $z_0$ and since the constants $D_{\alpha_0}$ in Lemma~\ref{pp} are locally uniform in $z_0$,  then for any compact set $K$ in the quadrant $\R_+^2$ with $d((0,0), K) > 0$, \nb{we have}
\begin{equation}\label{unif}
\sup_{z \in K} \left| g^{z}(r\cos(\alpha), r\sin(\alpha)) -  e^{-r(\cos(\alpha)x(\alpha) + \sin(\alpha)y(\alpha))}\frac{1}{\sqrt{r}}
  \sum_{k=0}^n \frac{c^{z}_k(\alpha)}{ r^{k}}\right| \underset{r\to\infty \atop \alpha \to \alpha_0}{=} o\left(\frac{e^{-r(\cos(\alpha)x(\alpha) + \sin(\alpha)y(\alpha))}}{r^n\sqrt{r}}\right).
\end{equation}
%This is justified by the local %uniformity in $z_0$ of %constants $D_{\alpha_0}$ in %Lemma \ref{pp}, but also in %the asymptotics %%(\ref{ints_col}). Indeed, the interested reader may look at the details of the saddle-point proof in \cite[Lemma 8.1]{franceschi2023asymptotics} and note that the constant $C$ in the corresponding proof is continuous in $z_0$.
By \nr{(\ref{unif})}\nb{this expansion},the asymptotics of (\ref{C1}) yield
\begin{equation}\label{berzingue}
g^{(0,0)}(a,b) \underset{r\to\infty \atop \alpha \to \alpha_0}{\sim} 
 \frac{e^{-r(\cos(\alpha)x(\alpha) + \sin(\alpha)y(\alpha))}}{\sqrt{r}}\int_{\Theta} h_{\alpha_0}(z) \P(Z_{T_{\Theta}} = dz).
\end{equation}
Lemma \ref{moyenne} combined with \eqref{berzingue} gives the result.
\end{proof}

\subsection{\nr{Non nullity}\nb{Positivity} of $h_\alpha(z_0)$.}\label{sub:6.2}
\nb{To make our asymptotics consistent, we prove here the positivity of the constants $h_\alpha(z_0)$.}\ngray{I clarified some points in this section, but there is no additional content.}
\begin{lemma}\label{lem:A1}
Let $\alpha \in (\alpha^*, \alpha^{**})$. Then, for every $R > 0$, there exists $z_0$ such that $|z_0| \geq R$ and $h_\alpha(z_0) > 0$. \nb{If $\alpha^* > 0$ (resp. $\alpha^{**} < \pi/2)$, then the same result holds for $h_{\alpha^*}$ (resp. $h_{\alpha^{**}}$)}.
\end{lemma}

\begin{proof}
By \nr{Theorem \ref{thm:main}}\nb{the explicit formulas \eqref{harm1} for $h_\alpha$, \eqref{harm2} for $h_{\alpha^{*}}$, and its equivalent for $h_{\alpha^{**}}$}, the following asymptotics hold as $r\to +\infty$ \nb{for $\alpha \in (\alpha^*, \alpha^{**})$, and for $\alpha = \alpha^*$ (resp. $\alpha = \alpha^{**}$) if $\alpha^* > 0$ (resp. $\alpha^{**} < \pi/2$):}
\begin{equation}\label{eqtheta}
h_\alpha(r\cos(\alpha), r\sin(\alpha)) \underset{r\to\infty}{\sim} %C_{\alpha}
e^{r(x(\alpha)\cos(\alpha) + y(\alpha)\sin(\alpha))}.
\end{equation} 
The conclusion follows with $z_0 = (r\cos(\alpha), r\sin(\alpha))$ for $r$ large enough.
\end{proof} 
The following Lemma is inspired by \cite[Lemma 8.3]{franceschi2023asymptotics}, \nr{. Combined with Lemma~\ref{lem:A1}, it}\nb{and} establishes the positivity of constants $h_{\alpha}(z_0)$ in the framework of Theorem~\ref{thm1}.

\begin{lemma}[\nb{Positivity of $h_\alpha(z)$}]\label{lem:lescstessontnonnulles}
Let $\alpha \in (\alpha^*, \alpha^{**})$, \nb{or $\alpha = \alpha^*$ (resp. $\alpha = \alpha^{**}$) if $\alpha^* > 0$ (resp. if $\alpha^{**} < \pi/2$). Let $z \in \R_+^2$. Then, $h_\alpha(z) \neq 0$.}\nr{ $z_0 \in \R_+^2$ and suppose that $h_\alpha(z_0) > 0$. Then, for all $z_1$ such that both coordinates of $z_1$ are strictly inferior to those of $z_0$,\;$h_\alpha(z_1)  > 0$.}
\end{lemma}

\begin{proof}
\nb{Let $\alpha \in (\alpha^*, \alpha^{**})$ and let $z_0$ such that both coordinates are larger than those of $z$, and such that $h_{\alpha}(z_0) > 0$, see Lemma~\ref{lem:A1}.} Let $V$ be a compact neighbourhood of $z_0$, and denote by
$T_V := \inf\{t \geq 0\,:\, Z_t \in V\}$
the hitting time of $V$. By the hypothesis on $z$, $\P_{z}(T_V < +\infty) > 0$ . By the strong Markov property,
\begin{align}
        g^{z}(r\cos(\alpha),r\sin(\alpha))& \geq \P_{z}(T_V < +\infty)\inf_{z'_{0}\in V}g^{z'_0} (r\cos(\alpha),r\sin(\alpha))\\
        &\geq \P_{z}(T_V < +\infty)\inf_{z'_{0}\in V}(h_\alpha(z'_0) + \varepsilon_{z'_0,\alpha}(r)) \frac{e^{-r(\cos(\alpha)x(\alpha) + \sin(\alpha)y(\alpha))}}{\sqrt{r}}.\label{B2}
    \end{align}
    \nb{for $r$ large enough} where \eqref{unif} provides the asymptotics
    $$\sup_{z'_0\in V}|\varepsilon_{z'_0,\alpha}(r)|\underset{r\to\infty}{\longrightarrow}0.$$
    Furthermore, by continuity of $h_\alpha$, the set $V$ can be chosen to satisfy $\inf_{z'_{0}\in V}h_\alpha(z'_0) > 0$. On the other hand, we also have $$g^{z}(r\cos(\alpha),r\sin(\alpha)) =  \frac{e^{-r(\cos(\alpha)x(\alpha) + \sin(\alpha)y(\alpha))}}{\sqrt{r}}(h_\alpha(z) + \varepsilon_{z'_0,\alpha}(r))$$ where $\varepsilon_{z'_0,\alpha}(r) \longrightarrow 0$ as $r \to +\infty$. Therefore, comparing the two expressions, we conclude that $h_\alpha(z) > 0$. \nb{If $\alpha = \alpha^*$ or $\alpha = \alpha^{**}$, the proof is analogous.}
\end{proof}

\nb{\begin{proof}[End of the proof of Proposition~\ref{pole}]
    The remaining part of the proof is equivalent to showing that $h_{\alpha^*}(z_0) > 0$: indeed, $h_{\alpha^*}(z_0)$ is equal to $\textnormal{res}_{x=x^*}\phi_2(x)$ up to a non-null multiplicative constant (see \eqref{harm2} and \eqref{explicite}). The positivity is established in the previous lemma.
\end{proof}}

\section{Asymptotics of Green's kernel in the particular directions $0, \alpha^*, \alpha^{**}$ and $\pi/2$.}\label{sec:8}
In Section \ref{sub:7.1}, we study the asymptotics of Green's functions in the direction $\alpha = 0$ under the assumption that $\gamma_2(x_{max}, Y^\pm(x_{max}))\neq 0$. In Section~\ref{sub:7.2}, we provide these asymptotics in the direction $\alpha^*$ if $\alpha^* > 0.$ Then, in Section~\ref{sub:7.3} we analyse the limiting case where $\alpha^* = 0$ and $\gamma_2(x_{max}, Y^\pm(x_{max}))= 0$. The analysis of the directions $\alpha = \pi/2$, $\alpha = \alpha^{**}$ if $\alpha^* < \pi/2$ and $\alpha^{**} = \pi/2, \gamma_1(X^\pm(y_{max},y_{max})) = 0$ is symmetrical. \nb{We then derive the proof of Theorem~\ref{thm:main} in Section~\ref{sec:toutepetite}.}
%\color{gray}
%The additional value compared to the non-degenerate case \cite{franceschi2023asymptotics} is about more than just the explicit expressions of harmonic functions. We also treat the case $\gamma_2(x_{max}, Y^\pm(x_{max}))= 0$ which were not considered in this article.
%\color{black}

\subsection{Case $\alpha \to 0$ if $\gamma_2(x_{max}, Y^\pm(x_{max}))\neq 0$.}\label{sub:7.1}
Before \nr{setting}\nb{deriving} the asymptotics, let us relate Green's densities $g^{z_0}$ to $f_1^{z_0}, f_2^{z_0}$. 

\begin{prop}[Link between densities]
Let $a, b \geq 0$. Suppose $z_0 \neq (a,0)$ and $z_0 \neq (0, b)$. Then, we have
$$
f_1^{z_0}(b)=\frac{1}{2} g^{z_0}(0,b)
\quad\text{and}\quad
f_2^{z_0}(a)=\frac{1}{2} g^{z_0}(a,0).
$$
\label{prop:linkhg}
\end{prop}
\begin{proof}
    By the functional equation~\eqref{Equation fonctionnelle}, if $x, y < 0$ then 
    \begin{equation}\label{YAYA}
    - \frac{\gamma(x, y)}{x}\varphi(x, y) = \frac{\gamma_1(x, y)}{x}\varphi_1(y) + \frac{\gamma_2(x, y)}{x}\varphi_2(x) + \frac{e^{(x, y)\cdot z_0}}{x}.
     \end{equation}
Furthermore, by elementary properties of Laplace transforms,
$$x\phi(x,y) \underset{x \to -\infty}{\longrightarrow} -\int_0^{+\infty} e^{by}g(0,b)db.$$
Then, letting $x \to -\infty$ in \eqref{YAYA}, we get $\int_0^{+\infty} e^{by}g(0,b)db = \phi_1(y) = \int_0^{+\infty} e^{by}f_1(b)db$. The injectivity of the Laplace transform concludes the proof. The case of $f_2$ is symmetrical.
\end{proof}

\begin{lemma}[Asymptotics at $\alpha = 0$] \label{tauberian} 
Suppose that $\phi_2$ does not have a pole. Let $\kappa = \left(\frac{1 + \mu_2}{2}\Gamma(1/2)\right)^{-1}$ where $\Gamma$ denotes the usual Gamma function. Then,
$$f_2^{z_0}(x) \sim_{x\to+\infty} \frac{\kappa\partial_\alpha[h_{\alpha}(z_0)]_{\alpha = 0}}{x^{3/2}}e^{-x_{max}x}.$$
\end{lemma}
\begin{proof}
%\color{gray}
%We have $Y^-(x) = Y^-(x_{max}) - \left({2(\mu_2\sqrt{\sigma_1\sigma_2} + \mu_1\sigma_2})\right)^{1/2}\sqrt{x_{max} - x} + o(\sqrt{x - x_{max}})$ by (\ref{Y-}). Then, using the continuation formula (\ref{cont}), we get that $\phi_2$ has a continuation at $x_{max}$ and that
%\begin{align*}\frac{\phi_2(x) - \phi_2(x_{max})}{\left({2(\mu_2\sqrt{\sigma_1\sigma_2} + \mu_1\sigma_2})\right)^{1/2}\sqrt{x - x_{max}}} &= r_{1}\phi_1(Y^-(x_{max})) + \gamma_1(x_{max}, Y^\pm(x_{max}))\phi_1'(Y^\pm(x_{max})) + b_0e^{z_0\cdot(x_{max}, Y^\pm(x_{max}))}\\
%& + \frac{\gamma_1(x_{max}, Y^\pm(x_{max}))\phi_1(Y^-(x_{max})) + e^{z_0\cdot(x_{max}, Y^\pm(x_{max}))}}{\gamma_2(x_{max}, Y^\pm(x_{max}))} + o_{x \to x_{max}}(1) \\
%&=r_{1}\phi_1(Y^-(x_{max})) + \gamma_1(x_{max}, Y^\pm(x_{max}))\phi_1'(Y^\pm(x_{max})) + b_0e^{z_0\cdot(x_{max}, Y^\pm(x_{max}))} \\
%&+ \phi_2(x_{max}) + o_{x \to x_{max}}(1) \\
%&= A + o(1)
%\end{align*}
%\color{black}

Note that $Y^-(x) = Y^-(x_{max}) - \sqrt{2(x_{max} - x)} + o(\sqrt{x - x_{max}})$ by (\ref{Y-}). Then, using the continuation formula (\ref{cont}), $\phi_2$ is continuous at $x_{max}$ and
\begin{align*}\frac{\phi_2(x) - \phi_2(x_{max})}{\sqrt{2(x_{max} - x)}} &= r_{1}\phi_1(Y^-(x_{max})) + \gamma_1(x_{max}, Y^\pm(x_{max}))\phi_1'(Y^\pm(x_{max})) + b_0e^{z_0\cdot(x_{max}, Y^\pm(x_{max}))}\\
& + \frac{\gamma_1(x_{max}, Y^\pm(x_{max}))\phi_1(Y^-(x_{max})) + e^{z_0\cdot(x_{max}, Y^\pm(x_{max}))}}{\gamma_2(x_{max}, Y^\pm(x_{max}))} + o_{x \to x_{max}}(1) \\
&=r_{1}\phi_1(Y^-(x_{max})) + \gamma_1(x_{max}, Y^\pm(x_{max}))\phi_1'(Y^\pm(x_{max})) + b_0e^{z_0\cdot(x_{max}, Y^\pm(x_{max}))} \\
&+ \phi_2(x_{max}) + o_{x \to x_{max}}(1) \\
&=: A + o_{x \to x_{max}}(1)
\end{align*}
where $z_0 = (a_0, b_0)$.
Then, \nr{by some Tauberian theorem (see in particular \cite[Lemma C.2]{dai_reflecting_2011})}\nb{by the Tauberian theorem given by \cite[Lemma C.2]{dai_reflecting_2011}}, we obtain $$f_2^{z_0}(x) \underset{x \to +\infty}{\sim} \frac{A}{\Gamma(1/2)x^{3/2}}e^{-x_{max}x}.$$
It then remains to show that $\frac{A}{\Gamma(1/2)}=\kappa\partial_\alpha[h_{\alpha}(z_0)]_{\alpha = 0}$. From equation \eqref{s_alpha}, we have:
\begin{align*}
\partial_\alpha[h_{\alpha}(z_0)]_{\alpha = 0^+} &= (\mathfrak{s}^{-1})'(s_{max})\partial_s\left[\gamma_1(x(s),y(s))\phi_1(y(s)) + \gamma_2(x(s),y(s))\phi_2(x(s)) + e^{a_0x(s) + b_0y(s)}\right]_{s=s_{max}}\\
&= \left(\frac{-(\mu_1 + \mu_2)}{(\mu_1 + \mu_2)^2}\right)\big(r_{1}y'(s_{max})\phi_1(y(s_{max})) + \gamma_1(x(s_{max}),y(s_{max}))y'(s_{max})\phi_1'(y(s)) \\&+ y'(s_{max})\phi_2(x(s_{max})) + y'(s_{max})b_0e^{a_0x(s_{max}) + b_0y(s_{max})}\big)\\
&= \frac{1 + \mu_2}{2}A
\end{align*}
since $y'(s_{max}) = y'(\mu_2) = -\frac{1 + \mu_2}{2}$ and $x'(s_{max}) = 0$. The conclusion follows.
\end{proof}

\nb{We can now establish the asymptotics of Green’s functions as $\alpha \to 0$.} We use the notation $c_0(\alpha)$ and $c_1(\alpha)$ for the constants in the first and second terms, respectively, in the asymptotic expansion of Green's fonctions (cf~\eqref{unif}).
\begin{theorem}[Asymptotics with $\alpha \to 0$] \label{alpha=0}
Suppose $\gamma_2(x_{max}, Y^\pm(x_{max})) \neq 0$. % (i.e. that we are not in the case of subsection \ref{particulier}). 
Then,
\begin{equation}\label{7..2}
    h_\alpha(z_0) \underset{\alpha\to 0}{\sim} \alpha \partial_\alpha[h_{\alpha}(z_0)]_{\alpha = 0}
\end{equation}  
{and thus $c^{z_0}_0(\alpha) \underset{\alpha\to 0}{\sim} \frac{1}{\sqrt{2\pi}}\alpha \partial_\alpha[h_{\alpha}(z_0)]_{\alpha = 0} $}. Furthermore, if $\alpha^* = 0$, 
\begin{equation}\label{72}
    c_1^{z_0}(\alpha)\underset{\alpha\to 0}{\longrightarrow} 2\kappa \partial_\alpha[h_{\alpha}(z_0)]_{\alpha = 0}.
\end{equation}
%non null for sufficiently large $z_0$ (at least when $s^* = \tilde\mu_2$) where $\kappa > 0$ is a positive constant which does not depend on $z_0$.
Moreover,
\begin{itemize}
\item  If $\alpha^* = 0$ (i.e. $\phi_2$ has no pole), then
$$ g^{z_0}(r\cos(\alpha), r\sin(\alpha)) \underset{r\to\infty \atop \alpha \to 0}{\sim} \partial_\alpha[h_{\alpha}(z_0)]_{\alpha = 0} \frac{e^{-r(\cos(\alpha)x(\alpha) + \sin(\alpha)y(\alpha))}}{\sqrt{r}}\left({\frac{\alpha}{\sqrt{2\pi}}}+ \frac{2\kappa}{r}\right) .$$
%+ o\left(\alpha + \frac{1}{r}\right).
\item If $\alpha^* > 0$  (i.e. $\phi_2$ has a pole), then
$$
g^{z_0}(r\cos(\alpha), r\sin(\alpha)) \underset{r\to\infty \atop \alpha \to 0}{\sim} c^*h_{\alpha^*}(z_0)e^{-r(\cos(\alpha)x^* + \sin(\alpha)y^*)}
$$
where $h_{\alpha^*}(z_0)$ and $c^*$ are given by \eqref{harm2} and \eqref{eq:c*}, respectively.
\end{itemize}
% (see the parallel in [IrinaSandroMax])
Moreover, the constants $\partial_\alpha[h_{\alpha}(z_0)]_{\alpha = 0}$ and $h_{\alpha^{*}}(z_0)$ are nonzero in the corresponding asymptotics. %The symmetric result holds when $\alpha \to \pi/2$ if $\gamma_1(X^\pm(y_{max}), y_{max}) \neq 0$.
\end{theorem}

\begin{proof}
First, \eqref{7..2} follows from the regularity of $h_\alpha(z_0)$ in $\alpha$ and from the convergence $h_\alpha(z_0)\longrightarrow 0$ as $\alpha \to 0$ (see~\eqref{harm1}). Now, we analyze the asymptotics of the sum of the three integrals in \eqref{ints_col} along the saddle-point curves as $\alpha \to 0$. The integrands in the second and third terms are holomorphic in a neighbourhood of the saddle point $Y^+(x_{max})$. The integrand of the first term, namely $\phi_2$, has a branching point at $x_{max}$. For this reason, we perform the change of variables $\Gamma_{x, \alpha}= \underrightarrow{\overleftarrow{X^+ (\Gamma_{y, \alpha})}}$:
\begin{equation}\label{eq:7.2}
\int\limits_{\Gamma_{x, \alpha(a,b)}}\frac{\phi_2(x) \gamma_2(x, Y^+(x))}{\partial_y\gamma(x, Y^+(x))} \exp(-ax -b Y^+(x))dx =  \int_{ 
\Gamma_{y, \alpha(a,b)}   }
\frac{ \gamma_2(X^+(y), y) \phi_2 (X^+(y))}{ \partial_x\gamma(X^+(y), y) } \exp(-aX^+(y) -by)dy.
\end{equation}

Additionally, from \eqref{cont}:
\begin{equation}\label{poulet2}
    \phi_2(X^+(y)) =  \frac{ -\gamma_1(X^+(y),Y^-(X^+(y))) \varphi_1(Y^-(X^+(y))) -\exp \big(a_0 X^+(y) + b_0 Y^{-}(X^+(y)) \big)}{\gamma_2(X^+(y), Y^{-}(X^+(y)))}.
\end{equation} 
Note that $X^+(y)$ is holomorphic in a neighbourhood of $Y^{\pm}(x_{max})$. The crucial point is that $Y^-(X^+(y))$ is also holomorphic there. Indeed, it can be expressed as 
\begin{equation}\label{Y-X+}
Y^-(X^+(y)) = \frac{X^+(y)^2 + 2\mu_1X^+(y)}{y}.
\end{equation} 
To see this, note that $Y^-(x)$ and $Y^+(x)$ are the two roots of $y\longmapsto \frac{1}{2}(x-y)^2 + \mu_1x + \mu_2y$. Then, by Vieta's equations and since $Y^+(X^+(y)) = y$, \eqref{Y-X+} follows immediately. Since $\gamma_2(x_{max}, Y^{\pm}(x_{max}) \neq 0$, it follows from \eqref{poulet2} that $\phi_2(X^+(y))$ is holomorphic at $Y^\pm(x_{max})$, so the saddle-point method applies to the right-hand side of \eqref{eq:7.2}. Then, asymptotics of Green's functions become
\begin{equation}\label{poulet}
g(r\cos(\alpha), r\sin(\alpha))  \underset{r\to\infty \atop \alpha \to \alpha_{0}}{\sim} c^*h_{\alpha^*}(z_0)e^{-r(\cos(\alpha)x^* + \sin(\alpha)y^*)}\fc_{\alpha^* >0}\quad\quad\quad\quad
\end{equation}
%\vspace{-3.5mm}
$$\quad\quad\quad\quad\quad\quad\quad\quad\quad+\  e^{-r(\cos(\alpha)x(\alpha) + \sin(\alpha)y(\alpha))} \frac{1}{\sqrt{r}}\left(c_0(\alpha) + \frac{c_1(\alpha)}{r}\right)$$
for all $\alpha_0 \in [0, \epsilon]$ where $\epsilon >0$ is sufficiently small. % where constants $\tilde c_0(\alpha)$ and $\tildec_1(\alpha)$ are continuous in $\alpha$.
%(for $\alpha = 0$, this is justified by continuity in $\alpha$). %This step in the proof aimed at justifying that the formula of the asymptotic in (\ref{dvt}) also holds for $\alpha_0 = 0$. 
It remains to show that if $\alpha^* = 0,$ then \eqref{72} holds. With $\alpha = 0$, equation \eqref{poulet} becomes \begin{equation}
g(r,0)  \underset{r\to\infty}{\sim}  \frac{e^{-rx_{max}}}{r^{3/2}}c_1(0).
\end{equation}  %see \eqref{harm1}, $c_0(\alpha) = h_\alpha(z_0) \underset{\alpha\to 0, \alpha > 0}{\longrightarrow} 0$ because the sum becomes telescopic)
By Proposition~\ref{prop:linkhg}, $g(r,0) = 2f_2(r)$. Finally, Lemma \ref{tauberian} applies and \nr{concludes}\nb{completes} the asymptotic \nb{analysis}. The non-vanishing of $\partial_\alpha[h_{\alpha}(z_0)]_{\alpha = 0^+}$ \nr{and $h_{\alpha^{*}}(z_0)$ respectively }is \nr{symmetrical}\nb{analogous to}\nr{ from} the case $\alpha \in (\alpha^*, \alpha^{**})$, see Section~\ref{sub:6.2}. \nb{The non-vanishing of $h_{\alpha^{*}}(z_0)$ if $\alpha^* > 0$ is already proved in Lemma~\ref{lem:lescstessontnonnulles}}.
\end{proof}

\subsection{Case $\alpha \to \alpha^*$ when $\alpha^* > 0$}\label{sub:7.2}
%Again, we only study the direction $\alpha^*$ (the direction $\alpha^{**}$ being studied symmetrically). Note that $\gamma_2(x_{max}, Y^\pm(x_{max}))\neq 0$. 

\begin{theorem}[Asymptotics with $\alpha \to \alpha^*$]\label{directions_poles}
Suppose $\alpha^* > 0$ (i.e., $\phi_2$ has a pole). Then:
\begin{itemize}
\item If $r(\alpha - \alpha^{*})^2\to 0$, then:
\begin{equation}\label{asympole}
g^{z_0}(r\cos(\alpha), r\sin(\alpha))
\underset{r\to\infty \atop \alpha\to \alpha^*}{\sim} 
\frac{1}{2}h_{\alpha^*}(z_0)e^{-r(\cos(\alpha)x^* + \sin(\alpha)y^*)}.
  \end{equation}  
\nb{where $h_{\alpha^*}(z_0)$ is given by \eqref{harm2}.}
\item If $r(\alpha - \alpha^{*})^2\to K>0$ for some constant $K$, then for $\alpha < \alpha^*$ (resp. $\alpha > \alpha^*$),
\begin{equation}
g^{z_0}(r\cos(\alpha), r\sin(\alpha))
\underset{r\to\infty \atop \alpha\to \alpha^*}{\sim} 
c_Kh_{\alpha^*}(z_0)e^{-r(\cos(\alpha)x^* + \sin(\alpha)y^*)}
  \end{equation}  
  $$\left(resp.\quad g^{z_0}(r\cos(\alpha), r\sin(\alpha))\underset{r\to\infty \atop \alpha\to \alpha^*}{\sim} 
\tilde c_Kh_{\alpha^*}(z_0)e^{-r(\cos(\alpha)x^* + \sin(\alpha)y^*)}\right).$$
where constants $c_K > 0$, $\tilde c_K > 0$ are independant of initial condition $z_0$.
\item If $r(\alpha-\alpha^{*})^2\to \infty$, then:
\begin{itemize}
\item If $\alpha<\alpha^{*}$, then \begin{equation}
g^{z_0}(r\cos(\alpha), r\sin(\alpha))
\underset{r\to\infty \atop \alpha\to \alpha^*}{\sim} 
h_{\alpha^*}(z_0)e^{-r(\cos(\alpha)x^* + \sin(\alpha)y^*)}.
  \end{equation}  
\item If $\alpha>\alpha^{*}$, then:
\begin{equation}
g^{z_0}(r\cos(\alpha), r\sin(\alpha))
\underset{r\to\infty \atop \alpha\to \alpha^*}{\sim} 
h_{\alpha^*}(z_0)e^{-r(\cos(\alpha)x(\alpha) + \sin(\alpha)y(\alpha)} \frac{1}{\sqrt{r}}\frac{C}{\alpha-\alpha^{*}}
  \end{equation} where $C$ is a positive constant independent of initial condition $z_0$.
\end{itemize}
\end{itemize}
Furthermore, $h_{\alpha^*}(z_0) > 0$.
%The symmetric result holds for $\alpha^{**}$ in the case $\alpha^{**} < \pi/2$ (i.e. when $\phi_1$ has a pole).
Constants $c_K, \tilde c_K$ and $C$, are made explicit in~\cite[Section 10]{franceschi2023asymptotics}.
\end{theorem}
The proof is analogous to~\cite[Section 10]{franceschi2023asymptotics}\nb{, which compares the asymptotic contribution of the pole term and the saddle-point term in the expressions of Lemma~\ref{residus} for $g^{z_0} = I_1 + I_2 + I_3$.} The non-vanishing of $h_{\alpha^{*}}(z_0)$ \nb{was already proved in Lemma~\ref{lem:lescstessontnonnulles}.}\nr{is symmetrical from the case $\alpha \in (\alpha^*, \alpha^{**})$, see Section~\ref{sub:6.2}.}

\nr{\begin{proof}[End of the proof of Proposition~\ref{pole}]
    The remaining part of the proof is equivalent to showing that $h_{\alpha^*}(z_0) \neq 0$, which is established by the previous theorem. 
\end{proof}}\ngray{This is more natural to prove Proposition 3.4 earlier, after the proof of Theorem 1}

\subsection{Last particular case: $\gamma_2(x_{max}, Y^\pm(x_{max})) = 0$}\label{particulier}\label{sub:7.3}
\begin{theorem}\label{thm:7}
Suppose $\gamma_2(x_{max}, Y^\pm(x_{max})) = 0$ (so $\alpha^* = 0$). Then,
\begin{equation}\label{le_dernier}
g^{z_0}(r\cos(\alpha), r\sin(\alpha))
\underset{r\to\infty \atop \alpha\to 0}{\sim} 
e^{-r(\cos(\alpha)x(\alpha) + \sin(\alpha)y(\alpha)} \frac{h_{0}(z_0)}{\sqrt{r}}.
  \end{equation}
where $h_{0}(z_0)$ is given by (\ref{harm3}). Furthermore, $h_{0}(z_0) \neq 0$.
\end{theorem}

\begin{proof}
The proof follows the same approach as that of Theorem \ref{alpha=0} but here $c_0(\alpha) \longrightarrow h_0(z_0) \neq 0$. Thus, we only consider the first term in \eqref{ints_col} with the representation~\eqref{eq:7.2} for $I_1(a,b)$. First, note from \eqref{Y-} that $\frac{d}{dy}\left[X^+(y)\right]_{y=Y^{\pm}(x_{max})} = 0$. Then, by \eqref{Y-X+}:
\begin{equation}\label{calcul1}
    \frac{d}{dy}\left[\gamma_2(X^+(y), Y^{-}(X^+(y)))\right]_{y=Y^{\pm}(x_{max})} = -\left(\frac{(\mu_2/2)^2 + 2\mu_1\mu_2^2/2}{(\mu_2^2/2 - \mu_2)^2}\right) = -1.
\end{equation}
Hence, $\gamma_2(X^+(y), Y^{-}(X^+(y))) = (y - Y^{\pm}(x_{max}))(-1 + o_{y\to Y^{\pm}(x_{max})}(1))$. Furthermore, $\gamma_2(X^+(y), y) = (y - Y^{\pm}(x_{max}))(1 + o_{y\to Y^{\pm}(x_{max})}(1))$ by similar calculations.

Using the same arguments as in the proof of Theorem~\ref{alpha=0}, the function \begin{equation}\label{eq:rediteeeee}
\frac{ \gamma_2(X^+(y), y) \phi_2 (X^+(y))}{ \partial_x\gamma(X^+(y), y) }
\end{equation}is holomorphic in a neighbourhood of $Y^{\pm}(x_{max}),$ expect possibly in $Y^{\pm}(x_{max})$ where $\phi_2(X^+(y))$ has a \nb{simple} pole\nr{of the first order} by \eqref{poulet2} and \eqref{calcul1}. Since $\gamma_2(X^+(y), y)$ has a zero of the same order at this point, then \nr{$\frac{ \gamma_2(X^+(y), y) \phi_2 (X^+(y))}{ \partial_x\gamma(X^+(y), y) }$}\nb{the quantity \eqref{eq:rediteeeee}} turns out to be holomorphic at $Y^{\pm}(x_{max})$ as well. Moreover, by \eqref{phi2decroit}, the following asymptotic expansion holds as $y \to Y^\pm(x_{max})$
\begin{equation}\label{nopb}
\frac{ \gamma_2(X^+(y), y) \phi_2 (X^+(y))}{ \partial_x\gamma(X^+(y), y) } =
 (-1+o_{y\to Y^{\pm}(x_{max})}(1))\Big(-e^{z_0\cdot (X^+(y), Y^-(X^+(y)))}
\end{equation}
$$+\ \ \frac{\gamma_1(\psi_{1}(X^+(y), y))}{\frac{\gamma_1}{\gamma_2}(\psi_{2}(X^+(y), y))}\sum_{n=1}^{+\infty} \left[\prod_{k=2}^{n}  \frac{\frac{\gamma_1}{\gamma_2}(\psi_{2k-1}(X^+(y), y)}{\frac{\gamma_1}{\gamma_2}(\psi_{2k}(X^+(y), y))}\right]\left[ \frac{e^{z_0\cdot \psi_{2n}(X^+(y), y))}}{\gamma_2(\psi_{2n}(X^+(y), y))} - \frac{e^{z_0\cdot \psi_{2n+1}(X^+(y), y)}}{\gamma_2(\psi_{2n +1}(X^+(y), y))}\right]\Big).$$
This implies \eqref{harm3}. %\eqref{nopb} is then holomorphic at $Y^{\pm}(x_{max})$.  %This quantity being continuous (and then bounded) at $Y^{\pm}(x_{max})$, the formula \eqref{nopb} is holomorphic at $Y^{\pm}(x_{max})$.
%Then, the saddle-point procedure can be applied to \eqref{ints_col} with representation~\eqref{eq:7.2} for $I_1(a,b)$. This gives asymptotics \eqref{le_dernier}. %keeping the first term like in Theorem~\ref{thm1}.% Note that $I_1, I_2, I_3$ have the same order of asymptotics so they all contribute in the constant $h_0(z_0)$ which is explicitly given by \eqref{harm3}. 
The proof of the non-vanishing of $h_0(z_0)$ is \nb{analogous to Lemma~\ref{lem:lescstessontnonnulles}.}\nr{symmetrical from the case $\alpha \in (\alpha^*, \alpha^{**})$, see Section~\ref{sub:6.2}.}
\end{proof}

\color{black}

\subsection{Proof of Theorem~\ref{thm:main}}\label{sec:toutepetite}
This is a direct consequence of Theorem~\ref{thm1}, \ref{alpha=0}, \ref{directions_poles}, and \ref{thm:7}.

\section{Harmonic functions and Martin boundary}\label{harmonic}\label{sec:9}

In this Section, we prove Theorem~\ref{thm:Martin}. In particular, we show in Section~\ref{sec:8.1} that the Martin boundary is homeomorphic to $[\alpha^*, \alpha^{**}]$ and in Section~\ref{8.2} that the Martin boundary is minimal.

\subsection{Context of Martin the boundary}\label{sec:8.1}
In this section, we consider the construction of the Martin boundary as presented in~\cite[Section 7.1]{Pinsky_1995} for elliptic processes and we adapt this approach to reflected degenerate processes. This method allows us to consistently link the \nr{Martin}harmonic functions $(h_\alpha)_{\alpha\in[\alpha^*,\alpha^{**}]}$ found in Theorem~\ref{thm:main} and the Martin boundary. Note that another general construction of the Martin compactification is presented in \cite{kunitaWatanabe1965}.  %A general construction of the Martin compactification is given in \cite{kunitaWatanabe1965}. However, assumptions of  \cite{kunitaWatanabe1965} are not in line with  Theorems \ref{thm1}, \ref{alpha=0}, \ref{directions_poles} and \ref{thm:7}. In~\cite[Section 7.1]{Pinsky_1995}, construction of Martin boundary for elliptic processes is consistent with our   results. Nevertheless, it can't be generalised immediately to degenerate processes. Though, in this section, we adapt the method of \cite{Pinsky_1995} to reflected degenerate processes. 

%We use notation $g(z_0, z) = g^{z_0}(z)$ for $z_0, z \in \R_+^2$.
%We define the Martin Kernel and an associated suitable metric (see for example \cite{Pinsky_1995}).
\begin{defi}[Martin kernel]%Fix $x_0 = (0,0)$ in $\R_+^2$. 
For $z_0, z_1 \in \R_+^2$, we define the Martin kernel
\begin{equation}
k(z_0,z_1) =\left\{
    \begin{array}{ll}
        \frac{g^{z_0}(z_1)}{g^{(0,0)}(z_1)}    & \textnormal{if} \quad z_1 \neq (0,0)\\
      0 & \textnormal{if} \quad z_1 = (0,0) \quad \textnormal{or} \quad z_0 = z_1
    \end{array}
\right.
\end{equation}
and the Martin metric
\begin{equation}
\rho(z_1, z_2) = \int_{\R_+^2}\frac{|k(x,z_1) - k(x,z_2)|}{1 + |k(x,z_1) - k(x,z_2)|}e^{-|x|^2}dx. %+ d_{c}(z_1,z_2)
\end{equation}
%where $d_{c}$ is the stereographic distance of the euclidean space $\R^2$ (i.e. the one-point compactification distance).
\end{defi}
%The separation of $\rho$ should be justified with harmonicity arguments, because we should have $\mathcal{G}g(\cdot, z) = \delta_z$ (where $\mathcal{G}$ is defined after \eqref{BVP}) as the backward equation (8.2a) - (8.2c) in \cite{HaRe-81} mentions for the non degenerate case. 
By usual considerations \cite{Pinsky_1995}, $\rho$ is a metric equivalent to the Euclidean one on $\R_+^2$.  % The equivalence is essentially given by the continuity of $g(x,y)$ (expect at the diagonal $x = y$) and the dominated convergence theorem. 
%\textcolor{red}{For, the separation of $\rho$, this might could be done due to harmonicity arguments. Si $k(x,y) = k(x,z)$ pour tout $x$, peut on déduire aisément que $y = z$?? J'ai envie de dire que $k(., y)$ est harmonique partout sauf en $y$ et pareil pour $k(., z)$ en $z$ car c'est vrai dans le monde discret. Mais je suis un peu plus fébrile ici. Cf proposition 6.2 de \cite{kunitaWatanabe1965}: est on vraiment non harmonique en $y$? Sinon, je peux essayer d'ajouter $[\alpha^*, \alpha^{**}]$ formellement  avec une métrique ad-hoc (i.e. artificiellement) mais c'est dommage... Ca ferait vraiment encore plus "boite noire" et ça serait très très laid}
A sequence $(y_n)_{n\geq0}$ of $\R_+^2$ is \nb{called} a Martin sequence if $(k(\cdot,y_n))_{n\geq0}$ converges pointwise. Two Martin sequences are said to be equivalent if their limit functions are equal. We then define $M$ as the quotient of the set of all Martin sequences by this equivalence relation. Each $\xi\in M$ is then naturally associated with function denoted by $k(\cdot, \xi)$. \nr{Note that}\nb{The metric} $\rho$ extends naturally to $M$ with the same formula, so that the map
\begin{equation}\label{Martin_inj}
    \iota\colon\left\{
    \begin{aligned}
        \R_+^2 &\longrightarrow (M,\rho)\\
        z &\longmapsto k(\cdot,z)
    \end{aligned}
    \right.
\end{equation} 
is injective and continuous.  We define the Martin Boundary $\Gamma$ as $\Gamma = M\backslash \iota(\R_+^2)$.

\begin{lemma}\label{topoMartin}
Let $(h_{\alpha})_{\alpha \in [\alpha^*, \alpha^{**}]}$ be defined in Theorem~\ref{thm:main}. Then, the map $$\Phi : \alpha \in [\alpha^*, \alpha^{**}] \longmapsto h_\alpha(\cdot)/h_\alpha(0) \in \Gamma$$ is a homeomorphism.
\end{lemma}
Before proving this lemma, \nr{let us recall some behaviours of}\nb{we recall some properties of the family} $(h_\alpha)_{\alpha\in[\alpha^*, \alpha^{**}]}$.
\begin{rem} Note that for $z \in \R_+^2$
     \begin{itemize}
         \item If $\alpha^* > 0,$ then by \eqref{harm1}, $$h_\alpha(z) \underset{\alpha \to \alpha^* \atop\alpha > \alpha^*}{\longrightarrow} +\infty.$$
         \item If $\alpha^* = 0$ and $\gamma_2(x_{max}, Y^\pm(x_{max}) = 0$, then by \eqref{harm3}, $$h_\alpha(z) \underset{\alpha \to 0\atop \alpha > 0}{\longrightarrow} h_0(z) > 0.$$
         \item If $\alpha^* = 0$ and $\gamma_2(x_{max}, Y^\pm(x_{max}) \neq 0$, then by \eqref{harm2}, $$h_\alpha(z) \underset{\alpha \to 0\atop \alpha > 0}{\longrightarrow} 0.$$
     \end{itemize}
\end{rem}

\begin{proof}[Proof of Lemma~\ref{topoMartin}]
%The proof is the synthesis of all the asymptotic developments we made along the article. 
\nr{By Theorem~\ref{thm:main}}\nb{By Theorems~\ref{thm1}, \ref{thm:main}, \ref{alpha=0}, \ref{directions_poles}, and \ref{thm:7},} %Theorems~\ref{thm1}, \ref{alpha=0}, \ref{directions_poles}, \ref{thm:7},
$\Phi$ is surjective. %Since the equivalence can let $r\to\infty$ and $s\to s_0$ independently, we have showed without additional work the continuity of $\Phi$. Let us prove the continuity with explicit formulae of Theorem~\ref{thm:main}. 
To prove the continuity of $\Phi$, note that a sequence $(\xi_n)_{n\geq0}$ converges to some $\xi \in M$ if $k(\cdot, \xi_n)$ converges pointwise toward $k(\cdot, \xi)$ almost everywhere. Therefore, the proof of the continuity of $\Phi$ is reduced to showing that, for any $z \in \R_+^2$, the map $\alpha \longmapsto \Phi(\alpha)(z)$ is continuous. Let $z \in \R_+^2$.
\begin{itemize}
    \item  By (\ref{harm1}), the map $\alpha\mapsto\Phi(\alpha)(z)$ is continuous on $(\alpha^*, \alpha^{**})$.
    \item If $\alpha^* = 0$ and $\gamma_2(x_{max}, Y^{\pm}(x_{max})) \neq 0$, then we have: $$\Phi(\alpha)(z) = \frac{h_\alpha(z)}{h_\alpha(0)} = \frac{h_\alpha(z)}{\alpha}\frac{\alpha}{h_\alpha(0)} \underset{\alpha\to 0}{\longrightarrow} \frac{[\partial_\alpha h_\alpha(z)]_{\alpha=0}}{[\partial_\alpha h_\alpha(0)]_{\alpha=0}} = \Phi(0)(z)$$ 
so $\alpha\mapsto\Phi(\alpha)(z)$ is continuous at $\alpha^* = 0$.

\item  If $\alpha^* > 0$, then \nr{$\frac{h_{\alpha(s)}(z)}{h_{\alpha(s)}(0)}$}\nb{$\frac{h_{\alpha}(z)}{h_{\alpha}(0)}$} can be written as \nr{$$\frac{\sum_{m=-\infty}^{0} f_m(s)e^{z_0\cdot (a_{m}(s), b_{m}(s))} + \frac{1}{\gamma_2(\zeta s)}\sum_{m=1}^{+\infty} f_m(s)e^{z_0\cdot (a_{m}(s), b_{m}(s))}}{\sum_{m=-\infty}^{0} f_m(s) + \frac{1}{\gamma_2(\zeta s)}\sum_{m=1}^{+\infty} f_m(s)}$$}
\nb{$$\frac{\sum_{m=-\infty}^{0} \kappa_m(\alpha)e^{z\cdot (a_{m}(\mathfrak{s}(\alpha)), b_{m}(\mathfrak{s}(\alpha)))} + \frac{1}{\gamma_2(\zeta \mathfrak{s}(\alpha))}\sum_{m=1}^{+\infty} \kappa_m(\alpha)e^{z\cdot (a_{m}(\mathfrak{s}(\alpha)), b_{m}(\mathfrak{s}(\alpha)))}}{\sum_{m=-\infty}^{0}\kappa_m(\alpha) + \frac{1}{\gamma_2(\zeta \mathfrak{s}(\alpha))}\sum_{m=1}^{+\infty} \kappa_m(\alpha)}$$}
\nb{where $(a_m(\mathfrak{s}(\alpha)), b_m(\mathfrak{s}(\alpha))),$ are defined by \eqref{eq:an} and \eqref{eq:bn}, with $(a_0(\mathfrak{s}(\alpha)), b_0(\mathfrak{s}(\alpha))) = (x(\alpha), y(\alpha))$, and $\kappa_m(\alpha)$ is given by \eqref{eq:kappam}.}\nr{with continuous functions $f_m.$}
Since \nb{$\gamma_2(\zeta(\mathfrak{s}(\alpha))) \underset{\alpha\to \alpha^*}{\longrightarrow}0$}\nr{$\gamma_2(\zeta(s)) \underset{s\to s^*}{\longrightarrow}0$} (see Notations~\ref{not:1} and Proposition~\ref{pole}), the expected continuity in $\alpha^*$ follows from standard continuity theorems on series\nr{(note that homeomorphism (\ref{s_alpha}) enables us to use the variable $s$ or $\alpha$ equivalently)}.
\item The remaining case $\alpha^* = 0$ and $\gamma_2(x_{max}, Y^{\pm}(x_{max})) = 0$ is analogous.
\item The proof of the continuity of $\Phi$ at $\alpha^{**}$ is symmetric.
\end{itemize} 
Next, let us show that $\Phi$ is injective. By the explicit expressions in Theorem \ref{thm:main}, the following asymptotics hold as $r\to +\infty$. For $\alpha^* < \alpha < \alpha^{**}$ and $0 \leq \theta \leq \pi/2$, we have:
\begin{equation}\label{eqtheta}
h_\alpha(r\cos(\theta), r\sin(\theta)) \underset{r\to\infty}{\sim} %C_{\alpha}
e^{r(x(\alpha)\cos(\theta) + y(\alpha)\sin(\theta))}.
\end{equation}% where $C_{\alpha} > 0$.  
If $\alpha^* = 0$, then for any $0 < \theta\leq \pi/2$,
\begin{equation}\label{eqtheta2}
h_0(r\cos(\theta), r\sin(\theta)) \underset{r\to\infty}{\sim} r\sin(\theta)e^{r(x(0)\cos(\theta) + y(0)\sin(\theta))} .
\end{equation}
If $\alpha^* > 0$ and $0 \leq \theta \leq \pi/2$, then %there exists a constant $C_{\alpha^*} > 0$
\begin{equation}\label{eqtheta3}
h_{\alpha^*}(r\cos(\theta), r\sin(\theta)) \underset{r\to\infty}{\sim} e^{r(x(\alpha^*)\cos(\theta) + y(\alpha^*)\sin(\theta))}.
\end{equation}
The corresponding symmetric asymptotic behavior holds for $h_{\alpha^{**}}$. If $\alpha, \alpha' \in [\alpha^*, \alpha^{**}]$ are distinct, then by \eqref{xalpha} and the preceding formulae, $$\frac{h_\alpha(r\cos(\alpha),r\sin(\alpha))}{h_{\alpha'}(r\cos(\alpha),r\sin(\alpha))} \underset{r\to\infty}{\longrightarrow} +\infty.$$ Hence, $h_\alpha \neq Ch_{\alpha'}$ for any constant $C$, and $\Phi$ is injective.

Since $\Phi$ is continuous, and since $[\alpha^*, \alpha^{**}]$ is compact, $\Phi(F)$ is closed in $\Gamma$ for any closed subset $F$ of $[\alpha^*, \alpha^{**}]$. Therefore, $\Phi$ is a homeomorphism.
\end{proof}

\begin{cor}%[Points characterizing the Martin boundary]
\label{construction}
The following properties hold:
\begin{enumerate}
\item[(i)] If $\eta, \xi \in M$ satisfy $k(\cdot,\xi) = k(\cdot, \eta)$, then $\eta = \xi$.
\item[(ii)] The metric space $(M, \rho)$ is compact.
\item[(iii)] $\iota(\R_+^2)$ is dense in $M$ with respect to $\rho$.
\item[(iv)] If a sequence $(y_n)_{n\geq0}\subset\R_+^2$ converges to $\eta \in \Gamma$ with respect to $\rho$, then $k(\cdot, y_n)$ converges pointwise to $k(\cdot, \eta)$.
\end{enumerate}
\end{cor}

\begin{proof}
Properties (i), (iii), and (iv) \nr{are direct}\nb{follow directly} from our construction. We now prove (ii). 
Let $(y_n)_{n\geq0}$ be a sequence in $M$. Then:
\begin{itemize}
\item Either $(y_n)_{n\geq0}$ has infinitely many points in $\Gamma$, \nr{and then has}\nb{in which case it has} a convergent subsequence since $\Gamma$ is compact (see Lemma~\ref{topoMartin}). 
\item \nr{Either}\nb{Or} $(y_n)_{n\geq0}$ has a bounded subsequence, \nr{and then we conclude the same since}\nb{in which case the conclusion follows as} $\rho|_{\R_+^2\times \R_+^2}$ is equivalent to the Euclidean metric.
\item \nr{Either}\nb{Or} $(y_n)_{n\geq0}$ has a subsequence that tends to infinity. Since $[0,\pi/2]$ is compact, $(y_n)_{n\geq0}$ has a subsequence \nr{converging}\nb{that tends} to infinity in some direction $\alpha \in [0, \pi/2]$. By Theorems \ref{thm1}, \ref{alpha=0}, \ref{directions_poles}, and \ref{thm:7} this subsequence converges \nb{(with respect to the metric $\rho$)} to $\fc_{\alpha <\alpha^{*}}h_{\alpha^*} + \fc_{\alpha^* \leq \alpha \leq \alpha^{**}}h_{\alpha} + \fc_{\alpha < \alpha^{**}}h_{\alpha^{**}}$\nr{ for the distance $\rho$}.
\end{itemize} 
\end{proof}
\begin{rem}\label{rem:*}
    By Corollary \ref{construction}, $M$ is the Martin compactification in the sense of \cite{kunitaWatanabe1965,Pinsky_1995}, and the Martin Boundary $\Gamma$ is homeomorphic to $[\alpha^*, \alpha^{**}]$. 
\end{rem}
In particular, by \cite[Theorem 4]{kunitaWatanabe1965} the following representation theorem holds.% for harmonic functions.

\begin{theorem}[Integral representation]\label{repres}
If $h$ is a non-negative harmonic function, then there exists a Radon measure $\mu_h$ on $[\alpha^*, \alpha^{**}]$ satisfying
\begin{equation}\label{inte}
\forall z \in \R_+^2,\quad h(z) = \int_{[\alpha^*, \alpha^{**}]} h_\alpha(z)d\mu_h(\alpha).
\end{equation} 
Furthermore, every function defined by (\ref{inte}) is harmonic.
\end{theorem}

\subsection{Minimality of\nr{Martin harmonic} functions \nb{$(h_\alpha)_{\alpha\in[\alpha^*,\alpha^{**}]}$} and Martin boundary}\label{8.2}
In this section, we prove that the Martin boundary is minimal. %The used techniques are inspired by \textcolor{blue}{Il me semble que c'est le même genre de raisonnement (par grandes déviations?) que dans [Martin boundary OF A KILLED RANDOM WALK ON A HALF-SPACE - IRINA IGNATIOUK-ROBERT]}
%\begin{theorem}
%For initial points $(a_0, b_0) = (x(\alpha), y(\alpha))$ with $\alpha \in (\alpha^*, \alpha^{**})$, the associated function $f_\alpha(u, v) = \sum_{n \in \Z} c_n e^{a_n u + b_n v}$ converges and is a \textbf{positive} harmonic fonction.
%\end{theorem}
%
%\begin{proof}
%La positivité des fonctions trouvées risque d'être tricky.
%\end{proof}
%
%\begin{theorem}
%Up to a positive multiplicative constant, $h_\alpha = f_\alpha$ for all $\alpha \in [\alpha^*, \alpha^{**}]$ 
%\end{theorem}
\begin{defi}[Minimal harmonic function]\label{def:minimal}
A non-negative harmonic function $h$ is said to be minimal if, for every pair of non-negative harmonic functions $f_1$ and $f_2$ satisfying $f_1 + f_2 = h$, both $f_1$ and $f_2$ are proportional to $h$.
\end{defi}

\begin{prop}[$\Gamma$ is minimal]\label{8.5}
The Martin boundary is minimal in the sense that if $\eta \in \Gamma$, then $k(\cdot, \eta)$ is minimal. \nb{In particular, the measure $\mu_h$ in representation \eqref{inte} is unique.}
\end{prop}
To prove this we state the following lemma.

\begin{lemma}\label{asymp_uniforme}
Let $\alpha^* < \alpha_1 < \alpha^{**}$ and $\epsilon > 0$. Then, there exist constants $\eta > 0$ and $r_0 >0$ such that
\begin{equation}
h_\alpha(r\cos(\alpha_1), r\sin(\alpha_1)) \geq \frac{1}{2}e^{r(x(\alpha_1)\cos(\alpha_1) + y(\alpha_1)\sin(\alpha_1) - \epsilon)}
\end{equation}
for all $r \geq r_0$ and $\alpha \in [\alpha_1-\eta, \alpha_1 + \eta]$.
\end{lemma}
%Let us now show the minimality.
\begin{proof}
    This follows from Theorem~\ref{thm:main}, where explicit formulae of $h_\alpha$ are given. 
\end{proof}
\begin{proof}[Proof of Proposition \ref{8.5}]
Let $\alpha_0 \in [\alpha^*, \alpha^{**}]$. We aim to prove that if $h_{\alpha_0} = \int_{[\alpha^*, \alpha^{**}]} h_\alpha d\mu(\alpha)$ for some Radon measure $\mu$, then $\mu$ is the Dirac measure at $\alpha_0$. This directly implies the minimality of $\Gamma$ using Definition~\ref{def:minimal} and Theorem~\ref{repres}. It suffices to show that the support of $\mu$ is exactly $\{\alpha_0\}$. %By contradiction, suppose the contrary. 
Suppose first that $\alpha^* < \alpha_0 <\alpha^{**}$. Let us prove that $\mu((\alpha^*,\alpha^{**})\backslash\{\alpha_0\}) = 0$.  Let $\alpha_1 \in (\alpha^*, \alpha^{**})\backslash\{\alpha_0\}$. \nr{Then}\nb{First}, by \eqref{xalpha}, we can choose $\epsilon > 0$ such that
\begin{equation}\label{8.9}
x(\alpha_1)\cos(\alpha_1) + y(\alpha_1)\sin(\alpha_1) - \epsilon > x(\alpha_0)\cos(\alpha_1) + y(\alpha_0)\sin(\alpha_1).
\end{equation}
\nr{By}\nb{Secondly, by} Lemma \ref{asymp_uniforme}, \nr{we obtain}\nb{there exists some} $\eta > 0$ such that 
$$h_\alpha(r\cos(\alpha), r\sin(\alpha)) \geq \frac{1}{2}e^{r(x(\alpha_1)\cos(\alpha_1) + y(\alpha_1)\sin(\alpha_1) - \epsilon)}$$
for $\alpha \in [\alpha_1-\eta, \alpha_1 + \eta]$ and $r\geq r_0$ large enough. Then,
$$h_{\alpha_0}(r\cos(\alpha_1), r\sin(\alpha_1)) \geq \int_{\alpha \in [\alpha_1-\eta, \alpha_1 + \eta]} h_{\alpha}(re_{\alpha_1})\mu(d\alpha)
 \underset{r\geq r_0}{\geq} \frac{\mu([\alpha_1-\eta, \alpha_1 + \eta])}{2}e^{r(x(\alpha_1)\cos(\alpha_1) + y(\alpha_1)\sin(\alpha_1) - \epsilon)}. $$
\nr{Combining with $\theta = \alpha_1$ and $\alpha = \alpha_0$ in (\ref{eqtheta}) with $\theta = \alpha_0$, we obtain from \eqref{8.9} that $\mu([\alpha_1-\eta, \alpha_1 + \eta]) = 0$ so the support of $\mu$ is included in $\{\alpha^*, \alpha_0, \alpha^{**}\}$.}
\nb{Considering $\theta = \alpha_1$ and $\alpha = \alpha_0$ in (\ref{eqtheta}), we obtain, for $r\geq r_1$ large enough:
$$e^{r(x(\alpha_1)\cos(\alpha_0) + y(\alpha_1)\sin(\alpha_0) - \epsilon)} \geq \frac{\mu([\alpha_1-\eta, \alpha_1 + \eta])}{2}e^{r(x(\alpha_1)\cos(\alpha_1) + y(\alpha_1)\sin(\alpha_1) - \epsilon)}. $$
By \eqref{8.9}, the asymptotics of the previous inequality as $r\to+\infty$ yield $\mu([\alpha_1-\eta, \alpha_1 + \eta]) = 0$.}
Therefore, $\mu$ can be written as $\mu = A\delta_{\alpha_0} + B\delta_{\alpha^*} + C\delta_{\alpha^{**}}$ for some non-negative constants $A, B,$ and $C$, i.e., $(1-A)h_{\alpha_0} = Bh_{\alpha^*} + C{h_{\alpha^{**}}}$. Now, considering the asymptotics (\ref{eqtheta}), (\ref{eqtheta2}), and (\ref{eqtheta3}), we immediately get $B = C = 0$ and $A = 1$. Hence, $\mu$ is the Dirac measure at $\alpha_0$ and $h_{\alpha_0}$ is minimal. The cases $\alpha_0 = \alpha^*$ and $\alpha_0 = \alpha^{**}$  are treated similarly. %$$\int_{\alpha \in [\alpha_1, \alpha_2]} h_{\alpha}(re_{\alpha_2})\mu(d\alpha) %\underset{r\to+\infty}{\succ}\inf_{\alpha \in [\alpha_1, \alpha_2]} c^{re_{\alpha}}(s(\alpha))% \underset{r\to+\infty}{\succ} \inf_{\alpha \in [\alpha_1, \alpha_2]} c^{re_{\alpha_2}}(s(\alpha)) \underset{r\to+\infty}{\succ} \frac{1}{2}e^{r(z(\alpha_1)\cdot z(\alpha_2))}$$%Then, for all $\alpha \in [\alpha_1, \alpha_2]$,$$%$$h_\alpha(r\cos(\alpha_2), r\sin(\alpha_2)) \simeq \geq C_\alpha e^{r(\cos(\alpha_2)x(\alpha_1) + \sin(\alpha_2)y(\alpha_1)}$$%(à préciser) and then by the integral representation $$f_{\alpha_0}(r\cos(\alpha_2), r\sin(\alpha_2)) \simeq \geq C \mu([\alpha_1, \alpha_2])e^{r(\cos(\alpha_2)x(\alpha_1) + \sin(\alpha_2)y(\alpha_1)}.$$%\item If $\alpha^* = 0$ (Lemma 7.4 + theorem 4 + theorem 6)%\item If $\alpha^* \neq 0$ : Equation \ref{residu}%Then $\mu = B\delta_{\alpha_0}$ and the conclusion follows.
\end{proof}
\begin{proof}[Proof of Theorem~\ref{thm:Martin}] This is a direct consequence of Lemma~\ref{topoMartin}, Remark~\ref{rem:*} and Propostion~\ref{8.5}.
    
\end{proof}

\color{gray}

\color{black}

\section{From Assumption \eqref{simple} to the general case}\label{App:A}
We stated and proved Theorems \ref{thm1}, \ref{thm:main} and \ref{thm:Martin} under Assumption (\ref{simple}). In this section, we generalise these theorems without \nb{assuming} \eqref{simple}. To achieve this, we apply transformations to the $x-$axis, $y-$axis, and time $t$, in order to reduce the problem \nb{to} a process $\tilde Z$ which satisfies \eqref{simple}.% Let be $\sigma_1, \sigma_2$ and $\mu = (\mu_1, \mu_2)$ with $\mu_1, \mu_2, \sigma_1, \sigma_2 > 0$. %In this appendix, any notation with some tilde refers to the case $\sigma_1 = \sigma_2 = 1$, $\mu_1 + \mu_2 = 1$ and those without tilde to our general case.
%The following lemma is directly obtained by verifying the conditions to be a DRBM.
\begin{prop}[Space-time dilatation]\label{dilatation}
Let $(Z_t)_{t\geq0}$ be a degenerate reflected Brownian motion with parameters $$\Sigma = \begin{pmatrix}
\sigma_1^2 & -{\sigma_1\sigma_2} \\
-{\sigma_1\sigma_2} & \sigma_2^2
\end{pmatrix},
\mu = \begin{pmatrix}
\mu_1  \\
\mu_2
\end{pmatrix},
R = \begin{pmatrix}
1 & r_{2} \\
r_{1} & 1
\end{pmatrix}.
$$
Then the process $$(\tilde Z_t)_{t\geq0} := \left(\left(\frac{\mu_1}{{\sigma_1}}+\frac{\mu_2}{{\sigma_2}}\right)\begin{pmatrix}
\frac{1}{{\sigma_1}} & 0 \\
0 & \frac{1}{{\sigma_2}} 
\end{pmatrix} Z\left(\frac{t}{\left(\frac{\mu_1}{{\sigma_1}}+\frac{\mu_2}{{\sigma_2}}\right)^2}\right)
\right)_{t\geq0}$$ is a degenerate reflected Brownian motion \nr{of}\nb{with} parameters 
$$\begin{pmatrix}
1 & -1 \\
-1 & 1
\end{pmatrix}, \tilde\mu:=\frac{1}{\left(\frac{\mu_1}{{\sigma_1}}+\frac{\mu_2}{{\sigma_2}}\right)}
\begin{pmatrix}
\frac{\mu_1}{{\sigma_1}}  \\
\frac{\mu_2}{{\sigma_2}}
\end{pmatrix},
\tilde R := \begin{pmatrix}
1 & r_{2}{\frac{\sigma_2}{\sigma_1}} \\
r_{1}{\frac{\sigma_1}{\sigma_2}} & 1
\end{pmatrix}
.
$$
Furthermore, $\tilde Z$ satisfies \eqref{drift} to \eqref{simple}.
\end{prop}

\begin{proof}
    This is a direct consequence of Definition~\ref{defi} applying the corresponding transformation to \eqref{semimart}.
\end{proof}

\begin{theorem}[\nr{Martin h}\nb{H}armonic functions and Martin boundary: general case]
Suppose \nb{that} (\ref{drift}) and (\ref{vectors}). Let $\alpha^*, \alpha^{**}$ and $(\tilde h_\alpha)_{\alpha\in[\alpha^*, \alpha^{**}]}$ be the angles and the\nr{Martin} harmonic functions in Theorem \ref{thm:main} for the degenerate reflected Brownian motion of parameters $\left(\begin{pmatrix}
1 & -1 \\
-1 & 1
\end{pmatrix}, \tilde\mu,
\tilde R
\right)
$ (\nr{with}\nb{using the} notation of Proposition~\ref{dilatation}). Then, the family of \nr{Martin}\nb{minimal} harmonic functions for the initial process is given by 
\begin{equation}
(x_0, y_0)\longmapsto \tilde h_\alpha\left(\left(\frac{\mu_1}{{\sigma_1}}+\frac{\mu_2}{{\sigma_2}}\right)\left(\frac{x_0}{{\sigma_1}},\frac{y_0}{{\sigma_2}}\right)\right), \quad \quad \alpha\in[\alpha^*, \alpha^{**}].
\end{equation}
Furthermore, the Martin boundary remains homeomorphic to $[\alpha^*, \alpha^{**}]$ and is minimal.
\end{theorem}

\begin{proof}
Let $\psi : \R^2 \longmapsto \R^2$ be the map defined by $\psi(z) = \left(\frac{\mu_1}{{\sigma_1}}+\frac{\mu_2}{{\sigma_2}}\right)\begin{pmatrix}
\frac{1}{{\sigma_1}} & 0 \\
0 & \frac{1}{{\sigma_2}} 
\end{pmatrix}
 z$.
We denote by $G(z_0, \cdot)$ (resp $\tilde G(\tilde z_0, \cdot)$) the Green's measure associated with  $(Z_t)_{t\geq0}$ (resp. with $(\tilde Z_t)_{t\geq0}$) and $g^{z_0}(z)$ (resp. $\tilde g^{\tilde z_0}(\tilde z))$ the corresponding Green's functions, where $\tilde z_0 = \psi(z_0)$. Note that 
\begin{align*}
G(z_0, A)&= \int_0^{+\infty}\P_{z_0}(Z_t\in A)dt \\
&= \int_0^{+\infty}\P_{z_0}\left(Z\left(\frac{u}{\left(\frac{\mu_1}{{\sigma_1}}+\frac{\mu_2}{{\sigma_2}}\right)^2}\right)\in A\right)\frac{du}{\left(\frac{\mu_1}{{\sigma_1}}+\frac{\mu_2}{{\sigma_2}}\right)^2}\\
&=\int_0^{+\infty}\P_{\tilde z_0}\left(\tilde Z_u\in \psi(A)\right)\frac{du}{\left(\frac{\mu_1}{{\sigma_1}}+\frac{\mu_2}{{\sigma_2}}\right)^2}.
\end{align*}
Furthermore, $$\P_{z_0}\left(\tilde Z_u\in \psi(A)\right) = \int_{\psi(A)}\P_{\tilde z_0}\left(\tilde Z_u = u \right)du = \int_{A}\P_{\tilde z_0}\left(\tilde Z_u = \psi(v) \right)|Jac(\psi)|dv.$$
Therefore, the following holds for all $z_0, a \in \R_+^2$:
\begin{equation}
g^{z_0}(a) = \frac{1}{{\sigma_1\sigma_2}}\tilde g^{\tilde z_0}(\psi(a)).
\end{equation}
Then, \begin{align}\label{lien_green}
g^{z_0}(r\cos(\alpha), r\sin(\alpha))
&=\frac{1}{{\sigma_1\sigma_2}}\tilde g^{\psi(z_0)}\left(\tilde r\cos(\tilde\alpha), \tilde r\sin(\tilde\alpha)\right)
\end{align} 
with $\tilde\alpha = \arctan\left({\frac{\sigma_2}{\sigma_1}} \tan(\alpha)\right)$.
The conclusion follows from relation \eqref{lien_green}.
%the asymptotics of the Martin kernel give the same Martin harmonic function as the reflected Brownian motion of parameters 
%$\left(\begin{pmatrix}
%1 & -1 \\
%-1 & 1
%\end{pmatrix}, \tilde\mu,
%\tilde R
%\right).$
\end{proof}
\color{black}
\newpage
\section*{Acknowledgments}
I would like to extend my gratitude to Sandro Franceschi and Irina Kourkova for their invaluable insights and discussions about this article. First and foremost, I sincerely thank Sandro Franceschi for introducing me to the compensation method and for the many fruitful discussions we have had. His guidance has been crucial in shaping this research, and our conversations have greatly deepened my understanding of the subject. I am also deeply grateful to Irina Kourkova for her numerous valuable suggestions regarding the writing of this article, as well as for our exchanges on a specific case related to asymptotic analysis.

\maketitle

\end{document}